\documentclass[11pt, letterpaper,english]{amsart}
\usepackage[left=1in,right=1in,bottom=1in,top=1in]{geometry}
\usepackage{amsmath,amsthm,amssymb}
\usepackage{mathrsfs}
\usepackage[all,cmtip]{xy}
\usepackage{cite}
\usepackage{hyperref}
\usepackage{enumerate}
\usepackage{graphicx}
\usepackage[usenames,dvipsnames]{color}
\usepackage{colonequals}   %
\usepackage{mathtools}
\usepackage{tikz-cd}

\theoremstyle{plain}
\newtheorem{lem}{Lemma}
\newtheorem{lemma}[lem]{Lemma}
\newtheorem{thm}[lem]{Theorem}
\newtheorem{theorem}[lem]{Theorem}
\newtheorem{prop}[lem]{Proposition}
\newtheorem{proposition}[lem]{Proposition}
\newtheorem{cor}[lem]{Corollary}
\newtheorem{corollary}[lem]{Corollary}

\newtheorem{fact}[lem]{Fact}

\newtheorem{question}[lem]{Question}

\theoremstyle{definition}
\newtheorem{defn}[lem]{Definition}
\newtheorem{example}[lem]{Example}
\newtheorem{remark}[lem]{Remark}

\numberwithin{equation}{section}
\numberwithin{lem}{section}

\newcommand{\mathfont}{\mathbf}

\newcommand{\Z}{\mathfont Z}
\newcommand{\ZZ}{\mathfont Z}
\newcommand{\Zhat}{\widehat{\ZZ}}

\newcommand{\QQ}{\mathfont Q}

\newcommand{\FF}{\mathfont F}
\newcommand{\Fbar}{\overline{\FF}}
\newcommand{\Fpbar}{\Fbar_p}

\newcommand{\fC}{\mathfrak{C}}

\newcommand{\fg}{\mathfrak{g}}
\newcommand{\fl}{\mathfrak{l}}
\newcommand{\fm}{\mathfrak{m}}

\newcommand{\ft}{\mathfrak{t}}

\newcommand{\bF}{\mathfont{F}}
\newcommand{\bG}{\mathfont{G}}

\newcommand{\bR}{\mathfont{R}}

\newcommand{\bZ}{\mathfont{Z}}

\newcommand{\cC}{\mathcal{C}}

\newcommand{\cN}{\mathcal{N}}
\newcommand{\cO}{\mathcal{O}}

\DeclareFontFamily{OT1}{rsfs}{}
\DeclareFontShape{OT1}{rsfs}{n}{it}{<-> rsfs10}{}
\DeclareMathAlphabet{\mathscr}{OT1}{rsfs}{n}{it}

\newcommand{\sD}{{\mathscr D}}

\newcommand{\sG}{{\mathscr G}}
\newcommand{\sH}{{\mathscr H}}

\newcommand{\sO}{{\mathscr O}}

\newcommand{\vp}{\varphi}

\newcommand{\into}{\hookrightarrow}

\DeclareMathOperator{\Hom}{Hom}
\DeclareMathOperator{\Ker}{ker}

\DeclareMathOperator{\coker} {coker}
\newcommand \tensor[1] {\otimes_{#1}}
\DeclareMathOperator{\Aut}{Aut}

\DeclareMathOperator{\rank}{rank}

\newcommand{\kbar}{\bar{k}}
\DeclareMathOperator{\Ext}{Ext}

\DeclareMathOperator{\Sym}{Sym}

\DeclareMathOperator{\chara} {char}

\newcommand{\ttm}[4]{\begin{pmatrix}
#1 & #2 \\
#3 & #4
\end{pmatrix}}

\newcommand{\Zp}{\mathfont{Z}_p}
\newcommand{\Zpbar}{\overline{\mathfont{Z}}_p}

\newcommand{\rhobar}{{\overline{\rho}}}
\newcommand{\taubar}{{\overline{\tau}}}
\newcommand{\omegabar}{{\overline{\omega}}}

\newcommand{\gma}[1]{\Gamma_{#1}}

\DeclareMathOperator{\Spec}{Spec}

\DeclareMathOperator{\Spf}{Spf}

\DeclareMathOperator{\Lie}{Lie}

\DeclareMathOperator{\GL}{GL}
\DeclareMathOperator{\SL}{SL}

\DeclareMathOperator{\PGL}{PGL}

\DeclareMathOperator{\SO}{SO}

\DeclareMathOperator{\Sp}{Sp}

\DeclareMathOperator{\Out}{Out}
\newcommand{\ab}{\textrm{ab}}

\newcommand{\Gm}{\mathfont{G}_m}
\DeclareMathOperator{\ad}{ad}
\DeclareMathOperator{\Ad}{Ad}

\newcommand{\cent}[2]{C_{#1}(#2)}
\newcommand{\nlz}[2]{N_{#1}(#2)}

\newcommand{\Ld}{\Lambda}

\newcommand{\Lbar}{\overline{\Lambda}}
\newcommand{\Dt}{\Delta}

\DeclareMathOperator{\CLN}{\widehat{\cC}}
\DeclareMathOperator{\mr}{m.r.}

\newcommand{\g}{\mathfrak{g}}
\newcommand{\h}{\mathfrak{h}}

\renewcommand{\mr}{\operatorname{m.r.}}
\newcommand{\spn}{\operatorname{span}}

\title{Lifting $G$-Valued Galois Representations when $\ell \neq p$}

\author{Jeremy Booher}
\address{School of Mathematics and Statistics, 
         University of Canterbury, Private Bag 4800,
         Christchurch 8140, New Zealand}
\email{jeremy.booher@canterbury.ac.nz, jeremybooher@ufl.edu}

\author{Sean Cotner}
\address{Department of Mathematics,
Stanford University, 450 Jane Stanford Way, Building 380,
Stanford, CA 94305, United States}
\email{scotner@stanford.edu}

\author{Shiang Tang}
\address{Department of Mathematics,
Purdue University, 150 N University St,
West Lafayette, IN 47907, United States}
\email{tang573@purdue.edu}

\begin{document} 

\begin{abstract}
In this paper we study the universal lifting spaces of local Galois representations valued in arbitrary reductive group schemes when $\ell \neq p$. In particular, under certain technical conditions applicable to any root datum we construct a canonical smooth component in such spaces, generalizing the minimally ramified deformation condition previously studied for classical groups. Our methods involve extending the notion of isotypic decomposition for a $\GL_n$-valued representation to general reductive group schemes. To deal with certain scheme-theoretic issues coming from this notion, we are led to a detailed study of certain families of disconnected reductive groups, which we call \textit{weakly reductive} group schemes. Our work can be used to produce geometric lifts for global Galois representations, and we illustrate this for $\mathrm{G}_2$-valued representations.

\end{abstract}

\maketitle

\section{Introduction}

\subsection{Galois Deformations}  \label{ss:galoisdef}

Fix a local field $F$ with residue characteristic $\ell$.  Fix a prime $p \neq \ell$ and a reductive group scheme $G$ (with connected fibers) over the ring of integers $\cO$ in a $p$-adic field.  Let $k$ be the residue field of $\cO$ and $\Gamma_F$ the absolute Galois group of $F$.

Given a $G$-valued representation of $\Gamma_F$ over $k$, i.e. a continuous homomorphism $\rhobar: \Gamma_F \to G(k)$,
Tillouine\cite{tilouine96} introduced a Galois lifting ring $R_{\rhobar}^\square$ (building on work of Mazur treating the $\GL_n$ case \cite{mazur87,mazur95}).  These lifting rings and the associated formal schemes are central in many modern developments in number theory, and it is important to:
\begin{itemize}
    \item understand whether $\rhobar$ lifts to characteristic zero; and
    
    \item understand the geometry of $\Spf(R_{\rhobar}^{\square})$, especially to find formally smooth components and understand how they intersect.
\end{itemize}
(There are similar questions in the complementary situation when $\ell = p$, but the analysis is quite different and is connected with $p$-adic Hodge theory.)  Work on these questions originally focused on the case that $G = \GL_n$ (especially $\GL_2$), but recent developments in the Langlands program have made it increasingly important to understand general reductive groups $G$.  
For example, progress on these questions can be used to:
\begin{itemize}
    \item produce ``nice'' lifts of global mod-$p$ representations to characteristic zero \cite{ram02,hr08,patrikis16,booher19jnt,tang19,fkp21,fkp22};
    \item  investigate $\ell \neq p$ versions of the Breuil-M\'{e}zard conjecture \cite{shotton18,shotton20,shotton23};
    \item  establish potential automorphy theorems and automorphy lifting theorems over global function fields for general $G$ \cite{BHKT,boeckle-cyclic}. 
    \end{itemize}

\begin{theorem} \label{thm:intromain} (see Theorems \ref{thm:liftingfixed} and \ref{thm:main theorem})
Let $\rhobar : \Gamma_F \to G(k)$ be a continuous homomorphism.  
Suppose $p$ is large enough for the root datum of $G$ (see Remark \ref{rmk:bound for G}). 
Then there exists a continuous homomorphism $\rho: \Gamma_F \to G(\cO)$ lifting $\rhobar$ such that $\cent{G}{\rho(I_F)}$ is $\cO$-smooth. 
Moreover, there is a canonical $\cO$-formally smooth irreducible component $\Spf R_{\rhobar}^{\mr, \square}$ of $\Spf R_{\rhobar}^{\square}$. %
\end{theorem}

We call %
$R_{\rhobar}^{\mr,\square}$ the \emph{minimally ramified lifting ring} and the associated Galois representations \emph{minimally ramified}. %
Previous work about the smoothness of (components of) $R_{\rhobar}^{\square}$ has been for classical groups \cite{cht08,Booher19} or about the generic fiber \cite{bp19,bg19} or has had strong dependencies between $\ell$ and $p$ \cite[\S 5]{DHKM}.

The existence of lifts of $\rhobar$ without the smooth centralizer condition also follows from $\bZ_p$-flatness of $R_{\rhobar}^{\square}$, a consequence of results about moduli of Langlands parameters in \cite{DHKM}, \cite{Fargues-Scholze}, and \cite{Zhu} which are established using completely different techniques.  In \cite{DHKM}, one ingredient in the proof of $\bZ_p$-flatness is to find a finite extension $\cO'$ of $\cO$ such that $\rhobar$ lifts to $\rho: \Gamma_F \to G(\cO')$. The lifts produced in \cite{DHKM} are different than those produced using our method: they have finite image and we do not expect the centralizers of the inertia to be $\cO$-smooth.

Our initial motivation for this project was to produce characteristic zero lifts of global $\rhobar : \Gamma_K \to G(k)$ that are geometric in the sense of the Fontaine-Mazur conjecture when $K$ is a number field, using variations on a local-to-global lifting result going back to Ramakrishna \cite{ram02}.  As a sample application of Theorem~\ref{thm:intromain}, we give a lifting result for the exceptional group of type $G_2$, making use of \cite[Theorem A]{fkp21} (a generalization of the local-to-global lifting result) and \cite[Theorem C]{lin_g2} (giving local lifts at $p$).  For a representation $\rhobar : \Gamma \to G(k)$, let $\rhobar(\fg)$ denote the Lie algebra of $G_k$ with $\Gamma$ acting via the composition of $\rhobar$ and the adjoint action.

\begin{corollary} \label{corollary:lifting}
Let $G$ be the exceptional split group $G_2$ over $\ZZ$. Let $p$ be a sufficiently large prime and let $\rhobar \colon \Gamma_{\QQ} \to G(\Fpbar)$ be a continuous representation. Assume that 
\begin{itemize}
    \item $\rhobar$ is odd, i.e. $\dim \mathrm{H}^0(\gma{\bR}, \rhobar(\fg))=\dim \mathrm{Flag}_{G}$.
    \item $\rhobar|_{\gma{\QQ(\zeta_p)}}$ is absolutely irreducible, i.e. its image in $G(\kbar)$ is not contained in any proper parabolic subgroup of $G$. 
\end{itemize}
Then $\rhobar$ lifts to a $\rho \colon \gma{\QQ} \to G(\Zpbar)$ which is geometric in the sense of the Fontaine--Mazur conjecture.
\end{corollary}

The proof is given in Appendix~\ref{appendix-corollary}.  
All that is needed from Theorem~\ref{thm:intromain} is the existence of local lifts at places away from $p$, which follows as above from \cite{DHKM,Fargues-Scholze,Zhu}.

\begin{remark}
We emphasize that $\cO$-smoothness of the centralizer of the inertia in Theorem \ref{thm:intromain} is crucial for establishing the existence of a formally smooth component of the universal lifting ring; see Theorem \ref{thm:liftingcondition}.
Besides producing minimally ramified lifts, our method can also produce lifts with other inertial types and sometimes establish smoothness of the component of the deformation ring containing the lift.  We do not systematically explore this as this is not the focus of this work, but we do build flexibility into our results in Section~\ref{sec:liftingsmr} and give a simple illustration in Example~\ref{ex:inertialtype}.  This is of interest when studying generalizations of the $\ell \neq p$ version of the Breuil-M\'{e}zard conjecture and the irreducible components of the moduli of Langlands parameters as in \cite{shotton18,shotton20,shotton23}.  
\end{remark}

\begin{remark}\label{remark:intro-moduli}
The Galois lifting space for $\rhobar$ is a formal completion of the moduli spaces of \cite{DHKM}, \cite[\S VIII]{Fargues-Scholze}, and \cite{Zhu} at $\rhobar$.  To make this precise, one can apply \cite[Lemma 2.4.10]{Zhu} to the various moduli spaces of Galois representations and Weil-Deligne representations studied and compared in \cite[\S 3.1]{Zhu}.  Then the main results in these papers show that the Galois lifting spaces are flat local complete intersections, and \cite[\S 5]{DHKM} provides some results on generic smoothness.

In this language, Theorem~\ref{thm:intromain} shows in particular that the underlying reduced subscheme of the mod-$p$ fiber of the moduli space $\sH_W$ of Weil-Deligne representations is smooth away from the intersections of components for large $p$, and even at a point of intersection there is some smooth component passing through this point. Indeed, Theorem~\ref{thm:intromain} asserts that every point lies in a formally smooth irreducible component of $\Spf R_{\rhobar}^{\square}$, whose mod-$p$ fiber is therefore a smooth localization of the underlying reduced subscheme of an irreducible component of the mod-$p$ fiber of $\sH_W$. As our bound on $p$ is independent of the size of the residue field of $F$, it applies in situations where the mod-$p$ fibers of the above moduli space are not reduced (see \cite[Proposition 5.26]{DHKM} and note that the banality of $p$ depends on the size of the residue field of $F$ by \cite[Lemma 5.28]{DHKM}). See Remark~\ref{remark:dhkm-interpretation} for a slight reinterpretation of our main result in terms of this moduli space.
\end{remark}

\begin{remark} \label{remark:automorphic} 
In \cite[Section 4.3]{boeckle-cyclic}, an automorphy lifting theorem for $G$-valued Galois representations over global function fields is established, assuming that the mod-$p$ residual automorphic Galois representation has suitably large image and that the local deformation problems are \emph{balanced} in the sense of \cite[Definition 3.4]{boeckle-cyclic}.
Section 5 of \emph{loc.}\ \emph{cit.}\ shows that the \emph{unrestricted} local condition at a place of ramification is balanced (and hence formally smooth) if $p$ is larger than an ineffective constant depending on the automorphic representation using global arguments. In contrast, Theorem \ref{thm:intromain} produces a natural balanced local deformation condition in the general case with an effective lower bound on $p$ depending only on the root datum of $G$ (see Remark \ref{rmk:bound for G} for this lower bound). 
\end{remark}

\begin{remark}
    The restrictions on $p$ in Theorem~\ref{thm:intromain} are effective but not optimal: see Remark~\ref{rmk:bound for G}.  We expect a similar result should hold as long as $p$ is a pretty good prime for $G$ and $p > 3$.
\end{remark}

\subsection{Weakly Reductive Group Schemes} \label{ss:weakly-reductive}

To prove Theorem~\ref{thm:intromain}, we will directly adapt the argument of \cite{cht08} (which dealt with $G = \GL_n$) to a general $G$. For clarity, we will outline (a reinterpretation of) the argument from \cite{cht08} which constructs a canonical lift of $\rhobar: \Gamma_F \to \GL(V \otimes_{\cO} k)$ up to conjugacy, where $V$ is a finite free $\cO$-module of rank $n$. Let $\Lambda_F$ be the maximal prime-to-$p$ closed subgroup of the inertia subgroup $I_F$ of $\Gamma_F$.
\begin{enumerate}
    \item Lift $\rhobar|_{\Lambda_F}$ to a representation $\rho_0: \Lambda_F \to \GL(V)$. Let $V = \bigoplus_i V_i \otimes_{\cO} W_i$ be the isotypic decomposition of $V$.
    \item Show that if $p > n$, then there is a unique extension of the $\Lambda_F$-representation $\bigoplus_i V_i$ to a representation $\tau = (\tau_i): I_F \to \prod_i \GL(V_i)$ such that $\det(\tau_i(\sigma))$ is of finite prime-to-$p$ order for all $\sigma \in I_F$ and all $i$.
    \item Choosing an identification $I_F/\Lambda_F \cong \bZ_p$ and a splitting $I_F/\Lambda_F \to I_F$, show that there is a unipotent element $u_0 \in \prod_i \GL(W_i \otimes_{\cO} k)$ such that $\rhobar(n) = u_0^n$ for all $n \in \bZ_p$. Show that $u_0$ lifts uniquely up to conjugacy to a section $u \in \prod_i \GL(W_i)$ with the same Jordan block decomposition on both fibers, and define $\rho_1: I_F = \Lambda_F \rtimes \bZ_p \to \GL(V)$ by
    \[
    \rho_1(\lambda n) = \tau(\lambda n) u^n.
    \]
    \item Using the uniqueness assertions of (2) and (3), finally extend $\rho_1$ to $\rho: \Gamma_F \to \GL(V)$.
\end{enumerate}
To adapt this argument for general $G$ in place of $\GL_n$, one first needs to interpret the objects appearing. For instance, in steps 2 and 3, we need analogues of $\prod_i \GL(V_i)$ and $\prod_i \GL(W_i)$. The main observation is that when $G = \GL(V)$, the centralizer $\fC = C_G(\Lambda_F)$ is equal to $\prod_i \GL(W_i)$, while the double centralizer $\Delta = C_G(C_G(\Lambda_F))$ is equal to $\prod_i \GL(V_i)$.   
From this perspective, it natural to study $\fC = C_G(\Lambda_F)$ and $\Delta = C_G(C_G(\Lambda_F))$ for general $G$, which we formalize using the notion of a \emph{decomposition type}.  This provides a structure for extending a prime-to-$p$ inertial type to a $G$-valued representation of $\Gamma_F$.
 It is also necessary to understand the abelianization morphism $\Delta \to \Delta^{\rm{ab}}$ (the analog of the determinant) and the center $Z(\Delta)$ (the analogue of the group of scalar matrices).  We must, therefore, understand representability and smoothness properties of the $\cO$-group schemes $\fC$, $\Delta$, $\Delta^{\rm{ab}}$, and $Z(\Delta)$.

\begin{remark}
In step 3, it is also necessary to find a suitable meaning of ``the same Jordan block decomposition on both fibers" for a unipotent section of a general $G$, and to show a suitable conjugacy result for these. This has been handled in \cite{cotner}; see also \cite{Hardesty} for similar results.
\end{remark}

 Unlike the case $G = \GL_n$, it is not evident that $\Delta$ is smooth (or even representable), and it is usually \textit{not} true that $\fC$ and $\Delta$ are reductive group schemes; they often have disconnected fibers. This causes serious difficulties when working integrally, and the theory developed in \cite{sga3} is not sufficient to handle this situation. These difficulties are well-known to experts, and we describe some pathologies in Examples~\ref{example:non-smooth-aut}, \ref{example:p-is-bad}, and \ref{example:p-is-still-bad}.
 
 In order to handle the families of disconnected reductive groups that are relevant to us, we introduce the notion of \textit{weak reductivity}.  Recall \cite[Proposition 3.1.3]{conrad} that if $S$ is a scheme and $G$ is a smooth affine $S$-group scheme with (possibly disconnected) reductive fibers, then the relative identity component $G^0$ is a reductive group scheme, and $G/G^0$ is an \'{e}tale separated $S$-group scheme of finite presentation.

\begin{defn}
Over a scheme $S$, a weakly reductive group scheme is a smooth affine $S$-group scheme $G$ such that $G^0$ is reductive and the component group $G/G^0$ is finite \'{e}tale over $S$ with order invertible on $S$.
\end{defn}

We emphasize the condition that $G/G^0$ is of order invertible on $S$; if this assumption is omitted, such a smooth affine $G$ is often called \textit{geometrically reductive} or \textit{generalized reductive}. However, as we show in Examples~\ref{example:p-is-bad}, \ref{example:p-is-still-bad}, and \ref{example:non-smooth-aut}, general geometrically reductive smooth affine group schemes are more pathological than their reductive counterparts.

\begin{theorem}\label{theorem:intro-wr} (Corollary~\ref{corollary:centralizer-of-weakly-reductive}, Corollary~\ref{cor:double-centralizer}, Proposition~\ref{prop:center}, Proposition~\ref{prop:abelianization})
    Let $S$ be a scheme and let $G$ be a weakly reductive $S$-group scheme.
\begin{enumerate}
    \item If $\Lambda$ is a finite \'etale group scheme acting on $G$ whose order is invertible on $S$, then the fixed point scheme $\fC = C_G(\Lambda)$ is weakly reductive. If $\chara k(s)$ is pretty good\footnote{See Definition~\ref{def:prettygoodprimes}.} for $G_s$ for all $s \in S$, then the centralizer $\Delta = C_G(\fC)$ is weakly reductive.
    \item If $H$ is a simple reductive group scheme acting on $G$ and $(\dim G/\rank H)!$ is invertible on $S$, then the fixed point scheme $C_G(H)$ is weakly reductive.
    \item The center $Z(G)$ is a group scheme of multiplicative type, and it is smooth if $Z(G^0)$ is smooth.
    \item The derived group $\sD(G)$ (in the sense of fppf group sheaves) is represented by a weakly reductive $S$-group scheme, and the abelianization $G^{\mathrm{ab}} = G/\sD(G)$ of $G$ is a smooth group scheme of multiplicative type.
\end{enumerate}
\end{theorem}

The main new input in the proof of Theorem~\ref{theorem:intro-wr} is an analysis of schemes of homomorphisms between weakly reductive group schemes. If $S$ is a scheme and $G$ and $H$ are $S$-group schemes, we let $\underline{\Hom}_{S\textrm{-}\rm{gp}}(H, G)$ denote the functor which sends an $S$-scheme $S'$ to the set of $S'$-homomorphisms $H_{S'} \to G_{S'}$. In \cite[Exp.\ XXIV, Corollaire 7.2.3]{sga3}, it is proved that if $H$ is a reductive $S$-group scheme and $G$ is a smooth affine $S$-group scheme, then $\underline{\Hom}_{S\textrm{-}\rm{gp}}(H, G)$ is representable by a separated $S$-scheme locally of finite presentation. Usually $\underline{\Hom}_{S\textrm{-}\rm{gp}}(H, G)$ is usually not quasi-compact or flat over $S$ (see Example~\ref{example:sga3-hom-scheme}). However, we will show in Theorem~\ref{theorem:hom-scheme} that $\underline{\Hom}_{S\textrm{-}\rm{gp}}(H, G)$ is always a disjoint union of finitely presented $S$-\textit{affine} $S$-schemes. Proving this involves revisiting the proof of representability of $\underline{\Hom}_{S\textrm{-}\rm{gp}}(H, G)$ in \cite{sga3}, using ind-quasi-affine descent and affineness results for schemes of tori in \cite{Faisceaux-amples}. Over a field, this affineness result was proved in \cite[Theorem 6.3]{Brion}, and in general it strengthens affineness results from \cite[Theorem 3.1.4]{Romagny}. 

Let us give a sense of some key steps in the proof of Theorem~\ref{theorem:intro-wr}, starting with part (2).
\begin{enumerate}
    \item Use known cohomology vanishing results (Theorem~\ref{theorem:jantzen-semisimple}) to see that $C_G(H)$ is smooth affine, and use classical arguments over a field to show that $C_G(H)/C_G(H)^0$ is of order invertible on $S$.
    \item Reduce to the case $S = \Spec A$ for a DVR $A$, and use Matsushima's theorem (Theorem~\ref{theorem:alper}) to reduce to showing that the quotient $G/C_G(H)$ is affine.
    \item Show that the natural monomorphism $i: G/C_G(H) \to \underline{\Hom}_{S\textrm{-}\rm{gp}}(H, G)$, given as the orbit map through the inclusion $H \to G$, is a closed embedding, and conclude using the above geometric property of $\underline{\Hom}_{S\textrm{-}\rm{gp}}(H, G)$.
\end{enumerate}
The proof of Theorem~\ref{theorem:intro-wr}(1) is similar, but when $H = C_G(\Lambda)$ we cannot show that $i$ is a closed embedding, so our argument is slightly longer.

To prove Theorem~\ref{theorem:intro-wr}(3), we first show that the automorphism functor $\underline{\Aut}_{G/S}$ is representable by a smooth clopen subscheme $\Aut_{G/S}$ of $\underline{\Hom}_{S\textrm{-}\rm{gp}}(G, G)$. Since $Z(G)$ is the kernel of the natural $S$-homomorphism $\vp: G \to \Aut_{G/S}$, this reduces us to understanding $\Aut_{G/S}$ and $\vp$. We prove Theorem~\ref{theorem:intro-wr}(4) through a somewhat complicated reduction to the separate cases of reductive group schemes and finite etale group schemes, both of which are understood.

\begin{remark}
Appendix~\ref{section:appendix-b} provides a curious consequence of Theorem~\ref{theorem:intro-wr} to the sizes of component groups of centralizers over fields.  It illustrates the power of working with group schemes over rings, even when interested in questions over fields.
\end{remark}

\subsection{Outline of the Paper}
Sections~\ref{sec:weaklyreductive} and \ref{sec:centralizers} develop the theory of weakly reductive group schemes and establish a variety of results about centralizers.  Section~\ref{sec:decomp types} introduces the notion of a decomposition type, and Sections~\ref{sec:clifford} and \ref{sec:liftingsmr} use this and the results about weakly reductive group schemes to construct lifts and the minimally ramified deformation condition.  Depending on the reader's interests, Sections~\ref{sec:decomp types}-\ref{sec:liftingsmr} can be read first relying  on the properties of weakly reductive group schemes summarized in Theorem~\ref{theorem:intro-wr}.

\subsection{Notation and Terminology}
Given a group scheme $H$ defined over a ring $R$ and an $R$-algebra $A$, we write $H_A$ for the base change of $H$ to $A$, and write $H(A)$ for the $A$-points of $H$.

If $S$ is a scheme and $H$ is an $S$-group scheme acting on another $S$-group scheme $G$, then we denote by $C_G(H)$ the functor of fixed points for the action of $H$ on $G$. If $H$ is an $S$-subgroup scheme of $G$ then we denote by $N_G(H)$ the functor of sections of $G$ normalizing $H$. Note that if $H$ is an $S$-subgroup scheme of $G$ then $C_G(H)$ is the centralizer of $H$ in $G$. For representability results, see \cite[Exp.\ XII, Proposition 9.2]{sga3} and \cite[Proposition 2.1.6]{conrad}; when these functors are representable we will use the same notation to denote their representing objects.

We follow the convention in \cite{sga3}
and require that reductive group schemes have connected fibers.  In Section~\ref{sec:weaklyreductive} we introduce the notion of \textit{weakly reductive} group schemes which allows disconnected fibers under some hypotheses.  
However, when working over a field we do allow reductive groups to be disconnected, following general practice. We will require all groups of multiplicative type to be finitely presented, unlike the definition in \cite[Exp.\ IX]{sga3}.

For a local field $F$, we use $\Gamma_F$ to denote the absolute Galois group of $F$, $I_F$ to denote the inertia subgroup of $\Gamma_F$, and $\Lambda_F \subset I_F$ to be the kernel of a homomorphism $I_F \to \Z_p$ as in Section~\ref{ss:galois}.

We also recall the definition of a good and pretty good primes for a root datum $(X,\Phi,Y,\Phi^\vee)$. 

\begin{defn}\label{def:goodprimes for G} 
A prime $p$ is \emph{good} if for every closed subsystem $\Sigma \subset \Phi$, $\Z\Phi/\Z\Sigma$ is $p$-torsion free. 
\end{defn}

\begin{defn} \label{def:prettygoodprimes}
We say that $p$ is \emph{pretty good} if the groups $X/\Z \Phi'$ and $Y/\Z {\Phi'}^\vee$ have no $p$-torsion for all subsets $\Phi' \subset \Phi$.
\end{defn}

A prime is good (resp. pretty good) for a weakly reductive group $G$ if it is good (resp. pretty good) for the root datum associated to $G^0$.   By convention, we also say that $0$ is good (and pretty good).

\begin{remark} \label{rmk:pretty good}
By \cite[Lemma 2.2]{cotner}, a prime $p$ is pretty good for a connected reductive group $G$ over a field of characteristic $p$ if and only if all of the following conditions hold:
\begin{enumerate}
    \item $p$ is good for $G$,
    \item $p$ does not divide the order of $\pi_1(\sD(G))$,
    \item $Z(G)$ is smooth.
\end{enumerate}
\end{remark}

\subsection{Acknowledgments}  Booher was partially supported by the Marsden Fund Council administered by the Royal Society of New Zealand.  We thank
Patrick Allen,
Brian Conrad,
Pol van Hoften,
Mikko Korhonen,
Daniel Le,
Martin Liebeck, 
Ben Martin,
Gil Moss,
Stefan Patrikis,
Jeroen Schillewaert,
Jay Taylor, and
Felipe Voloch
for helpful conversations.  We thank the anonymous referees for their work.

\section{Weakly reductive group schemes} \label{sec:weaklyreductive}

In this section we study weakly reductive group schemes. Weak reductivity is slightly more stringent than the condition of \emph{geometric reductivity}, as introduced in \cite[Definition 9.1.1]{Alper-adequate}. We will not recall the definition in general, but we quote the following theorem, which will be used several times in the sequel, especially in Section~\ref{sec:centralizers}.

\begin{theorem}\label{theorem:alper}\cite[Theorems 9.4.1, 9.7.6]{Alper-adequate}
Let $S$ be a scheme, and let $H \subset G$ be flat, finitely presented, and separated $S$-group schemes, with $H$ closed in $G$.
\begin{enumerate}
    \item If $G$ is smooth and affine, then it is geometrically reductive if and only if $G^0$ is reductive and $G/G^0$ is finite.
    \item If $G$ is affine and geometrically reductive, then $H$ is geometrically reductive if and only if $G/H$ is affine.
\end{enumerate}
\end{theorem}

In light of Theorem~\ref{theorem:alper}, a weakly reductive group scheme is just a geometrically reductive smooth affine group scheme with tame component group in the sense of \cite{AOV}. We will extend some fundamental constructions for reductive group schemes to weakly reductive group schemes. These extensions do not generally work for more general geometrically reductive smooth affine group schemes (see Examples~\ref{example:non-smooth-aut}, \ref{example:p-is-bad}, and \ref{example:p-is-still-bad}). We note that, by \cite[Theorem 9.9]{Alper-Hall-Rydh}, all smooth affine linearly reductive group schemes are weakly reductive; we omit the definition of linear reductivity (which may be found in \cite[Definition 2.1]{Alper-Hall-Rydh} because it is of a technical nature orthogonal to the goals of this paper. %

We work throughout with arbitrary base schemes, but many proofs begin by reducing to simpler cases. For the most part, we do not spell out these reductions in detail, and we refer the reader to \cite[IV\textsubscript{3}, Sections 8, 9, 11]{EGA} for the techniques involved in such reduction steps.

\subsection{Schemes of homomorphism}\label{section:hom-scheme}

Let $S$ be a scheme. If $G$ and $H$ are $S$-group schemes, then we define the set-valued functor $\underline{\Hom}_{S\textrm{-}\rm{gp}}(H, G)$ on $S$-schemes by
\[
\underline{\Hom}_{S\textrm{-}\rm{gp}}(H, G)(S') \coloneqq \Hom_{S'\textrm{-}\rm{gp}}(H_{S'}, G_{S'}).
\]
The goal of this section is to study this functor. The first aim is the following theorem.

\begin{theorem}\label{theorem:hom-scheme}
Let $S$ be a scheme, let $G$ be a smooth affine $S$-group scheme, and let $H$ be a geometrically reductive smooth affine $S$-group scheme. The functor $\underline{\Hom}_{S\textrm{-}\rm{gp}}(H, G)$ is representable by an ind-quasi-affine $S$-scheme locally of finite presentation. Moreover, suppose $S$ is normal, quasi-compact, and quasi-separated, and $H^0$ admits a maximal $S$-torus. Then $\underline{\Hom}_{S\textrm{-}\rm{gp}}(H, G)$ is representable by a disjoint union of finitely presented $S$-affine $S$-schemes.
\end{theorem}

As a general rule, we like to use underlines to refer to functors, and omit the underline when referring to a representing object. However, in this case the notation without the underline has an independent meaning, so it would be confusing to omit it.

It is not clear a priori that $\underline{\Hom}_{S\textrm{-}\rm{gp}}(H, G)$ is even representable; for this, we begin with the following fundamental result of Demazure.

\begin{lemma}\label{lemma:demazure-rep}
Suppose that $H$ is a reductive $S$-group scheme and that $G$ is smooth and quasi-projective over $S$ with affine fibers. Then $\underline{\Hom}_{S\textrm{-}\rm{gp}}(H, G)$ is representable by a separated $S$-scheme locally of finite presentation.
\end{lemma}

\begin{proof}
This is \cite[Exp.\ XXIV, Corollaire 7.2.3]{sga3}.
\end{proof}

\begin{example}\cite[Exp.\ XXIV, 7.4]{sga3}\label{example:sga3-hom-scheme}
    The scheme $\underline{\Hom}_{\bZ\textrm{-}\rm{gp}}(\SL_{2, \bZ}, \SL_{2, \bZ})$ is the disjoint union of the following $\bZ$-schemes:
    \begin{enumerate}
        \item a scheme isomorphic to $\Spec \bZ$ (corresponding to the trivial homomorphism $\SL_{2, \bZ} \to \SL_{2, \bZ}$),
        \item a scheme isomorphic to $\PGL_{2, \bZ}$ (corresponding to conjugates of the identity homomorphism $\SL_{2, \bZ} \to \SL_{2, \bZ}$),
        \item for each prime number $p$ and each positive integer $n$, a scheme isomorphic to $\PGL_{2, \bF_p}$ (corresponding to conjugates of the $p^n$-Frobenius morphism $\SL_{2, \bF_p} \to \SL_{2, \bF_p}$).
    \end{enumerate}
    In particular, $\underline{\Hom}_{\bZ\textrm{-}\rm{gp}}(\SL_{2, \bZ}, \SL_{2, \bZ})$ is neither flat nor quasi-compact.
\end{example}

To prove Theorem~\ref{theorem:hom-scheme}, we will use \'{e}tale descent to pass further to the case that $H/H^0$ is constant and then realize $\underline{\Hom}_{S\textrm{-}\rm{gp}}(H, G)$ as a closed subscheme of $\underline{\Hom}_{S\textrm{-}\rm{gp}}(H^0, G) \times G^n$ for some $n$. However, in order to make the descent argument one needs to know effectivity of \'{e}tale descent for closed subschemes of $\underline{\Hom}_{S\textrm{-}\rm{gp}}(H^0, G) \times G^n$. Since $\underline{\Hom}_{S\textrm{-}\rm{gp}}(H^0, G)$ is usually not quasi-compact over $S$, this descent argument is not trivial. Thus before showing representability we will show that $\underline{\Hom}_{S\textrm{-}\rm{gp}}(H^0, G)$ is \textit{ind-quasi-affine} over $S$ \cite[\href{https://stacks.math.columbia.edu/tag/0AP6}{Tag 0AP6}]{stacks-project} in order to apply effectivity of fpqc descent for ind-quasi-affine morphisms \cite[\href{https://stacks.math.columbia.edu/tag/0APK}{Tag 0APK}]{stacks-project}. We begin with a more detailed study of $\underline{\Hom}_{S\textrm{-}\rm{gp}}(H, G)$ in the case that $H$ is reductive.

\begin{lemma}\label{lemma:ind-quasi-affine-criterion}
    Let $X$ be a locally noetherian scheme, and let $\pi: Y \to X$ be a finite surjective morphism, where $Y$ is a disjoint union of affine schemes. Then $X$ is ind-quasi-affine. If $\pi$ is open (e.g., flat), then $X$ is a disjoint union of affine schemes.
\end{lemma}

\begin{proof}
    Let $U \subset X$ be a quasi-compact open subscheme; to show that $X$ is ind-quasi-affine, we must show that $U$ is quasi-affine. Note that $\pi^{-1}(U) \subset Y$ is quasi-compact, so it is contained in an affine clopen subscheme $V \subset Y$ by assumption. By Chevalley's theorem that affineness can be checked after passing to a finite cover \cite[II, Th\'eor\`eme 6.7.1]{EGA}, the closed subset $\pi(V) \subset X$ is affine (when considered with its reduced subscheme structure). The schematic closure $\overline{U}$ of $U$ in $X$ is a closed subset of $\pi(V)$, so $\overline{U}_{\mathrm{red}}$ is affine. Thus by Chevalley's theorem again, $\overline{U}$ is affine. Since $U$ is open in $\overline{U}$, it follows that $U$ is quasi-affine.\smallskip

    If $\pi$ is open, then the closed subset $\pi(V) \subset X$ is also open, so it is affine when considered with the structure of an open subscheme of $X$.
\end{proof}

\begin{lemma}\label{lemma:torus-source}
In the setting of Theorem~\ref{theorem:hom-scheme}, suppose that $H$ is a torus. Then $\underline{\Hom}_{S\textrm{-}\rm{gp}}(H, G)$ is representable by a smooth ind-quasi-affine $S$-scheme. If $S$ is normal, quasi-compact, and quasi-separated, then $\underline{\Hom}_{S\textrm{-}\rm{gp}}(H, G)$ is representable by a disjoint union of smooth affine $S$-schemes.
\end{lemma}

\begin{proof}
First, smoothness of $\underline{\Hom}_{S\textrm{-}\rm{gp}}(H, G)$ is proved in \cite[Exp.\ XI, Corollaire 4.2]{sga3}. If $S$ is normal and locally noetherian, then the result follows from \cite[Th\'eor\`eme IX 2.6]{Faisceaux-amples}. For the remainder, we therefore assume that $S$ is quasi-compact and quasi-separated. By spreading out (using \cite[Theorem C.9]{Thomason-Trobaugh}), we may assume that $S$ is of finite type over $\Spec \bZ$. In particular, the normalization $S' \to S$ is finite. Now $\underline{\Hom}_{S'\textrm{-}\rm{gp}}(H_{S'}, G_{S'})$ is a disjoint union of affine $S$-schemes and the morphism $\underline{\Hom}_{S'\textrm{-}\rm{gp}}(H_{S'}, G_{S'}) \to \underline{\Hom}_{S\textrm{-}\rm{gp}}(H, G)$ is finite and surjective, so ind-quasi-affineness of $\underline{\Hom}_{S\textrm{-}\rm{gp}}(H, G)$ follows from Lemma~\ref{lemma:ind-quasi-affine-criterion}.
\end{proof}

\begin{lemma}\label{lemma:hom-scheme-reductive}
In the setting of Theorem~\ref{theorem:hom-scheme}, assume that $H$ is reductive. Then $\underline{\Hom}_{S\textrm{-}\rm{gp}}(H, G)$ is representable by an ind-quasi-affine $S$-scheme. If $S$ is normal, quasi-compact, and quasi-separated, and $H$ admits a maximal $S$-torus, then $\underline{\Hom}_{S\textrm{-}\rm{gp}}(H, G)$ is representable by a disjoint union of finitely presented affine $S$-schemes.
\end{lemma}

\begin{proof}
To show ind-quasi-affineness, we may work locally on $S$ and spread out to assume that $S$ is affine, noetherian, and excellent. (For the second claim, we may use \cite[Theorem C.9]{Thomason-Trobaugh} to make the same reduction.) Using Lemma~\ref{lemma:ind-quasi-affine-criterion}, we may also pass from $S$ to its normalization to assume that $S$ is normal. Passing to a further \'{e}tale cover, we may and do assume that $H$ admits a maximal $S$-torus. By \cite[Exp.\ XXIV, Corollaire 7.1.9]{sga3}, the natural restriction map
\[
\underline{\Hom}_{S\textrm{-}\rm{gp}}(H, G) \to \underline{\Hom}_{S\textrm{-}\rm{gp}}(T, G)
\]
is finitely presented and affine, so the result follows from Lemma~\ref{lemma:torus-source}. (Note that a smooth affine $S$-scheme is automatically finitely presented.)
\end{proof}

\begin{proof}[Proof of Theorem~\ref{theorem:hom-scheme}]
Recall that $H$ is now assumed to be a geometrically reductive smooth affine $S$-group scheme. In particular, $H^0$ is a reductive group scheme and $H/H^0$ is finite by Theorem~\ref{theorem:alper} (and similarly for $G$). There are two issues: first, we need to show that $\underline{\Hom}_{S\textrm{-}\rm{gp}}(H, G)$ is representable, and then we need to show that if $S$ is normal, quasi-compact, and quasi-separated, and $H$ admits a maximal $S$-torus, then $\underline{\Hom}_{S\textrm{-}\rm{gp}}(H, G)$ is a disjoint union of finitely presented $S$-affine $S$-schemes (at which point ind-quasi-affineness in general follows from Lemma~\ref{lemma:ind-quasi-affine-criterion}). For both points, by working locally and spreading out we may assume that $S$ is noetherian and connected. \smallskip

To prove representability, first assume that $H/H^0$ is constant and that the natural map $H(S) \to (H/H^0)(S)$ is surjective. Let $h_1, \dots, h_n \in H(S)$ be a system of representatives for $(H/H^0)(S)$. We may and do assume $h_1 = 1$. There is then a natural morphism of functors
\[
\beta: \underline{\Hom}_{S\textrm{-}\rm{gp}}(H, G) \to \underline{\Hom}_{S\textrm{-}\rm{gp}}(H^0, G) \times G^n,
\]
given by $f \mapsto (f|_{H^0}, f(h_1), \dots, f(h_n))$. We claim that $\beta$ is a closed embedding. To this end, we need to understand when a tuple $(f_0, g_1, \dots, g_n)$ in $\underline{\Hom}_{S\textrm{-}\rm{gp}}(H^0, G)(S') \times G(S')^n$ lies in the image of $\beta$. \smallskip

For indices $i, j$, let $h_i h_j = h_{\delta(i, j)} h_{i,j}$, where $1 \leq \delta(i, j) \leq n$ and $h_{i,j} \in H^0(S)$. In any case, there is a unique morphism of $S$-schemes $f: H \to G$ with $f|_{H^0} = f_0$ and $f(h_i) = g_i$ for all $i$: for this, note that for any $S$-scheme $S'$ and any $h \in H(S')$, there is a unique open decomposition $S' = \bigsqcup_{i=1}^n S'_i$ such that $h|_{S'_i} = h_i h'_i$ for some $h'_i \in H^0(S'_i)$. 
Thus $f$ is defined uniquely by requiring $f(h_i h') = g_i f_0(h')$ for every $S$-scheme $S'$ and every $h' \in H^0(S')$. Now $(f_0, g_1, \dots, g_n)$ lies in the image of $\beta$ if and only if the above-defined $f$ is a homomorphism. \smallskip

Unraveling, we find that $f$ is a homomorphism if and only if the following conditions are satisfied.
\begin{enumerate}
    \item $g_1 = 1$,
    \item $g_i^{-1} f_0(h) g_i = f_0(h_i^{-1} h h_i)$ for all $i$,
    \item $g_i g_j = g_{\delta(i, j)} f_0(h_{i,j})$.
\end{enumerate}
So indeed $\beta$ is a finitely presented closed embedding, whence $\underline{\Hom}_{S\textrm{-}\rm{gp}}(H, G)$ is representable, and in fact it is ind-quasi-affine over $S$ by Lemma~\ref{lemma:hom-scheme-reductive}. Moreover, if $S$ is normal then this shows that $\underline{\Hom}_{S\textrm{-}\rm{gp}}(H, G)$ is a disjoint union of finitely presented $S$-affine $S$-schemes.\smallskip

Now pass to the general case, i.e., no longer assume that $H/H^0$ is constant and that $H(S) \to (H/H^0)(S)$ is surjective. In any case, there is a finite \'{e}tale cover $S' \to S$ such that $H_{S'}/H_{S'}^0$ is constant and $H(S') \to (H/H^0)(S')$ is surjective (e.g., take $S'$ to be a Galois closure of the finite \'etale $H/H^0$), and so $\underline{\Hom}_{S'\textrm{-}\rm{gp}}(H_{S'}, G_{S'})$ is ind-quasi-affine over $S'$ by the above. Thus by effectivity of fpqc descent for ind-quasi-affine morphisms \cite[\href{https://stacks.math.columbia.edu/tag/0APK}{Tag 0APK}]{stacks-project}, we see that $\underline{\Hom}_{S\textrm{-}\rm{gp}}(H, G)$ is representable and locally of finite presentation. Now that we have representability, we may assume that $S$ is normal, quasi-compact, and quasi-separated. As we have already seen, $\underline{\Hom}_{S'\textrm{-}\rm{gp}}(H_{S'}, G_{S'})$ is representable by a disjoint union of finitely presented $S'$-affine $S'$-schemes, so because the morphism $\underline{\Hom}_{S'\textrm{-}\rm{gp}}(H_{S'}, G_{S'}) \to \underline{\Hom}_{S\textrm{-}\rm{gp}}(H, G)$ is finite \'etale, the result follows from Lemma~\ref{lemma:ind-quasi-affine-criterion}.
\end{proof}

\begin{example}\label{example:hom-scheme-pathologies}
The schemes in Theorem~\ref{theorem:hom-scheme} are usually not quasi-compact or flat, and they can fail to have smooth fibers. We saw examples of the first two phenomena in Example~\ref{example:sga3-hom-scheme}. For an example in which smoothness fails, let $k$ be an algebraically closed field of characteristic $p > 0$, and consider the component $C$ of $\underline{\Hom}_{k\textrm{-}\rm{gp}}(\SL_2, \GL_{p+1})$ containing the representation $\Sym^p V$, where $V$ is the standard representation of $\SL_2$. This representation is not semisimple: in the notation of \cite[II, Chapter 2]{Jantzen}, it has Jordan--H\"older factors $L(p)$ and $L(p-2)$. In fact, one can check that $C$ consists of three orbits for the $\GL_{p+1}$-action: the orbit of $\Sym^p V$, the orbit of $L(p) \oplus L(p-2)$, and the orbit of $(\Sym^p V)^*$. The first and last of these orbits are smooth and open, and their closures intersect in the second orbit.

Even worse, components of $\underline{\Hom}_{k\textrm{-}\rm{gp}}(\SL_n, \GL_N)$ can have infinitely many orbits for large $n$ and $N$; one can deduce this using \cite[Theorem 5.2]{Scott-Xi}. Consequently, it can be difficult to predict the dimensions of components of $\underline{\Hom}_{k\textrm{-}\rm{gp}}(H, G)$. 
\end{example}

\begin{question}
\begin{enumerate}
\item  If $k$ is a field of characteristic $p > 0$ and $H$ and $G$ are reductive, can $\underline{\Hom}_{k\textrm{-}\rm{gp}}(H, G)$ be non-reduced? This cannot occur if $p > \dim G/\rank H$, essentially by Theorem~\ref{theorem:jantzen-semisimple}.

\item More generally, if $H$ and $G$ are reductive group schemes over $\bZ$, can $\underline{\Hom}_{\bZ\textrm{-}\rm{gp}}(H, G)$ be non-reduced?

\item If $A$ is a DVR and $H$ and $G$ are reductive, can $\underline{\Hom}_{A\textrm{-}\rm{gp}}(H, G)$ have a non-flat component with an integral point? Again, this cannot happen if $p$ is ``large".
\end{enumerate}
\end{question}

The geometry of $\Hom$-schemes can be related to the theory of complete reducibility as in \cite{BMR05}. Recall that if $k$ is a field, $G$ is a reductive $k$-group, and $H \subset G$ is a closed $k$-subgroup scheme, then $H$ is $G$-completely reducible (or $G$-cr) if, for any R-parabolic $\overline{k}$-subgroup $P \subset G_{\overline{k}}$ such that $H_{\overline{k}} \subset P$, there exists an R-Levi $L \subset P$ such that $H_{\overline{k}} \subset L$. (The notions of R-parabolic and R-Levi subgroup are defined in \cite[Section 6]{BMR05}; they coincide with the usual notions of parabolic and Levi if $G$ is connected.)

By \cite[Proposition 2.16, Theorem 3.1, Section 6]{BMR05}, if $H$ is topologically generated by $x_1, \dots, x_n \in H(k)$ (in the sense that the subgroup of $H(k)$ generated by $x_1, \dots, x_n$ is Zariski-dense in $H$), then $H$ is $G$-cr if and only if the $G$-orbit of the $n$-tuple $(x_1, \dots, x_n)$ is closed in $G^n$. Moreover, \cite[Property 4]{Serre03} shows that if $H$ is smooth and $G$-cr then $H^0$ is reductive. With these two facts in mind, the following lemma relates $G$-complete reducibility to orbits in $\Hom$-schemes.

\begin{lemma}\label{lemma:bmr}
Let $k$ be a field and let $G$ and $H$ be (possibly disconnected) reductive $k$-groups. If $H$ is topologically generated by $x_1, \dots, x_n \in H(k)$, then the $k$-morphism $\iota: \underline{\Hom}_{k\textrm{-}\rm{gp}}(H, G) \to G^n$ sending $f$ to $(f(x_1), \dots, f(x_n))$ is monic and satisfies the valuative criterion of properness. In particular, $\iota$ is a closed embedding when restricted to any connected component of $\underline{\Hom}_{k\textrm{-}\rm{gp}}(H, G)$.
\end{lemma}

\begin{proof}
    It is clear that $\iota$ is monic, and the final claim follows from the others by \cite[IV\textsubscript{3}, Proposition 8.11.5]{EGA} and Theorem~\ref{theorem:hom-scheme}, which shows that every connected component of $\underline{\Hom}_{k\textrm{-}\rm{gp}}(H, G)$ is of finite type over $k$. We now verify that $\iota$ satisfies the valuative criterion of properness. Let $A$ be a $k$-algebra which is a DVR with fraction field $K$, and let $(g_1, \dots, g_n) \in G(A)^n$ be such that there exists a $K$-homomorphism $f_1: H_K \to G_K$ satisfying $f_1(x_i) = g_i$ for all $i$. Let now $\Gamma \subset H \times_{\Spec A} G$ be the schematic closure of the graph of $f_1$, so that $\Gamma$ is a flat closed $A$-subgroup scheme of $H \times_{\Spec A} G$ whose projection map $\pi_1$ to $H$ is an isomorphism on generic fibers over $A$. Moreover, since $(x_i, g_i) \in H(A) \times G(A)$ for all $i$, it follows that $(x_i, g_i) \in \Gamma(A)$ and thus $\pi_{1,s}: \Gamma_s \to H_s$ is surjective.  Since $H_s$ is smooth, we see that $\pi_{1, s}$ is flat, and by fibral flatness it follows that $\pi_1$ is flat. Since $(\ker \pi_1)_K = \{1\}$, it follows from flatness that $\ker \pi_1 = \{1\}$. Thus $\pi_{1, s}$ is a closed embedding, and since $\pi_{1, s}$ is surjective and $H_s$ is smooth, it follows that $\pi_{1, s}$ is an isomorphism. By the fibral isomorphism criterion, $\pi_1$ is therefore an isomorphism and it is the graph of an $A$-homomorphism $f: H \to G$ whose generic fiber is $f_1$. This verifies the valuative criterion.
\end{proof}

\begin{lemma}\label{lemma:g-cr}
    Let $k$ be a field and let $G$ and $H$ be reductive $k$-groups. If $f: H \to G$ is a $k$-homomorphism, then $f(H) \subset G$ is $G$-cr if and only if the $G$-orbit through $f$ in $\underline{\Hom}_{k\textrm{-}\rm{gp}}(H, G)$ is closed.
\end{lemma}

\begin{proof}
    We may and do pass to a (possibly transcendental) field extension of $k$ to assume that there exist $x_1, \dots, x_n \in H(k)$ which topologically generate $H$. If $f(H)$ is $G$-cr, then by \cite[Proposition 2.16, Theorem 3.1]{BMR05} the $G$-orbit of $(f(x_1), \dots, f(x_n))$ in $G^n$ is closed. Thus the $G$-orbit of $f$ is closed in $\underline{\Hom}_{k\textrm{-}\rm{gp}}(H, G)$ since this orbit is simply the preimage of the $G$-orbit of $(f(x_1), \dots, f(x_n))$ under the map $\iota$ of Lemma~\ref{lemma:bmr}. Conversely, if the $G$-orbit of $f$ is closed, then $\iota(G\cdot f)$ is closed in $G^n$ by Lemma~\ref{lemma:bmr}, and we conclude with the fact that $\iota(G \cdot f)$ is simply the $G$-orbit of $(f(x_1), \dots, f(x_n))$.
\end{proof}

\begin{remark}
    Using \cite[Proposition 3.2, Theorem 10.3]{martin03} and Lemma~\ref{lemma:g-cr}, one can show that every component of $\underline{\Hom}_{k\textrm{-}\rm{gp}}(H, G)$ contains only finitely many closed $G$-orbits. Consequently, \cite[Th\'eor\`eme 4.4]{Serre-cr} shows that if $G$ is simple and $p > 1 + \rank G$, then $\underline{\Hom}_{k\textrm{-}\rm{gp}}(H, G)_{\mathrm{red}}$ is a disjoint union of $G$-orbits. For classical groups $G$, $\underline{\Hom}_{k\textrm{-}\rm{gp}}(H, G)$ is reduced under these conditions, but we do not know what happens if $G$ is exceptional.
\end{remark}

\subsection{Automorphism schemes}

Next, if $G$ is an $S$-group scheme we define the set-valued functor $\underline{\Aut}_{G/S}$ by
\[
\underline{\Aut}_{G/S}(S') = \{f \in \Hom_{S'\textrm{-}\rm{gp}}(G_{S'}, G_{S'}): f\text{ is an isomorphism}\}.
\]
The following two lemmas are the crucial inputs needed to analyze this functor.

\begin{lemma}\label{lemma:iso-kernel}
Let $G$ be a finitely presented $S$-group scheme. If $f: G \to G$ is an $S$-homomorphism, then $f$ is an isomorphism if and only if $\ker f_s = \{1\}$ for all $s \in S$.
\end{lemma}

\begin{proof}
If $f$ is an isomorphism, then certainly $\ker f = \{1\}$. If $\ker f_s = \{1\}$ for all $s \in S$, then $\ker f = \{1\}$: indeed, the identity section $e: S \to \ker f$ is a morphism of $S$-schemes which is an isomorphism on fibers over $S$, so because $S$ is $S$-flat it follows from the fibral isomorphism criterion \cite[IV\textsubscript{4}, Corollaire 17.9.5]{EGA} that $e$ is an isomorphism, i.e., $\ker f = \{1\}$. Thus $f$ is monic, and it follows from the Ax--Grothendieck theorem \cite[IV\textsubscript{4}, Proposition 17.9.6]{EGA} that $f$ is an isomorphism.
\end{proof}

\begin{lemma}\label{lemma:dvr-morphism}
Let $A$ be a DVR, and let $G$ and $H$ be geometrically reductive smooth affine $A$-group schemes. If $f: G \to H$ is an $A$-homomorphism, then the following are equivalent.
\begin{enumerate}
    \item $f_s$ is an isomorphism,
    \item $f_\eta$ is an isomorphism,
    \item $f$ is an isomorphism.
\end{enumerate}
\end{lemma}

\begin{proof}
Clearly (3) implies (1) and (2). Conversely, if $f_s$ and $f_\eta$ are both isomorphisms, then $f$ is an isomorphism by the fibral isomorphism criterion \cite[IV\textsubscript{4}, Corollaire 17.9.5]{EGA}. Thus it suffices to show that (1) and (2) are equivalent. We may and do further assume that $A$ is complete with algebraically closed residue field. 

First, suppose that $G$ (and hence also $H$ under either (1) or (2), due to finiteness of $H/H^0$) has connected fibers. Assume $f_s$ is an isomorphism. By the fibral isomorphism criterion, $f_{A/\fm^n}$ is an isomorphism for all $n \geq 1$. By the local flatness criterion \cite[Theorem 22.3]{Matsumura}, $f$ is flat. In fact, $f$ is \'{e}tale near $1$ because the \'{e}tale locus is open, and since it is a homomorphism with $G$ fppf over $A$, it follows that $f$ is \'{e}tale. In particular, $\ker f$ is \'{e}tale. One can check that an \'{e}tale normal subgroup scheme of a connected group scheme over a field is automatically central, so $\ker f_\eta$ is central and thus $\ker f$ is contained in $Z(G)$. But $Z(G)$ is of multiplicative type and $\ker f$ is a flat closed $A$-subgroup scheme of $Z(G)$, so it is also of multiplicative type by \cite[Corollary B.3.3]{conrad}. Since $\ker f_s = \{1\}$, we conclude that $\ker f_\eta = \{1\}$ and hence $f_\eta$ is a closed embedding by \cite[Exp.\ VI\textsubscript{B}, Corollaire 1.4.2]{sga3}. For dimension reasons, $f_\eta$ is dominant, so it is surjective by \cite[Exp.\ VI\textsubscript{B}, Proposition 1.2]{sga3}, and hence it is an isomorphism since $H_\eta$ is smooth.

Conversely, suppose that $f_\eta$ is an isomorphism. If $g \in G(k(s))$ is a nontrivial semisimple element, then there is a maximal $k(s)$-torus $T_0 \subset G_s$ such that $g \in T_0(k(s))$. By \cite[Exp.\ IX, Th\'eor\`eme 3.6]{sga3}, since $A$ is complete we may find a maximal $A$-torus $T \subset G$ with special fiber $T_0$. Since $\ker f|_{T_\eta} = \{1\}$, it follows from \cite[Exp.\ IX, Th\'eor\`eme 6.8]{sga3} that $\ker f|_T = \{1\}$, so in particular $g \not\in (\ker f)(k(s))$. 
Note that $(\ker f_s)^0_{\rm{red}}$ is a smooth connected closed normal subgroup of $G_s$, so it is \textit{reductive}, and thus its semisimple locus is dense. But the above argument show that its semisimple locus is $\{1\}$, so in fact $(\ker f_s)^0_{\rm{red}} = \{1\}$, whence $\ker f_s$ is \textit{finite}. Thus for dimension reasons, $f_s$ is dominant, and since $G$ is smooth we find that $f_s$ is flat: this follows from generic flatness and a simple translation argument. By the fibral flatness criterion \cite[IV\textsubscript{3}, Th\'eor\`eme 11.3.10]{EGA}, $f$ is flat, and so $\ker f$ is flat. Since $\ker f_\eta = \{1\}$, it follows that $\ker f = \{1\}$: as a flat closed subscheme of $G$, $\ker f$ is the closure of its generic fiber. So $f_s$ is a surjective closed embedding, hence an isomorphism.

Now consider the general case, namely that $G/G^0$ is finite. Suppose that either $f_s$ or $f_\eta$ is an isomorphism. By the reductive case settled above, we find that $f|_{G^0}: G^0 \to H^0$ is an isomorphism. Moreover, the homomorphism $G/G^0 \to H/H^0$ between constant groups is an isomorphism, since this can be checked on either the special or generic fiber. Thus a diagram chase shows that $f$ is an isomorphism.
\end{proof}

\begin{theorem}\label{theorem:aut-scheme}
Let $S$ be a scheme and let $G$ be a geometrically reductive smooth affine $S$-group scheme. The functor $\underline{\Aut}_{G/S}$ is representable by an open and closed subscheme $\Aut_{G/S}$ of $\underline{\Hom}_{S\textrm{-}\rm{gp}}(G, G)$.
\end{theorem}

\begin{proof}
By spreading out, we may and do assume that $S$ is noetherian, so that $\mathscr{H} \coloneqq \underline{\Hom}_{S\textrm{-}\rm{gp}}(G, G)$ is locally noetherian by Lemma~\ref{lemma:demazure-rep}. Note that there is a universal $S$-homomorphism  
\[
f: G \times_S \mathscr{H} \to G \times_S \mathscr{H}
\]
whose fiber over a given section of $\mathscr{H}$ is the corresponding endomorphism of $G$. By Lemmas~\ref{lemma:iso-kernel} and \ref{lemma:dvr-morphism}, since $\mathscr{H}$ is locally noetherian the locus $U$ of $u \in \mathscr{H}$ such that $\ker f_u = \{1\}$ is open and closed, and $f_U: G \times_S U \to G \times_S U$ is an isomorphism. It follows from this reasoning that $U$ represents $\underline{\Aut}_{G/S}$.
\end{proof}

To use Theorem~\ref{theorem:aut-scheme}, it will be necessary to establish some more properties of the $\Aut$-scheme. If $G$ is a weakly reductive group scheme over a scheme $S$, then we will see in Lemma~\ref{lemma:aut-scheme-smooth} that $\Aut_{G/S}$ is always smooth. However, the following example shows that $\Aut_{G/k}$ may fail to be smooth if $k$ is a field of characteristic $p > 0$ and $G$ is a reductive $k$-group with component group of order divisible by $p$.

\begin{example}\label{example:non-smooth-aut}
Let $G = \bG_m \times \underline{\bZ/p \bZ}$ over a field $k$ of characteristic $p > 0$. If $S$ is a $k$-scheme, then $S$-automorphisms of $G_S$ correspond to pairs $(\phi_0, g)$, where $\phi_0: \bG_{m, S} \to \bG_{m, S}$ is an $S$-automorphism and $g \in G(S)$ is a section of order $p$ such that for all $s \in S$, $g_s$ does not lie in $G^0(k(s))$; the correspondence is given by sending an $S$-automorphism $\phi$ of $G_S$ to $(\phi|_{G_S^0}, \phi(x))$, where $x = (1, 1) \in \bG_m(S) \times \underline{\bZ/p \bZ}(S)$. Using this, one sees that $\Aut_{G/k}^0 \cong \mu_p$. In particular, $\Aut_{G/k}$ is not smooth.
\end{example}

\begin{lemma}\label{lemma:inner-open}
Let $G$ be a (possibly disconnected) reductive group over a field $k$. The natural map $G \to \Aut_{G/k}$ is open. If $k$ is perfect, then $\phi: G \to (\Aut_{G/k})_{\rm{red}}$ is flat.
\end{lemma}

\begin{proof}
Since formation of $(\Aut_{G/k})_{\rm{red}}$ commutes with separable field extensions on $k$ and purely inseparable extensions leave topological spaces unchanged, we may and do assume that $k$ is algebraically closed. To show that $\phi$ is flat, it suffices to show that the map $G^0(k) \to (\Aut_{G/k})^0(k)$ is surjective. Indeed, then $\phi$ is a dominant map from a smooth finite type $k$-scheme to a smooth $k$-scheme, so it is generically flat, and being a group homomorphism translation arguments show that it is flat.

Let $r: \Aut_{G/k} \to \Aut_{G^0/k} \times \Aut_{(G/G^0)/k}$ denote the natural restriction homomorphism.  The map $G^0 \to \Aut_{G/k}$ induces a map $Z(G^0) \to \ker r$. We claim that the image of $Z(G^0)$ in $\ker r$ is of finite index. Once this is done, the lemma will follow: indeed, then 
\[
\dim \ker r|_{\phi(G^0)} = \dim \ker r|_{\phi(Z(G^0))} = \dim \ker r.
\]
Since the group scheme of outer automorphisms of $G^0$ is etale, $\mathrm{pr}_1 \circ r: G^0 \to \Aut_{G^0/k}^0$ is surjective. Since $\Aut_{(G/G^0)/k}$ is finite, we find in particular that $\dim r(\phi(G)) = \dim r(\Aut_{G/k})$, so
\[
\dim \phi(G^0) = \dim \ker r + \dim r(\Aut_{G/k}) = \dim \Aut_{G/k}.
\]
Therefore $\phi(G^0)$ is an open subgroup scheme of the smooth $k$-group scheme $(\Aut_{G/k})_{\rm{red}}$, so it contains $(\Aut_{G/k})^0$ and hence is equal to it.  

So it remains to show that the image of $Z(G^0)$ of finite index in $(\ker r)(k)$. First we must understand some properties of $\Aut_{G/k}$. Suppose that $f: G \to G$ is a $k$-homomorphism inducing the identity on $G^0$ and $G/G^0$, i.e., $f \in (\ker r)(k)$. We define a morphism $\lambda: G \to G$ by
\[
\lambda(g) = f(g)g^{-1}.
\]
One checks that $\lambda$ is a $1$-cocycle. Since $f|_{G^0}$ is the identity, we have
\[
\lambda(gh) = f(gh)(gh)^{-1} = f(g)f(h)h^{-1}g^{-1} = \lambda(g)
\]
for all functorial points $g$ and $h$ of $G$ and $G^0$, respectively. Thus $\lambda$ factors through a morphism $G/G^0 \to G$. Moreover, for such $g$ and $h$ we have
\[
\lambda(g) = \lambda(hg) = f(hg)(hg)^{-1} = h\lambda(g)h^{-1},
\]
so in fact $\lambda$ has image lying in $Z(G^0)$. Thus $\lambda$ is a $1$-cocycle $G/G^0 \to Z(G^0)$. Note that any such $\lambda$ determines $f$ uniquely.

As in ordinary group cohomology, there is a short exact sequence
\[
0 \to \mathrm{B}^1(G/G^0, Z(G^0)) \to \mathrm{Z}^1(G/G^0, Z(G^0)) \to \mathrm{H}^1(G/G^0, Z(G^0)) \to 0.
\]
Notice that $\mathrm{B}^1(G/G^0, Z(G^0)) = \mathrm{B}^1((G/G^0)(k), Z(G^0)(k))$ (and so on) because $G/G^0$ is constant. Under the correspondence between $f$ and $\lambda$ as in the previous paragraph, $\mathrm{B}^1(G/G^0, Z(G^0))$ consists of those $k$-homomorphisms $f$ induced by conjugation by an element of $Z(G^0)(k)$. Moreover, if $G/G^0$ is of order $n$ then $\mathrm{H}^1(G/G^0, Z(G^0))$ is $n$-torsion, so $\mathrm{H}^1(G/G^0, Z(G^0))$ admits a surjection from $\mathrm{H}^1(G/G^0, Z(G^0)[n])$ and hence is finite. So indeed the map $Z(G^0)(k) \to (\ker r)(k)$ has finite index image, and we are done. 
\end{proof}

\begin{lemma}\label{lemma:aut-scheme-smooth}
If $G$ is a weakly reductive group scheme over $S$, then $\Aut_{G/S}$ is smooth and the natural map $\phi: G \to \Aut_{G/S}$ is flat. If $Z(G^0)$ is smooth, then $\phi$ is smooth.
\end{lemma}

\begin{proof}
Let $\fg = \Lie G$. We claim that $\mathrm{H}^2(G_s, \fg_s) = 0$ for all $s \in S$. There is a Hochschild--Serre spectral sequence
\[
\mathrm{E}_2^{pq} = \mathrm{H}^p(G_s/G_s^0, \mathrm{H}^q(G_s^0, \fg_s)) \Rightarrow \mathrm{H}^{p+q}(G_s, \fg_s),
\]
and since $G_s/G_s^0$ has order prime to $\chara k(s)$ for all $s \in S$ we find
\[
\mathrm{H}^n(G_s, \fg_s) = \mathrm{H}^0(G_s/G_s^0, \mathrm{H}^n(G_s^0, \fg_s))
\]
for all $n$. By \cite[Exp.\ XXIV, Corollaire 1.13(ii)]{sga3} (in which reference reductive groups over \textit{fields} are also required to be connected), we have $\mathrm{H}^2(G_s^0, \fg_s) = 0$, so indeed $\mathrm{H}^2(G_s, \fg_s) = 0$ for all $s \in S$. The same argument, using \cite[Exp.\ XXIV, Corollaire 1.15.1]{sga3}, shows that $\mathrm{H}^1(G_s, \fg_s) = 0$ for all $s \in S$ if $Z(G^0)$ is smooth. 

To show that $\underline{\Aut}_{G/S}$ is smooth, it suffices to verify the infinitesimal criterion of smoothness, and for this \cite[Exp.\ III, Corollaire 2.9(ii)]{sga3} shows that it suffices to show $\mathrm{H}^2(G_s, \fg_s) = 0$ for all $s \in S$, which we showed above. Moreover, \cite[Exp.\ III, Corollaire 2.9(i)]{sga3} shows that if $\mathrm{H}^1(G_s, \fg_s) = 0$ for all $s \in S$, then the morphism $\phi$ satisfies the infinitesimal criterion of smoothness, so it is smooth. In particular, the previous paragraph shows that if $Z(G^0)$ is smooth then $\phi$ is smooth. 

Finally, to show that $\phi$ is flat in general, we may use the fibral flatness criterion \cite[IV\textsubscript{3}, Th\'eor\`eme 11.3.10]{EGA} to assume that $S = \Spec k$ for a field $k$. Thus since $\Aut_{G/k}$ is smooth, we may conclude using Lemma~\ref{lemma:inner-open}.
\end{proof}

\begin{prop}\label{prop:center}
If $G$ is a weakly reductive $S$-group scheme, the functorial center $Z(G)$ is an $S$-group scheme of multiplicative type. If $Z(G^0)$ is smooth, then $Z(G)$ is smooth.
\end{prop}

\begin{example}\label{example:p-is-bad}
Before proving the proposition, we offer the following example to show that geometric reductivity for a smooth affine $G$ is not enough for flatness of $Z(G)$. Let $A = \bZ_p[\zeta_p]$. There is a non-split central extension of constant $A$-group schemes
\[
1 \to \bZ/p \to U \to (\bZ/p)^2 \to 1,
\]
where $U$ is isomorphic to the group of $\bF_p$-points of the unipotent radical of a Borel in $\SL_3$, i.e., $U$ is the ``Heisenberg group" over $\bF_p$. There is a homomorphism of $A$-group schemes $\bZ/p \to \mu_p$ which is trivial on the special fiber and an isomorphism on the generic fiber, given by the choice of a primitive $p$th root of unity $\zeta_p$. Pushing the above extension forward by the map $\bZ/p \to \mu_p \to \bG_m$ gives an extension
\[
1 \to \bG_m \to G \to (\bZ/p)^2 \to 1,
\]
where $G$ has commutative special fiber and non-commutative generic fiber. In fact, $Z(G)$ has special fiber $G$ and generic fiber $G^0$, so it is not flat.
\end{example}

\begin{proof}
Working locally and spreading out, we may and do assume that $S$ is locally noetherian. Consider the morphism $\phi: G \to \Aut_{G/S}$. By Lemma~\ref{lemma:aut-scheme-smooth}, $\phi$ is flat, so $Z(G) = \ker \phi$ is a flat closed $S$-subgroup scheme of $G$, and it is smooth provided that $Z(G^0)$ is smooth. So we need only show that $Z(G)$ is of multiplicative type. 

We define $C_{G^0}(G) \colonequals \ker \phi|_{G^0} = \ker \phi|_{Z(G^0)}$; note that $Z(G^0)$ is of multiplicative type. By \cite[Exp.\ IX, Th\'eor\`eme 6.8]{sga3}, it follows that $C_{G^0}(G)$ is of multiplicative type. Moreover, there is a short exact sequence
\[
1 \to C_{G^0}(G) \to Z(G) \to Z(G)/C_{G^0}(G) \to 1.
\]
Since $Z(G)/C_{G^0}(G)$ is a closed commutative flat and finitely presented $S$-subgroup scheme of $G/G^0$, it is finite \'{e}tale commutative of order invertible on $S$, and in particular it is of multiplicative type. Thus $Z(G)$ is a commutative extension of multiplicative type group schemes, so it is of multiplicative type by \cite[Exp.\ XVII, Prop.\ 7.1.1]{sga3}.
\end{proof}

\begin{remark}
With more work, the kind of arguments used in the proof of Lemma~\ref{lemma:inner-open} (expanded upon in the case of ``abstract" groups in \cite{Extensions}) can be used to proved that if $G$ is weakly reductive over $S$ then there is a short exact sequence 
\[
1 \to G/Z(G) \to \Aut_{G/S} \to \Out_{G/S} \to 1
\]
in which $\Out_{G/S}$ is an \'{e}tale-locally constant $S$-group scheme, just as in the theory of reductive group schemes. Proving this would take us somewhat far afield, so we omit it.
\end{remark}

We note that the above discussion gives a slight strengthening of \cite[Lemma 6.8]{martin03}.

\begin{corollary} \label{cor:finiteness of weyl}
Let $k$ be a field, let $G$ be a finite type $k$-group scheme, and let $H \subset G$ be a closed (not necessarily connected) reductive $k$-subgroup. The quotient $N_G(H)/HC_G(H)$ is finite. If $H/H^0$ has order prime to $\chara k$, then $N_G(H)/HC_G(H)$ is \'{e}tale.
\end{corollary}

\begin{proof}
We may and do assume that $k$ is algebraically closed. Let $\phi: H \to \Aut_{H/k}$ be the natural map. By Lemma~\ref{lemma:inner-open}, $\phi(H)$ is an open subset of $\Aut_{H/k}$, so $\phi(H)$ is an open subgroup scheme of $(\Aut_{H/k})_{\rm{red}}$ and the quotient $A \coloneqq \Aut_{H/k}/\phi(H)$ (which exists as a scheme by \cite[Exp.\ VI\textsubscript{A}, Th\'eor\`eme 3.2]{sga3}) is locally of finite type. Moreover, $A_{\rm{red}}$ is \'etale: since $\phi(H)$ is (topologically) open in $\Aut_{H/k}$, $A$ is discrete, and any reduced discrete scheme locally of finite type over an algebraically closed field is \'{e}tale. The natural map $N_G(H) \to A$ has kernel $HC_G(H)$, so $N_G(H)/HC_G(H)$ is a finite type closed subscheme of $A$. Since $A_{\rm{red}}$ is \'{e}tale, it follows that $N_G(H)/HC_G(H)$ is finite. If $H/H^0$ has order prime to $\chara k$, so $H$ is weakly reductive, then already $A$ is \'{e}tale by Lemma~\ref{lemma:aut-scheme-smooth} and thus $N_G(H)/HC_G(H)$ is \'{e}tale.
\end{proof}

\begin{defn}
If $G$ is a finite type $k$-group scheme and $H$ is a reductive $k$-subgroup scheme of $G$, then we set $W_H \colonequals N_G(H)/HC_G(H)$, which we call the \textit{Weyl group} of the pair $(G, H)$.
\end{defn} 

If $H$ is a maximal torus of a smooth affine $G$, then this is the usual Weyl group $W$ of $G$.

\begin{prop} \label{prop:p not divide weyl}
Let $G$ be a reductive group over a field $k$, and let $H$ be a reductive $k$-subgroup scheme of $G$. For a prime $p$, if $p \nmid |W|$ and $p \geq |H/H^0|$ then $p \nmid |W_H|$.
\end{prop}

\begin{proof}
We may and do assume that $k$ is algebraically closed. Let $g \in N_G(H)(k)$; we need to show that under our hypotheses, $g^n$ acts by an inner automorphism on $H$ for some integer $n$ prime to $p$. Let $T$ be a maximal torus of $H$, and note that conjugacy of maximal tori in $H^0$ implies that after translation by $H^0(k)$, we may assume $g \in N_G(T)(k)$. Since $N_G(T)/C_G(T)$ is a subquotient of $W$, it follows that $g^{|W|} \in C_G(T)(k)$. Thus after replacing $g$ by $g^{|W|}$, we may and do assume that $g$ centralizes $T$. In particular, $g$ acts trivially on the Dynkin diagram of $(H^0, T)$, so $g$ acts on $H^0$ by an inner automorphism. Thus after further translation by $H^0(k)$ we may and do assume that $g$ centralizes $H^0$. Since $p \geq |H/H^0|$, we may pass to a further prime-to-$p$ power of $g$ to assume that $g$ acts trivially on $H/H^0$. (It is a general fact, easily checked, that if $p \mid |\Aut(A)|$ for a finite group $A$, then $p < |A|$.) 

Now that $g$ acts trivially on $H^0$ and $H/H^0$, the argument of Lemma~\ref{lemma:inner-open} shows that $\Ad(g)$ corresponds to a $1$-cocycle $\eta: H/H^0 \to Z(H^0)$ (which lands in $Z(H^0)(k)$ since $H/H^0$ is constant). The corresponding class in $\mathrm{H}^1(H/H^0, Z(H^0)(k))$ is killed by $|H/H^0|$, so further passing from $g$ to $g^{|H/H^0|}$ we may and do assume that $\Ad(g)$ is cohomologically trivial, from which is follows that $g$ acts on $H$ by conjugation by an element of $h \in Z(H^0)(k)$. Thus $g$ now acts on $H$ by inner automorphisms, and we are done.
\end{proof}

\subsection{Abelianization}

Our final goal in this section is to show the existence of the ``abelianization" of a weakly reductive group scheme $G$. We begin with the following folkloric lemma.

\begin{lemma}\label{lemma:faithfully-flat-criterion}
Let $S$ be a scheme, and let $f: G \to H$ be a homomorphism of finitely presented $S$-group schemes. The following are equivalent.
\begin{enumerate}
    \item $f$ is faithfully flat,
    \item $f$ is an epimorphism of fppf sheaves and $\ker f$ is flat.
\end{enumerate}
\end{lemma}

\begin{proof}
Omitted. 
\end{proof}

\begin{lemma}\label{lemma:power-homomorphism}
Let $1 \to M \to E \to H \to 1$ be a central extension, where $M$ is an $S$-group scheme of multiplicative type and $H$ is a finite \'{e}tale group scheme of constant order $n$ invertible on $S$. Letting $N = n^2$, for every integer $d \geq 1$ the $S$-morphism $[Nd]: E \to E$ is a homomorphism. If $M$ is moreover a torus, then $E[N] \to H$ is faithfully flat, where $E[N] \colonequals [N]^{-1}(1)$.
\end{lemma}

\begin{proof}
By spreading out and \'{e}tale-localizing around a point of $S$, we may and do assume that $S = \Spec A$ for a strictly henselian noetherian local ring $A$. 
In particular, since $H$ is finite \'{e}tale it is a constant group and there exists a scheme-theoretic section $H \to E$. We may moreover pushforward by an inclusion of $M$ into a torus to assume that $M$ is a torus. The existence of a section implies that for every $S$-scheme $S'$ the sequence
\[
1 \to M(S') \to E(S') \to H(S') \to 1
\]
is exact. 

First note that as an extension of $H$ by $M$ which admits a section, $E$ corresponds to a cohomology class in $\mathrm{H}^2(H, M)$, the Hochschild cohomology group (see for example \cite[II, \S 3, Prop.\ 2.3]{DG} or \cite[Prop.\ 2.3.6]{Demarche}). Concretely, this is the space of $2$-cocycles $H \times H \to M$ modulo coboundaries. Since $H$ is a constant group scheme, we have $\mathrm{H}^2(H, M) = \mathrm{H}^2(H(S), M(S))$. Since $H(S)$ is finite of order $n$, $\mathrm{H}^2(H(S), M(S))$ is killed by $n$ by the classical theory. Since $n$ is invertible on the strictly henselian $S$, via the correspondence between extensions and classes in $\mathrm{H}^2$, the image of the extension corresponding to multiplication by $n$ on $M$ is
\[
1 \to M(S) \xleftarrow[n]{\cong} M(S)/M[n](S) \to E(S)/M[n](S) \to H(S) \to 1.
\]
Thus this extension is split, i.e., there is a section $\alpha_0: H(S) \to E(S)/M[n](S)$ which is a homomorphism. This is equivalent to a section $H \to E/M[n]$, which we also denote by $\alpha_0$. 

Since $S$ is strictly henselian, $\alpha_0$ lifts to a section $\alpha: H \to E$. Now $\alpha_0(x)^n = \alpha_0(x^n) = 1$ for all $x \in H(S)$, so we have $\alpha(x)^n \in M[n](S)$ and thus
\[
\alpha(x)^N = \alpha(x)^{n^2} = 1.
\]
Moreover, since $\alpha_0$ is a homomorphism we have $\alpha(xx')^{-1}\alpha(x)\alpha(x') \in M[n](S)$ for all $x, x' \in H(S)$. In other words,
\[
\alpha(x)\alpha(x') = \alpha(xx') g_{x, x'}
\]
for some $g_{x, x'} \in M[n](S)$. 

Now we conclude the proof that $[Nd]: E \to E$ is a homomorphism. We shall check this on $S'$-points for every $S$-scheme $S'$. Any element of $G(S')$ is Zariski-locally of the form $\alpha(x)g$ for some $x \in H(S')$ and $g \in M(S')$, and using centrality of $M$ we compute, for $x, x' \in H(S')$ and $g, g' \in M(S')$,
\[
(\alpha(x)g\alpha(x')g')^{Nd} = \alpha(xx')^{Nd}g_{x, x'}^{Nd} g^{Nd} g'^{Nd} = g^{Nd} g'^{Nd} = (\alpha(x) g)^{Nd} (\alpha(x') g')^{Nd},
\]
so indeed $[Nd]: E \to E$ is a homomorphism. 

Finally, we show that $E[N] \to H$ is faithfully flat when $M$ is a torus. For this, let $h$ be a local section of $H$. After fppf localization, we may lift $h$ to a section $e$ of $E$. Note that $e^n$ is a local section of $M$, so because $M$ is a divisible group scheme, we may pass to a further fppf localization to assume $e^n = m^n$ for some local section $m$ of $M$. Because $M$ is central in $E$, we see that $(em^{-1})^n = 1$, and thus $em^{-1}$ is a local section of $E[N]$ mapping to $h$. So faithful flatness follows from Lemma~\ref{lemma:faithfully-flat-criterion}.
\end{proof}

\begin{prop}\label{prop:abelianization}
Suppose $G$ is a weakly reductive $S$-group scheme.  
There exists a smooth $S$-group scheme $G^{\mathrm{ab}}$ of multiplicative type and a faithfully flat homomorphism $\pi: G \to G^{\mathrm{ab}}$ with the universal property that for any fppf abelian sheaf $H$ on the category of $S$-schemes and homomorphism of sheaves $f: G \to H$, there is a unique homomorphism $G^{\mathrm{ab}} \to H$ through which $f$ factors.

The kernel $\sD(G)$ of $\pi$ represents the fppf-sheafification of the functor $S' \mapsto [G(S'), G(S')]$ on $S$-schemes $S'$ and in particular the formation of $G^{\mathrm{ab}}$ commutes with any base change on $S$.
\end{prop}

\begin{example}\label{example:p-is-still-bad}
Before giving the proof, we illustrate again the relevance of weak reductivity (as opposed to geometric reductivity). For the $G$ in Example~\ref{example:p-is-bad}, we claim that $G^{\mathrm{ab}}$ cannot exist as a scheme. Indeed, note that $G_s$ is commutative, so on the special fiber we have $\sD(G_s) = 1$. However, since $\zeta_p$ does not lie in $pA$, it follows that $G_{A/p}$ is not commutative, so $\sD(G_{A/p}) \neq 1$. Since every scheme over $A/p$ with trivial special fiber is trivial (by ``nilpotent Nakayama"), this is a contradiction.
\end{example}

\begin{proof}[Proof of Proposition~\ref{prop:abelianization}]
The claims in the second paragraph of the proposition follow directly from the universal property of the first paragraph. First note that if $G$ has connected fibers then \cite[Thm.\ 5.3.1]{conrad} shows that $G^{\mathrm{ab}}$ exists and is a torus. In general, any $S$-homomorphism $f: G \to H$ as in the statement of the proposition induces an $S$-homomorphism $G^0 \to H$, so $f$ kills $\sD(G^0)$ and hence factors through an $S$-homomorphism $G/\sD(G^0) \to H$, where $G/\sD(G^0)$ is a smooth affine $S$-group scheme with torus identity component. Thus we may and do assume that $G^0$ is a torus.  

Working \'{e}tale-locally on $S$ (using effectivity of \'{e}tale descent for affine $S$-schemes), we may choose representatives $h_1, \dots, h_n \in G(S)$ for $G/G^0$. Define the map $\phi: (G^0)^n \to G^0$ by $\phi(g_1, \dots, g_n) = \prod_{i=1}^n h_ig_ih_i^{-1}g_i^{-1}$. Note that $\phi$ is an $S$-homomorphism, so by \cite[Exp.\ IX, Thm.\ 6.8]{sga3} there is some $S$-subgroup scheme $M$ of $G^0$ of multiplicative type through which $\phi$ factors and such that the factored map $\phi: (G^0)^n \to M$ is faithfully flat. Since $f$ vanishes on all commutators, it annihilates $M$, and thus $f$ must factor through $G/M$. By definition, if $g \in G^0(S')$ is a section, then $h_igh_i^{-1} \in gM$, so since $G^0$ is commutative it follows that $G^0/M$ is central in $G/M$. Thus replacing $G$ by $G/M$ we may and do assume that $G^0$ is central in $G$.

Working Zariski-locally on $S$, we may and do assume that the index of $G^0$ in $G$ is constant, say equal to $n$. By Lemma~\ref{lemma:power-homomorphism}, $G[N] \colonequals [N]^{-1}(1)$ is an $S$-subgroup scheme of $G$ where $N = n^2$. By the same lemma, the map $G[N] \to G/G^0$ is faithfully flat, so there is a short exact sequence
\[
1 \to G^0[N] \to G[N] \to G/G^0 \to 1.
\]
Since $G^0[N]$ and $G/G^0$ are both finite \'{e}tale (the former because $N$ is invertible on the base), it follows that $G[N]$ is also finite \'{e}tale. 

By working \'{e}tale locally on $S$, we may assume that $G[N]$ is a constant group scheme. In this case, the functor on $S$-schemes $S' \mapsto [G[N](S'), G[N](S')]$ is represented by the constant $S$-group scheme $\sD(G[N])$. Moreover, if $\sD(G)$ denotes the sheafification of the functor $S' \mapsto [G(S'), G(S')]$ then we have $\sD(G) = \sD(G[N])$: indeed, for an $S$-scheme $S'$ and $g, h \in G(S')$, centrality of $G^0$ in $G$ shows that the commutator $ghg^{-1}h^{-1}$ depends only on the images of $g$ and $h$ in $(G/G^0)(S')$. In particular, since the map $G[N] \to G/G^0$ is an epimorphism of fppf sheaves we may pass to an fppf cover of $S'$ to assume $g, h \in G[N](S')$. So indeed $\sD(G) = \sD(G[N])$, which is a finite \'{e}tale group scheme killed by $N$, so its order is invertible on $S$. 

Finally, note that $G^{\mathrm{ab}} \colonequals G/\sD(G[N])$ is a smooth $S$-group scheme by \cite[Exp.\ V, Thm. 4.1]{sga3}. It is commutative by the previous paragraph. Moreover, since $G^0$ is a torus the image of $G^0$ in $G^{\mathrm{ab}}$ (which makes sense by \cite[Exp.\ IX, Thm.\ 6.8]{sga3}) is also a torus, which must be equal to the relative identity component $(G^{\mathrm{ab}})^0$ for dimension reasons. Moreover, $G/G^0 \to G^{\mathrm{ab}}/(G^{\mathrm{ab}})^0$ is an epimorphism of sheaves, so the component group of $G^{\mathrm{ab}}$ is of order invertible on $S$. Thus it follows from \cite[Exp.\ XVII, Prop.\ 7.1.1]{sga3} that $G^{\mathrm{ab}}$ is a group scheme of multiplicative type, as desired.
\end{proof}

The center and abelianization of $G$ are related.  We begin with a useful lemma.

\begin{lemma}\label{lemma:equivariant-isogeny}
Let $T_1$ and $T_2$ be $S$-group schemes of multiplicative type such that $T_2$ is a torus, and let $f: T_1 \to T_2$ be an isogeny. Let $\Gamma$ be a finite group acting on $T_1$ and $T_2$ compatibly with $f$, and define maps
\begin{align*}
    \phi: T_1 \to \prod_{\gamma \in \Gamma} T_1, \,\,\,\,\,\,\,\,\,\, &\phi(x) = ((\gamma \cdot x) x^{-1})_{\gamma \in \Gamma} \\
    \psi: \prod_{\gamma \in \Gamma} T_2 \to T_2, \,\,\,\,\,\,\,\,\,\, &\psi((y_\gamma)_{\gamma \in \Gamma}) = \prod_{\gamma \in \Gamma} (\gamma \cdot y_\gamma) y_\gamma^{-1}
\end{align*}
The natural map $\overline{f}: \ker \phi \to \coker \psi$ is an isogeny, and if $\Gamma$ is of order $n$ then $\ker \overline{f}$ is contained in $[n]^{-1}(\ker f)$.
\end{lemma}

Note that $\ker \phi$ and $\coker \psi$ are $S$-group schemes of multiplicative type and that $\ker \overline{f}$ is of multiplicative type, and in particular flat, by \cite[Exp.\ IX, Thm.\ 6.8]{sga3}.

\begin{proof}
By Lemma~\ref{lemma:faithfully-flat-criterion}, to show that $\overline{f}$ is faithfully flat it suffices to show that $\overline{f}$ is an epimorphism of fppf sheaves. To this end, let $t_2$ be a local section of $T_2$. Since $T_2$ is a torus, after localizing we may assume there is some section $t_2'$ such that $t_2'^n = t_2$. Further localizing, there exists a section $t_1$ of $T_1$ such that $f(t_1) = t_2'$. We  then have
\[
f\left(\prod_{\gamma \in \Gamma} (\gamma \cdot t_1)\right) = \prod_{\gamma \in \Gamma} (\gamma \cdot t_2') = t_2 \cdot \prod_{\gamma \in \Gamma} (\gamma \cdot t_2') t_2'^{-1},
\]
which has the same image in $\coker \psi$ as $t_2$. Since $\prod_{\gamma \in \Gamma} (\gamma \cdot t_1)$ lies in $\ker \phi$, we see that indeed $\overline{f}$ is an epimorphism of fppf sheaves. 

Now suppose that $t_1$ is a local section of $\ker \phi$ such that $\overline{f}(t_1) = 1$. We claim that then $f(t_1)^n = 1$. We may write $f(t_1) = \prod_{\gamma \in \Gamma} (\gamma \cdot t_{2, \gamma}) t_{2, \gamma}^{-1}$ locally. Using $\Gamma$-equivariance of $f$ and the fact that $t_1$ is fixed by $\Gamma$, we have
\[
f(t_1)^n = \prod_{\gamma' \in \Gamma} (\gamma' \cdot f(t_1)) = \prod_{\gamma, \gamma' \in \Gamma} (\gamma' \gamma \cdot t_{2, \gamma})(\gamma' \cdot t_{2, \gamma}^{-1}) = \prod_{\gamma \in \Gamma} \left(\prod_{\gamma' \in \Gamma} (\gamma'\gamma \cdot t_{2, \gamma})(\gamma' \cdot t_{2, \gamma})^{-1}\right) = 1,
\]
where the final equality follows by reindexing. So indeed $\ker \overline{f}$ lies in $[n]^{-1}(\ker f)$, and in particular it is quasi-finite.  Since it is of multiplicative type it is finite, and hence $\overline{f}$ is an isogeny.
\end{proof}

\begin{prop} \label{proposition:isogenies}
Let $G$ be a weakly reductive $S$-group scheme. If $Z(G)$ is smooth then the natural map $f: Z(G)^0 \to G^{\mathrm{ab}, 0}$ is an isogeny of $S$-tori. If moreover $Z(\sD(G^0))$ is smooth, then $f$ is smooth (equivalently, \'{e}tale).
\end{prop}

\begin{proof}
Propositions~\ref{prop:center} and \ref{prop:abelianization} show that $Z(G)^0$ and $G^{\mathrm{ab}, 0}$ are $S$-tori under our assumptions, so it suffices to check the result on geometric fibers. In other words, we may and do assume that $S = \Spec k$ for an algebraically closed field $k$. Note that the natural $Z(G^0) \to G^{0, \mathrm{ab}}$ (whose target is \textit{not} $G^{\mathrm{ab}, 0}$) is an isogeny of multiplicative type groups with kernel $Z(\sD(G^0))$ (irrespective of whether $Z(G^0)$ is smooth): a maximal central torus $T_0$ of $G^0$ is of finite index in $Z(G^0)$ and there is the standard central isogeny $T_0 \times \sD(G^0) \to G^0$. 

Now $G/G^0$ acts on $Z(G^0)$ and $G^0/\sD(G^0)$ by conjugation, so we can apply Lemma~\ref{lemma:equivariant-isogeny} with $T_1 = Z(G^0)$, $T_2 = G^{0, \mathrm{ab}}$, $\Gamma = G/G^0$, and $f: T_1 \to T_2$ the natural map. We will also use the notation $\phi$, $\psi$, $\overline{f}$ of Lemma~\ref{lemma:equivariant-isogeny}. In this case $\ker \phi = C_{G^0}(G)$, which admits $Z(G)^0$ as an open and closed $S$-subgroup scheme: indeed, $C_{G^0}(G) = G^0 \cap Z(G)$, and $Z(G)^0$ (resp.\ $G^0$) is open and closed in $Z(G)$ (resp.\ $G$). Moreover, $\coker\psi = G^{\mathrm{ab}, 0}$. Thus to establish the proposition it suffices to show that $\overline{f}$ is an isogeny, which is smooth provided $Z(\sD(G^0))$ is smooth. The fact that $\overline{f}$ is an isogeny follows directly from Lemma~\ref{lemma:equivariant-isogeny}. The last statement of Lemma~\ref{lemma:equivariant-isogeny} shows that $\ker \overline{f}$ lies in $[n]^{-1}(Z(\sD(G^0)))$, where $n$ is the order of $G/G^0$. Since $n$ is invertible on $S$ by hypothesis, smoothness of $Z(\sD(G^0))$ implies smoothness of $[n]^{-1}(Z(\sD(G^0)))$.
\end{proof}

\section{Centralizers} \label{sec:centralizers}

In this section we study various kinds of centralizers in weakly reductive group schemes. Before diving into the results, we would like to summarize what is available in the literature for centralizers of weakly reductive subgroups of weakly reductive groups.

\begin{enumerate}
    \item Over a field of pretty good characteristic (see Definition~\ref{def:prettygoodprimes}), \cite{herpel} shows that the centralizer of any subgroup scheme of any connected reductive group is smooth. In fact, this property characterizes pretty good characteristic.
    \item Over a field of good characteristic $p > 0$, the centralizer of any subgroup scheme of a connected reductive group has no $p$-torsion in its component group; this follows via a short argument with the Springer isomorphism. This general statement is false in every bad characteristic, as Springer showed in \cite[Theorem 4.12]{Springer} by considering centralizers of regular unipotent elements.
    \item\label{item:prasad-yu} If $G$ is a connected reductive group over a field $k$ of characteristic $p \geq 0$ and $H$ is a finite subgroup of $G(k)$ of order prime to $p$, then $C_G(H)$ has reductive identity component by \cite[Theorem 2.1]{Prasad-Yu}, and its component group is of order prime to $p$ (even in bad characteristic) by \cite[Proposition VIII.5.11]{Fargues-Scholze}.
    \item If $H$ is a connected reductive subgroup of a connected reductive group $G$ over a field of characteristic $p > 0$, then we are not aware of any general results concerning reductivity of $C_G(H)$ apart from the ``classical" case that $H$ is of multiplicative type \cite[Lemma 2.2.4]{conrad} or the upcoming Corollary~\ref{corollary:centralizer-of-weakly-reductive}.
    \item If $G$ is a weakly reductive group scheme over a base scheme $S$ and $H$ is a finite subgroup of $G(S)$ of order invertible on $S$, then $C_G(H)$ is smooth and affine with reductive identity component; this follows from simple deformation theory and (\ref{item:prasad-yu}). If $H$ is moreover solvable, then \cite[Theorem A.12]{DHKM} shows that $C_G(H)$ has finite component group. Apart from this, we are not aware of any results in the literature concerning smoothness, reductivity, or finiteness of component groups for centralizers of weakly reductive subgroup schemes of $G$ prior to Corollaries~\ref{corollary:centralizer-of-weakly-reductive} and \ref{cor:double-centralizer} below.  
\end{enumerate}

\subsection{Centralizers of weakly reductive subgroup schemes} \label{ss:centralizerswr}

We will now input the general results of Section~\ref{sec:weaklyreductive} into concrete results on centralizers. We begin with the following lemma.

\begin{lemma}\label{lemma:removing-special-fiber-components}
    Let $A$ be a DVR, and let $X$ be a locally noetherian $A$-scheme. If $X_0$ is an open and closed subscheme of the special fiber $X_s$, then the natural map $X - X_0 \to X$ is affine. In particular, if $X$ is affine then $X - X_0$ is also affine.
\end{lemma}

\begin{proof}
    Affineness of a morphism can be checked Zariski-locally on the target, so we may freely shrink $X$ to assume that the special fiber of $X$ is connected. In this case, $X_0$ is either empty or all of $X_s$. If $X_0$ is empty, then the lemma is obvious; otherwise, $X_0 = V(\pi)$, where $\pi$ is a uniformizer of $A$, and the lemma is again clear.
\end{proof}

The proof of the following theorem is similar in spirit to one of the proofs of \cite[Theorem 2.1]{Prasad-Yu}, which shows reductivity of the centralizer $C_G(\Lambda)$ of a finite group $\Lambda$ over a field $k$ by realizing it as the stabilizer of the conjugation action of $G$ on $\underline{\Hom}_{k\textrm{-}\rm{gp}}(\Lambda, G)$.

\begin{theorem}\label{theorem:main-centralizer-theorem}
    Let $S$ be a scheme, and let $G$ and $H$ be geometrically reductive smooth affine $S$-group schemes. Suppose $f: H \to G$ is an $S$-homomorphism such that
    \begin{enumerate}
        \item $\mathrm{H}^1(H_s, \fg_s) = 0$ for all $s \in S$,
        \item $f_s(H_s)$ is $G_s$-cr for all $s \in S$.
    \end{enumerate}
    Then $C_G(H)$ is geometrically reductive smooth affine $S$-scheme. If $G$ is weakly reductive and $\chara k(s)$ is good for $G_s$ for all $s \in S$, then $C_G(H)$ is also weakly reductive.
\end{theorem}

\begin{proof}
    By (1) and deformation theory \cite[Exp.\ III, Corollaire 2.8]{sga3}, the orbit map $\phi: G \to \underline{\Hom}_{S\textrm{-}\rm{gp}}(H, G)$ is smooth, and thus $C_G(H)$ is smooth and affine. To show that $C_G(H)$ is geometrically reductive it suffices by Theorem~\ref{theorem:alper} to assume that $S = \Spec A$ for a complete DVR $A$ with algebraically closed residue field. In this case, Theorem~\ref{theorem:hom-scheme} shows that $\underline{\Hom}_{S\textrm{-}\rm{gp}}(H, G)$ is a disjoint union of finite type $S$-affine $S$-schemes.
    
    Let $C$ be the schematic closure of the $G$-orbit of $f_\eta$ in $\underline{\Hom}_{S\textrm{-}\rm{gp}}(H, G)$, so $C$ is a $G$-stable closed subscheme of $\underline{\Hom}_{S\textrm{-}\rm{gp}}(H, G)$ through which the orbit map of $f$ factors. Let $C_1$ be the $G$-stable open subscheme of $C$ obtained by deleting all of the components of $C_s$ not containing a $G$-translate of $f_s$, so $C_1$ is affine by Lemma~\ref{lemma:removing-special-fiber-components}. Since $\phi$ is smooth, $\phi$ has open image; moreover, each fiber of $\phi$ has closed image by (2) and Lemma~\ref{lemma:bmr}. Consequently each fiber of $C_1$ is the (open) orbit of $f$ in that fiber, and the map $i: C_1 \to \underline{\Hom}_{S\textrm{-}\rm{gp}}(H, G)$ is an open embedding on both fibers. Since $C_1$ is flat, it follows that $i$ is \'etale and radicial, and thus \cite[IV\textsubscript{4}, Th\'eor\`eme 17.9.1]{EGA} shows that $i$ is an open embedding.
    
    Now we show that $C_G(H)$ is geometrically reductive; for this, it is equivalent to show that $G/C_G(H)$ is affine by Theorem~\ref{theorem:alper}. We will show in fact that the natural map $G/C_G(H) \to C_1$ is an isomorphism. By definition, $\phi$ factors through $C_1$, and the previous paragraph shows that this factored map is surjective. Note that the quotient $G/C_G(H)$ exists as a smooth separated algebraic space of finite type by work of Artin \cite[Corollary 6.3]{Artin-stacks}. Moreover, the induced morphism $G/C_G(H) \to C_1$ is a monomorphism, so by \cite[II, 6.15]{Knutson-algebraic-spaces} it follows that $G/C_G(H)$ is a scheme. We claim that the morphism $G/C_G(H) \to C_1$ is an isomorphism. Indeed, from the above we see that it is a smooth surjective monomorphism. Thus by \cite[IV\textsubscript{4}, Th\'eor\`eme 17.9.1]{EGA} it is a surjective open embedding and thus an isomorphism. As remarked above, this shows that $C_G(H)$ is geometrically reductive.

    For the final claim, we may and do assume that $S = \Spec k$ for an algebraically closed field $k$ of characteristic $p > 0$. If $\pi_0 C_G(H)$ has any $p$-torsion, then a simple argument with the Jordan decomposition shows that $C_G(H)(k)$ admits unipotent elements not lying in $C_G(H)^0(k)$. This does not happen in good characteristic by the argument of \cite[III, 3.15]{Springer-Steinberg}.
\end{proof}

In order to apply Theorem~\ref{theorem:main-centralizer-theorem}, we need the following simple lemma.

\begin{lemma}\label{lemma:cr-criterion}
    Let $k$ be an algebraically closed field, and let $G$ and $H$ be (possibly disconnected) reductive groups over $k$. Let $T$ be a maximal $k$-torus of $H$. Suppose $f: H \to G$ is a $k$-homomorphism such that $\mathrm{H}^1(H, V) = 0$ for all representations $V$ isomorphic to $\Lie G$ under some $k$-homomorphism $H \to G$ such that the multiset of weights for $T$ on $V$ is the same as the multiset of weights for $T$ on $\fg$. Then $f(H)$ is $G$-cr.
\end{lemma}

\begin{proof}
    Note that $\underline{\Hom}_{k\textrm{-}\rm{gp}}(H, G)$ admits a disjoint union decomposition into pieces on which the multiset of weights for $T$ on $\fg$ is constant; let $U$ be the piece containing $f$. By hypothesis and \cite[Exp.\ III, Corollaire 2.8]{sga3}, every orbit map $G \to U$ is smooth, and thus every orbit is an open subscheme of $U$. Thus every orbit is also closed. By Lemma~\ref{lemma:g-cr}, it follows that $f(H)$ is $G$-cr.
\end{proof}

Now we recall the following fundamental result of McNinch \cite{McNinch-semisimple}, which builds on work of Jantzen \cite{Jantzen-semisimple}. If $H$ is a (possibly disconnected) reductive group and $H_1, \dots, H_n$ are the simple factors of $\sD(H^0)$, we let $\ell_H = \inf_i(\rank H_i)$. Note that if $H^0$ is a torus, then $\ell_H = \infty$.

\begin{theorem}[{\hspace{1pt}\cite[Corollary 1.1.2]{McNinch-semisimple}}] \label{theorem:jantzen-semisimple}
    Let $k$ be a field of characteristic $p > 0$, and let $H$ be a connected reductive $k$-group. If $V$ is an algebraic $k$-representation of $H$ and $\dim V \leq p\ell_H$, then $V$ is semisimple. In particular, if $\dim V < p\ell_H$, then $\mathrm{H}^1(H, V) = \Ext_H^1(k, V) = 0$.
\end{theorem}

\begin{cor}\label{corollary:centralizer-of-weakly-reductive}
    Let $S$ be a scheme, and let $G$ and $H$ be geometrically reductive smooth affine $S$-group schemes. Suppose that $H$ is weakly reductive and that for every $s \in S$, either $\chara k(s) = 0$ or $\chara k(s) > \dim \sD(G^0_s)/\ell_{H_s}$. If $H$ acts faithfully on $G$, then $C_G(H)$ is smooth, affine, and geometrically reductive. If $G$ is weakly reductive, then $C_G(H)$ is also weakly reductive.
\end{cor}

Note that if $H^0$ is a torus (in particular if $H$ is finite \'{e}tale), then Corollary~\ref{corollary:centralizer-of-weakly-reductive} involves no hypothesis on the residue characteristics of $S$.

\begin{proof}
    By passing separately from $H$ to $\sD(H^0)$, $H^0/\sD(H^0)$, and $H/H^0$ (and from $G$ to centralizers of these), we can assume that $H$ is either semisimple, a torus, or finite \'etale of order invertible on $S$. Assume first that $H$ is either a torus or finite \'etale. In this case, let $G_1 = G \ltimes H$, a geometrically reductive smooth affine $S$-group scheme which is weakly reductive whenever $G$ is. By Theorem~\ref{theorem:main-centralizer-theorem} (whose hypotheses always hold in the current setting), the centralizer $C_{G_1}(H)$ is smooth, affine, and geometrically reductive, and it is weakly reductive whenever $G$ is. There is a natural projection map $f: C_{G_1}(H) \to H$, and $C_G(H) = f^{-1}(1)$. If $H$ is a torus then in particular it is commutative, and so $f$ is split by the natural inclusion $H \to G_1$.  Thus $C_{G_1}(H) = C_G(H) \times H$ and $C_G(H)$ is geometrically reductive, smooth, and affine. Moreover, $C_{G^0}(H)$ has connected fibers by the classical theory, so $C_G(H)$ is weakly reductive if $G$ is weakly reductive. Now assume that $H$ is finite \'etale of order invertible on $S$. Note $f$ factors through a constant map $C_{G_1}(H)/C_{G_1}(H)^0 \to H$, so $C_G(H)$ is an open and closed $S$-subgroup scheme of $C_{G_1}(H)$, from which the result follows in this case. If $G$ is weakly reductive, then to show that $C_G(H)$ is weakly reductive it suffices to show that $C_{G^0}(H)$ is weakly reductive, which follows from the above and \cite[Proposition VIII 5.11]{Fargues-Scholze} applied on $S$-fibers.

    Now assume that $H$ is semisimple. By Lemma~\ref{lemma:cr-criterion}, Theorem~\ref{theorem:jantzen-semisimple}, \cite[Corollary 3.42]{BMR05}, and our bounds on the residue characteristics, $H_s$ is $G_s$-cr for all $s \in S$. The action of $H$ on $G^0$ is given by a map $f: H \to G^0/Z(G^0)$, and Theorem~\ref{theorem:main-centralizer-theorem} combines with Theorem~\ref{theorem:jantzen-semisimple} to show that $C_{G^0/Z(G^0)}(H)$ is geometrically reductive, smooth, and affine.
    The assumption on residue characteristics implies (by considering the simple types) that $Z(G^0)$ is smooth, so as $H$ acts faithfully we conclude that $C_{G^0}(H)$ is geometrically reductive, smooth, and affine. If $X$ is the stabilizer of $f$ in $G$, then $C_{G^0}(H) = X \cap G^0$ is an open $S$-subgroup scheme of $X$, so $X$ is smooth affine, and we claim that $X$ is geometrically reductive. For this, we may and do assume $S = \Spec A$ for a DVR $A$. Now the quotient $G/X$ is affine: first, $G^0/C_{G^0}(H)$ is isomorphic to an affine open subscheme $U$ of $\underline{\Hom}_{S\textrm{-}\rm{gp}}(H, G^0/Z(G^0))$ which is also a closed subscheme on $S$-fibers, as follows from the proof of Theorem~\ref{theorem:main-centralizer-theorem}. The quotient $G/X$ is a finite union of $G/G^0$-translates of $U$ in $\underline{\Hom}_{S\textrm{-}\rm{gp}}(H, G^0/Z(G^0))$, so it is an open subscheme which is closed on fibers. By Theorem~\ref{theorem:hom-scheme} and Lemma~\ref{lemma:removing-special-fiber-components}, it follows that $G/X$ is affine and thus $X$ is geometrically reductive by Theorem~\ref{theorem:alper}. We may now pass from $G$ to $X$ to assume that $H$ acts trivially on $G^0$.

    Since $H$ has connected fibers, it also acts trivially on the finite \'etale $S$-group scheme $G/G^0$. It would now be enough to show that $H$ acts trivially on $G$. Thus there is a natural map $\vp: H \times G \to G^0$ given by $\vp(h, g) = (hgh^{-1})g^{-1}$. Since $H$ acts trivially on $G^0$, it follows that $\vp$ factors through a map $H \times G/G^0 \to Z(G^0)$. For fixed $g \in (G/G^0)(S)$, the map $\vp(-, g): H \to Z(G^0)$ is an $S$-homomorphism. Since $Z(G^0)$ is of multiplicative type and $H$ is semisimple, $\vp(-, g)$ is trivial, as desired. Finally, if $G$ is weakly reductive then to show that $C_G(H)$ is weakly reductive it is easy to pass from $G$ to $G^0/Z(G^0)$ and thus reduce to Theorem~\ref{theorem:main-centralizer-theorem}.
\end{proof}

Next we obtain a similar result under better bounds when $H$ is the centralizer of a finite group of order invertible on the base. For this, we will use recent fundamental results of Fargues--Scholze. First, we must recall the notion of good filtration. The following theorem is a basic consequence of \cite[Proposition VIII.5.12]{Fargues-Scholze}.

\begin{theorem}\cite[Proposition VIII.5.12]{Fargues-Scholze}\label{theorem:fargues-scholze-vanishing}
    Let $k$ be an algebraically closed field of characteristic $p > 0$, and let $G$ be a connected reductive $k$-group such that $p$ is pretty good for $G$. If $\Lambda$ is a finite group of order prime to $p$ acting on $G$, then $\mathrm{H}^i(C_G(\Lambda), \fg) = 0$ for all $i > 0$.
\end{theorem}

\begin{proof}
    By \cite[4.4]{Andersen-Jantzen}, every symmetric power $\Sym^n \fg^*$ admits a good filtration: in other words, there is a filtration of $\Sym^n \fg^*$ such that every subquotient is isomorphic to $\mathrm{H}^0(\lambda)$ for some dominant weight $\lambda$. Since $p$ is pretty good for $G$, \cite[Theorem 5.2]{herpel} shows that $\fg \cong \fg^*$ as $G$-representations and thus in particular $\fg$ admits a good filtration. Now if $H = C_G(\Lambda)$ then the above discussion and \cite[Proposition VIII.5.12]{Fargues-Scholze} show that $\cO(G/H)$ also admits a good filtration. Using the equality $\mathrm{H}^i(H, \fg) = \mathrm{H}^i(G, \fg \otimes_k \cO(G/H))$ (which relies on the fact that $G/H$ is affine), it is enough to show that $\mathrm{H}^i(G, \mathrm{H}^0(\lambda) \otimes_k \mathrm{H}^0(\mu)) = 0$ for all $i > 0$ and all dominant weights $\lambda$ and $\mu$. This is proved in \cite[II, 4.13]{Jantzen}.
\end{proof}

\begin{cor}\label{cor:double-centralizer}
Let $S$ be a scheme, and let $G$ be a weakly reductive $S$-group scheme. Let $\Lambda$ be a finite \'etale $S$-group scheme of order invertible on $S$ which acts on $G$. Then $C_G(\Lambda)$ is a weakly reductive $S$-group scheme. 
Moreover, suppose that for every $s \in S$, $\chara k(s)$ is pretty good for $G_s$. Then $C_G(C_G(\Lambda))$ is a weakly reductive $S$-group scheme.
\end{cor}

\begin{proof}
    The first claim follows immediately from Corollary \ref{corollary:centralizer-of-weakly-reductive}.
    Let $H = C_G(\Lambda)$. By Theorem~\ref{theorem:main-centralizer-theorem}, it is enough to show that $\mathrm{H}^1(H_s, \fg_s) = 0$ and $H_s$ is $G_s$-cr for all $s \in S$. The first condition holds by Theorem~\ref{theorem:fargues-scholze-vanishing}, and the second condition holds by \cite[Corollary 3.17]{BMR05}.
\end{proof}

\begin{example}
In general, some bound on the residue characteristics as in Corollary~\ref{corollary:centralizer-of-weakly-reductive} is necessary, even in pretty good characteristic. For example, let $p$ be any prime number, let $S = \Spec \bF_p[\![t]\!]$, let $H = \SL_2$, and let $G = \GL_{p+1}$. With notation as in \cite[II, Chapter 2]{Jantzen}, there is an indecomposable $p+1$-dimensional representation $\mathrm{H}^0(p)$ of $\SL_2$ equipped with a filtration
\[
0 \to \mathrm{L}(p) \to \mathrm{H}^0(p) \to \mathrm{L}(p-2) \to 0.
\]
This gives an extension class in $\Ext^1_H(\mathrm{L}(p-2), \mathrm{L}(p))$. Multiplying this class by $t$ gives a $p+1$-dimensional representation $V$ of $\SL_2$ over $S$ with special fiber $\mathrm{L}(p) \oplus \mathrm{L}(p-2)$ and generic fiber $\mathrm{H}^0(p)$. Relative to the corresponding map $H \to G$, we have $C_G(H)_s \cong \bG_m^2$, while $C_G(H)_\eta \cong \bG_m$ \cite[Proposition II.2.8]{Jantzen}, so $C_G(H)$ is not flat.

Moreover, the centralizer of a connected reductive subgroup of a reductive group need not be reductive if the characteristic is too small; this is related to the fact that representations of connected reductive groups can fail to be semisimple in positive characteristic. For explicit examples, see the MathOverflow answer \cite{McNinch-Mathoverflow}, as well as the comments on that answer.

We do not know the optimal bound on the residue characteristics in Corollary~\ref{corollary:centralizer-of-weakly-reductive}; if $G = \GL_n$, then one can improve the bound from $n^2/\ell_H$ to $n/\ell_H$. If $G = \Sp_n$ (resp.\ $\SO_n$), then one can improve the bound from $n(2n+1)/\ell_H$ (resp.\ $n(n-1)/2\ell_H$) to $n/\ell_H$. In general, it seems reasonable to expect that something like $h_G/\ell_H$ is the correct bound, where $h_G$ is the Coxeter number of $G$.
\end{example}

\subsection{Centralizers of Pure Unipotent Elements} \label{ss:centralizerunipotent}

Later arguments rely heavily on the smoothness of centralizers of pure fiberwise unipotent elements.  The notion of purity will capture the idea that an element of $G$ ``looks similar'' across all fibers.  For example, an element like $\ttm{1}{p}{0}{1} \in \GL_2(\Z_p)$ will not be pure.

\begin{defn} \label{defn:pure}
Let $G$ a reductive group over a scheme $S$.  
A section $g \in G(S)$ is \emph{pure} if $s \mapsto C_{G_s}(g_s)$ is locally constant on $s$.  There is an analogous definition for elements of the Lie algebra.  
\end{defn}

If $S$ is the spectrum of a discrete valuation ring, this means that the special and generic fibers of the centralizer of $g$ have the same dimension.

Using a Springer isomorphism, the smoothness for centralizers of unipotent elements is closely related to the smoothness for centralizers of nilpotent elements in the Lie algebra.  This smoothness was claimed in \cite{mcninch08}, but the argument there has a gap.  This has been fixed by Hardesty (for pure nilpotents) and by the second author (for pure unipotents and nilpotents) \cite{Hardesty,cotner}.  Since our later arguments are naturally phrased in terms of unipotent elements, we will build on the latter.

\begin{theorem} \label{theorem:unipotentcentralizer}
Let $A$ be a DVR with residue characteristic $p$, and $G$ be a weakly reductive group scheme over $\Spec A$.
 If $p$ is pretty good for $G$ and $u \in G(A)$ is a pure fiberwise unipotent element then $C_G(u)$ is $A$-smooth. 
\end{theorem}

Using Remark~\ref{rmk:pretty good}, it would be equivalent to suppose that $p$ is good for $G^0$, that $\# \pi_1 (\sD(G^{0}))$ is prime-to-$p$, and that $Z(G^0)$ is $A$-smooth. 

\begin{proof}
If $G$ is connected, flatness comes from {\cite[Theorem 1.1]{cotner}}.  The smoothness of the fibers follows for example from  \cite[Theorem 1.1]{herpel}.

In general, note that $\overline{u}$ has order a power of $p$ and hence $u$ lies in the identity component of $G$.
It suffices to show that every component of $G$ which contains a point in the special fiber centralizing $u$ has an $A$-point centralizing $u$.  For then the centralizer $C_G(u)$ is a union of copies of the smooth $C_{G^0}(u)$.

Given $g \in G(A)$ such that $\overline{g}$ centralizes $u_s$ in the special fiber,
as $u$ is pure \cite[Theorem 5.11]{cotner} shows that $gu g^{-1}$ and $u$ are $G^0(A)$-conjugate.  Thus there is $h \in G^0(A)$ such that $h u h^{-1} = g u g^{-1}$, and hence $h^{-1} g \in G(A)$ centralizes $u$ and lies in the desired component of $G$. 
\end{proof}

\subsection{Complements on Pure Unipotents}

We conclude with some additional results about pure unipotents which elaborate the sense in which a pure unipotent ``looks similar'' in the special and generic fibers. 

We begin by reviewing the Bala-Carter method which classifies nilpotent orbits for a connected reductive group over an algebraically closed field $k$ of good characteristic $p \geq 0$.  (More information can be found in \cite[\S 4]{jantzen04}, and a uniform proof without case-checking in small characteristic is due to Premet \cite{premet03}.)  Using a Springer isomorphism, this equivalently gives a classification of conjugacy classes of unipotent elements.  To state it, we need to define some terminology.

Let $H$ be a connected reductive $k$-group with $p$ good for $H$, and $\h = \Lie H$.

\begin{itemize}
 \item A nilpotent $N \in \h$ is a \emph{distinguished nilpotent} if each torus contained in $C_{H}(N)$ is contained in the center of $H$.
 
 \item  For a parabolic $P \subset H$ with unipotent radical $U$, the \emph{Richardson orbit} associated to $P$ is the unique nilpotent orbit of $H$ with dense intersection with $\Lie U$.  Its intersection with $\Lie P$ is a single orbit under $P$.

 \item  A parabolic subgroup $P \subset H$ with unipotent radical $U$ is a \emph{distinguished parabolic} if $\dim P/U = \dim U/ \mathscr{D}(U)$.
\end{itemize}

Bala and Carter classified nilpotent orbits when the characteristic is good.  One can check that if $p$ is good for $H$, it will also be good for any Levi factor of a parabolic subgroup of $H$. The following fact can be found in \cite[\S 4]{jantzen04}.

\begin{fact} \label{fact:balacarter}
If $p$ is a good prime for $H$, the nilpotent orbits for $H$ are in bijection with $H(k)$-conjugacy classes of pairs $(L,P)$ where $L$ is a Levi factor of a parabolic subgroup of $H$ and $P$ is a distinguished parabolic of $L$.  The nilpotent orbit for $H$ associated to $(L,P)$ is the unique one meeting $\Lie(P)$ in its Richardson orbit for $L$.
\end{fact}

The \emph{Bala-Carter data} for $H$ is the set of $H(k)$-conjugacy classes of pairs $(L,P)$ as above.  It turns out it is independent of $k$ in the sense that it can be described completely in terms of the root datum of $H$ as follows.  All Levi subgroups $L$ of a parabolic $k$-subgroup $Q$ of $H$ are a single $\mathcal{R}_{u,k}(Q)$-orbit, so in Fact~\ref{fact:balacarter} we may restrict to one $Q$ per $H(k)$-conjugacy class and one $L$ per $Q$.  We may pick $L$ so that it contains a (split) maximal torus $T$.  After conjugation by $L(k)$, the distinguished parabolic subgroup $P \subset L$ may be assumed to contain $T$ as well.  But we know that parabolic subgroups $Q$ of $H$ containing $T$ are in bijection with parabolic subsets of $\Phi(H,T)$ via $Q \mapsto \Phi(Q,T)$, so the possible Levi factors $L$ of $Q$ containing $T$ are described just in terms of the root datum.  Likewise, parabolic subgroups $P$ of $L$ containing $T$ are in bijection with parabolic subsets of $\Phi(L,T)$.  If we can 
characterize the condition 
that $P$ is distinguished just in terms of the root 
data, this would mean that the Bala-Carter data can be described solely in terms of the root data and so is completely combinatorial. 

We do so by constructing a grading on the Lie algebra of a parabolic $P$.  Pick a Borel subgroup $B \subset H$ satisfying $T \subset B \subset P$.  Let $\ft = \Lie(T)$ and $\Delta \subset \Phi = \Phi(L,T)$ be the set of positive simple roots determined by $B$.  There is a unique subset $I \subset \Delta$ such that $P = B W_I B$ where $W_I$ is the subset of the Weyl group generated by reflections with respect to roots in $I$.  Define a group homomorphism $f: \ZZ \Phi \subset \ZZ^\Delta \to \ZZ$ by specifying that on the basis $\Delta$ we have
\[
 f(\alpha) = \begin{cases}
              2 & \alpha \in \Delta - I, \\
              0 & \alpha \in I.
             \end{cases}
\]
This function gives a grading on $\fl = \Lie(L)$:
\[
\fl(i) = \bigoplus_{f(\alpha) = i} \fl_\alpha \quad \text{and} \quad \fl(0) = \left( \bigoplus_{f(\alpha) = 0} \fl_\alpha \right) \oplus \ft
\]
(sums indexed by $\alpha \in \Phi$).  With respect to this grading,
\[
 \Lie P = \bigoplus_{i \geq 0} \fl(i) \quad \text{and}\quad \Lie U = \bigoplus_{i >0 } \fl(i).
\]
The condition that $P$ is distinguished is equivalent to the condition that
\[
 \dim \fl(0) = \dim \fl(2) + \dim Z_{L}
\]
by \cite[Corollary 5.8.3]{carter85} as $p$ is good for $L$.  But this condition depends only on the root datum.  Thus the Bala-Carter data for $H$ can be described in a manner independent of the choice of algebraically closed field.  We call this combinatorial description the \emph{Bala-Carter label} for the nilpotent (or unipotent) orbit.

\begin{defn}
If $K$ is a field (not necessarily algebraically closed), then the Bala-Carter label label for a nilpotent (resp. unipotent) element of $\Lie H$ (resp. $H(K)$)  is the Bala-Carter label for the corresponding nilpotent (resp. unipotent) orbit over the algebraic closure.
\end{defn}

\begin{remark}
 From the classification of nilpotent orbits over algebraically closed fields, it is known that the corresponding nilpotent orbits in characteristic zero and characteristic $p$ have the same dimension.  Thus the dimension of centralizers of elements in an orbit is also  independent of the characteristic and depends only on the Bala-Carter label.   There is a similar statement for conjugacy classes of unipotent elements and their centralizers.
\end{remark}

We now return to the relative setting.

\begin{lem} \label{lemma:purelift}
Let $G$ be a connected reductive group scheme over a DVR $A$ with residue field $k$ whose characteristic is good for $G$.  Given a unipotent element $\overline{u} \in G(k)$ with Bala-Carter label $\sigma$, there exists a pure fiberwise unipotent element $u \in G(A)$ lifting $\overline{u}$ with Bala-Carter label $\sigma$ in the generic fiber, and a similar statement for nilpotent elements of the Lie algebra.  
\end{lem}

\begin{proof}
We will prove this for a nilpotent element $\overline{X}$; the unipotent case follows using an integral Springer isomorphism \cite[Theorem 1.1]{Integral-Springer}.  Let $\tau_s : (\Gm)_s \to G_s$ be a cocharacter associated to $\overline{X}$, which exists since $\textrm{char}(k)$ is good for $G$.  This lifts to a cocharacter $\tau : \Gm \to G$ by smoothness of the scheme of maximal tori \cite[Theorem 3.2.6]{conrad}.  
This cocharacter defines 
\[
\bigoplus_{n \geq 2} \g(\tau,n) \subset \g.
\]
Over the special fiber $\overline{X}$ is in $\g_s(\tau,2)$ and the $\Ad_{P(\tau_s)}$-orbit of $\overline{X}$ is open and dense.  Pick a fiberwise nilpotent $X \in \g(\tau,2)$ lifting $\overline{X}$ and consider the parabolic $P_G(\tau) \subset G$.  It naturally acts on $\bigoplus_{n \geq 2} \g(\tau,n) $ and the stabilizer of $X$ is $C_G(X)$.  As we know $\dim C_G(X)_\eta \leq \dim C_G(X)_s$ and since the orbit is open and dense in the special fiber it follows that $X$ is pure.  Then argue as in \cite[Lemma 5.2]{cotner} to show that $P_G(\tau)$ gives the instability parabolic for $X$ and $\tau$ is an associated cocharacter for $X$ in the special and generic fibers.  By \cite[Proposition 2.13(6)]{cotner} this information determines the Bala-Carter data in the special and generic fibers.  
\end{proof}

\begin{defn}
Let $K$ be a $p$-adic field and $G$ be a connected reductive group over the ring of integers $\cO_K$.  
We say that $g_1,g_2 \in G(\cO_K)$ are geometrically conjugate if there exists a finite extension $L$ of $K$ such that $g_1$ and $g_2$ are $G(\cO_L)$-conjugate.
\end{defn}

\begin{proposition} \label{prop:pureequivalent}
Let $K$ be a $p$-adic field with residue field $k$ and let $G$ be a connected reductive group over $\cO_K$ such that $p$ is pretty good for $G$. Fix a Bala-Carter label $\sigma$ and a pure unipotent $u_\sigma$ with Bala-Carter label $\sigma$ in $G(K)$ and $G(k)$ using Lemma~\ref{lemma:purelift}.
The following are equivalent for a fiberwise unipotent $u \in G(\cO_K)$:
\begin{enumerate}
    \item  \label{pureequivalent1} $u$ is 
    geometrically conjugate to $u_{\sigma}$; 
    \item \label{pureequivalent2} the images of $u$ in $G(K)$ and $G(k)$ have Bala-Carter label $\sigma$;
    \item \label{pureequivalent3}  $u$ is pure and it has Bala-Carter label $\sigma$ in $G(K)$;
    \item  \label{pureequivalent4}  $u$ is pure and it has Bala-Carter label $\sigma$ in $G(k)$;
\end{enumerate}
\end{proposition}

\begin{proof}
Note that
(\ref{pureequivalent1}) implies (\ref{pureequivalent2}) as the image of $u_{\sigma}$ in $G(K)$ and $G(k)$ have the same Bala-Carter labels.  Since the dimensions of the centralizer of a unipotent element over an algebraically closed field depends only on the Bala-Carter label and not the characteristic, (\ref{pureequivalent2}) implies (\ref{pureequivalent3}) and (\ref{pureequivalent4}).  If $u$ is pure and generically has Bala-Carter label $\sigma$, then $u$ and $u_\sigma$ are $G(K')$-conjugate for some extension $K'$ of $K$.  By
\cite[Theorem 5.11]{cotner}, $u$ and $u_\sigma$ are $G(\cO_{K'})$-conjugate.  
Thus (\ref{pureequivalent3}) implies (\ref{pureequivalent1}).  Similarly, (\ref{pureequivalent4}) implies (\ref{pureequivalent1}).
\end{proof}

\begin{remark}
\begin{itemize}
    \item   When defining the minimally ramified deformation condition in \cite{Booher19} and this paper, a key step is to restrict to deformations where a particular inertial element $\sigma$, whose mod-$p$ image is unipotent, is sent to a conjugate of a particular pure unipotent lift.  Proposition~\ref{prop:pureequivalent} shows that this condition is equivalent to controlling the conjugacy class of the lift in the generic fiber or to enforcing purity of the lift.

    \item  There is of course an analogous statement for nilpotents.  These can be cheaply deduced using an integral Springer isomorphism \cite[Theorem 1.1]{Integral-Springer} or proven directly using similar techniques.
    
\end{itemize}
\end{remark}

We will also need the following lemma later.

\begin{lemma} \label{lem:conjugacypower}
Let $K$ be a $p$-adic field and $G$ a connected reductive group defined over $\cO_K$ such that $p$ is good for $G$.   Suppose $u$ is a fiberwise unipotent element of $G(\cO_K)$ satisfying any of the equivalent conditions of Proposition~\ref{prop:pureequivalent}.  Then for any integer $n$ relatively prime to $n$, the elements $u$ and $u^n$ generically have the same Bala-Carter label and $u^n$ is pure.
\end{lemma}

\begin{proof}
Over an algebraically closed field, if $X$ is nilpotent then all non-zero multiples of $X$ lie in the same nilpotent orbit \cite[2.10 Lemma]{jantzen04}.  Thus the same is true for powers of unipotents.  In particular $u^n$ is therefore pure and $u^n$ has the same Bala-Carter labels as $u$ generically and in the special fiber.
\end{proof}

\section{Decomposition Types} \label{sec:decomp types}
Throughout this section, let $G$ be a weakly reductive group scheme over a DVR $\cO$ with residue field $k$ of characteristic $p$. 

\subsection{Definitions and Examples} \label{ss:decomposition defs}

\begin{defn} \label{def:decomp type}
A \emph{decomposition type} for $G$ over $\cO$ is a pair $(\fC, \Dt)$ where $\fC$ and $\Dt$ are closed subgroup schemes of $G$ defined over $\cO$ such that 
\begin{itemize}
    \item $\fC$ and $\Dt$ are weakly reductive, \textit{i.e.} $\cO$-smooth with (not necessarily connected) reductive fibers and finite \'{e}tale component groups of order invertible on $\cO$.
    \item $\Dt$ (resp. $\fC$) represents the scheme theoretic centralizer of $\fC$ (resp. $\Dt$).
\end{itemize}
\end{defn}

\begin{defn} \label{def:adapted}
Let $\Lbar \subset G(k)$ be a finite subgroup with order prime to $p$. We say a decomposition type $(\fC, \Dt)$ over $\cO$ is \emph{adapted to} $\Lbar$ if 
$\cent{G_k}{\Lbar}=\fC_k$ (hence also $\cent{G_k}{\cent{G_k}{\Lbar}}=\Dt_k$).
\end{defn}

This definition is motivated by the goal of giving a group theoretic reformulation of the isotypic decomposition of a representation, which we now explain.

\begin{example} \label{example:gln}
This is an elaboration of the example of $\GL_n$ discussed in the introduction.  Let $M$ be a free $\cO$-module of rank $n$, $G = \Aut(M)$, and $V = M_k$.   As $\Lbar$ has pro-order prime to $p$, $V$ decomposes as a direct sum of irreducible $\Lbar$-representations
\[
V=V_1^{\oplus m_1} \oplus \cdots \oplus V_r^{\oplus m_r} 
\]
where $V_i$ and $V_j$ are non-isomorphic for $i \neq j$, and the $m_i$'s are multiplicities of the irreducibles.  Letting $W_i = \Hom_{\Lbar}(V_i,V)$ be the weight space associated to $V_i$, we have that $\dim W_i = m_i$ and
\begin{equation} \label{eq:isotypicdecomp}
V = (V_1 \tensor{k} W_1) \oplus \ldots \oplus (V_r \tensor{k} W_r).
\end{equation}
Enlarging $k$ if necessary, we may and do assume that $V_i$ is absolutely irreducible for all $i$. 
Let 
\[
L_k \colonequals \GL(V_1)^{m_1} \times \cdots \times \GL(V_r)^{m_r} \subset \GL(V)
\]
be the standard Levi associated to the decomposition of $\taubar$, and let $d_i=\dim_k V_i$.
By Schur's lemma, 
\[
\fC_k \colonequals C_{\GL(V)} (\Lbar) = \Aut_\Lambda(V) = \prod_{i=1}^r \GL(W_i) \subset \GL(V).
\]
A direct computation then shows that
\[
\Delta_k \colonequals C_{\GL(V)}(\fC_k) = \prod_{i=1}^r \GL(V_i) \subset \GL(V)
\]
where the embedding into $\GL(V)$ comes from acting on the first factors of the tensor products in \eqref{eq:isotypicdecomp}.  Note that $\Delta_k \subset L_k$, with $\GL(V_i)$ embedding in $\GL(V_i)^{m_i}$ diagonally.  One checks that $(\fC_k,\Delta_k)$ is a decomposition type adapted to $\Lbar$.

A decomposition type for $\GL_n$ over $\cO$ adapted to $\Lbar$ is just a version of this decomposition for the $\cO$-module $M$ phrased group-theoretically.  We can specify a representation (of a group containing $\Lbar$) on $M$ compatible with the isotypic decomposition using the groups $\fC$ and $\Delta$, which are the automorphisms of the weight spaces and irreducible constituents.  
\end{example}

We also give some examples for the symplectic groups that are reminiscent of the ``isotypic decomposition with pairings'' in \cite[\S6.1]{Booher19}. Fix the symmetric form so that 
$$\Sp_{2n}= \left\{M \in \GL_{2n} :  M \begin{pmatrix}
0 & I_n \\
-I_n & 0
\end{pmatrix} M^t=\begin{pmatrix}
0 & I_n \\
-I_n & 0
\end{pmatrix} \right \}.$$
For $A \in \GL_n$, we write $A^*$ for $(A^t)^{-1}$.  Notice that $L=\left\{ 
\begin{pmatrix}
A & 0 \\
0 & A^*
\end{pmatrix}: A \in \GL_n
\right \}$ is a Levi subgroup of $\Sp_{2n}$. Let $\Lbar \subset L(k)$ be an $L$-irreducible subgroup. This gives a natural $2n$-dimensional representation of the form $\tau \oplus \tau^*$ where $\tau$ is irreducible representation of $\Lbar$ with dimension $n$. 

\begin{example} \label{ex:sympl 1}
If $\tau$ and $\tau^*$ are non-isomorphic $\Lbar$-representations, we have 
\[
\fC_k=\left\{ 
\begin{pmatrix}
aI_n & 0 \\
0 & a^{-1}I_n
\end{pmatrix}: a \in \Gm
\right \}=Z(L)_k, \quad \Dt_k=L_k, 
\]
\end{example}

\begin{example} \label{ex:sympl 2}
Suppose $\tau \cong \tau^*$. In this case, it is easy to see that there exists $J \in 
\GL_n(k)$, either symmetric or skew-symmetric, such that $\tau J \tau^t=J$. Let us suppose that $p \neq 2$ and $J$ is skew-symmetric (the other case is similar but occurs inside an orthogonal group). Conjugating $\tau$ if necessary, we may assume that $J=\begin{pmatrix}
0 & I_{n/2} \\
-I_{n/2} & 0
\end{pmatrix}$, so $\tau$ takes values in $\Sp_n(k)$. We then have
\[
\fC_k=\left\{ 
\begin{pmatrix}
aI_n & 0 \\
0 & a^{-1}I_n
\end{pmatrix} : a \in \Gm
\right\}
\cup \begin{pmatrix}
0 & J \\
-J & 0
\end{pmatrix}
\left\{ 
\begin{pmatrix}
aI_n & 0 \\
0 & a^{-1}I_n
\end{pmatrix}
: a \in \Gm
\right \}, 
\]

\[
\Dt_k=\left\{\begin{pmatrix}
A & 0 \\
0 & A^*
\end{pmatrix}: A \in \Sp_n
\right \}.
\]
\end{example}

\begin{defn} \label{def:goodprime}
Let $(\fC,\Dt)$ be a decomposition type over $\cO$. We say $p$ is good for $(\fC,\Dt)$ if 
\begin{enumerate}
    \item $p$ is pretty good for $G$.
    \item $Z(\sD(\Dt^{0}))_k$ is smooth. 
    \item $p$ does not divide $\# \nlz{G_k}{\Dt_k}/\fC_k \Dt_k$. 
\end{enumerate}
\end{defn}

\subsection{Basic Properties}

As the group schemes constituting a decomposition type are weakly reductive, we begin by recalling what we know about their centers and abelianizations.

\begin{prop} \label{prop:centers and abelianization}
Let $H$ be a weakly reductive group scheme over $\cO$. 
\begin{enumerate}
    \item The center $Z(H)$ and abelianization $H^{\ab}$ both exist as group schemes of multiplicative type over $\cO$. 
    \item $H^{\ab}$ is smooth.
    \item \label{prop:centers and abelianization3} If $Z(H)$ is smooth, then the natural map $Z(H)^0 \to H^{\ab,0}$ is an isogeny of tori. If $Z(\sD(H^{0}))$ is smooth, then this map is smooth.
\end{enumerate}
\end{prop}

\begin{proof}
Combine 
Proposition~\ref{prop:center}, Proposition~\ref{prop:abelianization}, and Proposition~\ref{proposition:isogenies}.
\end{proof}

\begin{lem} \label{lem:martin}
Let $(\fC,\Dt)$ be a decomposition type over $k$. 
Then the quotient group scheme
$\nlz{G}{\fC}/\fC \Dt = \nlz{G}{\Dt}/\fC \Dt$ is finite \'{e}tale.
\end{lem}

\begin{proof}
Note that the quotient is well-defined since $\cent{G}{\Dt}=\fC$ and $\cent{G}{\fC}=\Dt$. 
Since $\fC$ and $\Dt$ are weakly reductive, the finiteness follows from Corollary \ref{cor:finiteness of weyl}. 
\end{proof}

\begin{lem} \label{lem:symmetry}
Let $(\fC, \Dt)$ be a decomposition type for $G$ over $\cO$. Then for any $\cO$-algebra $R$, $\nlz{G(R)}{\fC(R)}=\nlz{G(R)}{\Dt(R)}$.
Moreover, 
$Z(\fC)=Z(\Dt)$.
\end{lem}
\begin{proof}
This is basic group theory using Definition~\ref{def:decomp type}. For $\cO$-algebras $R$, if $g \in G(R)$ normalizes $\fC(R)$, then it normalizes $\Dt(R)=\cent{G(R)}{\fC(R)}$, so $\nlz{G(R)}{\fC(R)} \subset \nlz{G(R)}{\Dt(R)}$, and hence $\nlz{G(R)}{\fC(R)}=\nlz{G(R)}{\Dt(R)}$ by symmetry. Note that for $\cO$-algebras $R$, $Z(\Dt)(R)=\Dt(R) \cap \fC(R)=Z(\fC)(R)$ by Definition~\ref{def:decomp type}. 
\end{proof}

\begin{lemma} \label{lem:adapted}
If $(\fC,\Delta)$ is a decomposition type adapted to $\Lbar$, and $\Ld$ is as in Proposition \ref{prop:adapted complete dvr},
then
$\cent{\Dt}{\Lambda}=Z(\Dt)$. 
\end{lemma}

\begin{proof}
This is basic group theory using Definition~\ref{def:decomp type}: for $\cO$-algebras $R$ we see that $\cent{\Dt(R)}{\Lambda}=\cent{G(R)}{\Lambda} \cap \Dt(R)=\fC(R) \cap \Dt(R) \subset Z(\Dt(R)) = Z(\Dt)(R)$.
\end{proof}

We also record a few useful properties about $C_G(\Lambda)$ for later use, where $\Lambda$ is a constant group scheme whose order is invertible in $\cO$ (equivalently, the order of $\Lambda$ is prime to $p$).  We know it is $\cO$-smooth and weakly reductive by Corollary~\ref{cor:double-centralizer}.

\begin{lemma} \label{lem:cglambda pretty good}
If $p$ is \textit{pretty good} for $G$, then $p$ is \textit{pretty good} for $\cent{G}{\Ld}$.
\end{lemma}

\begin{proof}
It suffices to work over a field of characteristic $p$.  
By \cite[Theorem 3.3]{herpel}, $p$ is pretty good if and only if the centralizer of every closed subgroup is smooth.  Note $(G, C_G(\Lambda))$ is a reductive pair in the sense of \cite[2.7]{herpel} since the $\Ld$-isotypic decomposition of $\Lie G$ is stable under $C_G(\Lambda)$.  Using \cite[Lemma 3.6]{herpel}, if $H$ is a closed subgroup of $C_G(\Lambda)$ such that $C_G(H)$ is smooth then $C_{C_G(\Lambda)}(H)$ is smooth as well. Since $p$ is pretty good for $G$, $C_G(H)$ is always smooth, and therefore $C_{C_G(\Lambda)}(H)$ is always smooth.
\end{proof}

\begin{prop} \label{prop:adapted complete dvr}
Let $(\fC,\Dt)$ and $\Lbar$ be as in Definition \ref{def:adapted}. Suppose that $\cO$ is a complete DVR. Then there exists a subgroup $\Ld \subset G(\cO)$ reducing to $\overline{\Ld}$ such that
$\cent{G}{\Ld}=\fC$ (hence also $\cent{G}{\cent{G}{\Ld}}=\Dt$). 
\end{prop}

\begin{proof}
Note that $\Lbar \subset \Dt(k)$. Since $\Dt$ is $\cO$-smooth and $\cO$ is complete, there is a subgroup $\Ld \subset \Dt(\cO)$ lifting $\Lbar$. We have $\cent{G}{\Ld} \supset \cent{G}{\Dt} = \fC$, which is an equality on special fibers. Since $\fC$ and $C_G(\Lambda)$ are smooth $\cO$-group schemes of the same dimension, it follows that $\fC^0 = C_G(\Lambda)^0$: by the fibral isomorphism criterion, it is enough to show that if $k$ is an algebraically closed field and $H \subset H'$ is a closed embedding of reduced connected $k$-schemes of the same dimension, then $H = H'$. For this, note that the natural map $H \to H'$ is dominant for dimension reasons, so it is surjective by closedness. Being a surjective closed embedding with reduced target, we find $H = H'$.  

In view of the previous paragraph, there is a monomorphism $\pi_0 \fC \into \pi_0 \cent{G}{\Ld}$. Since $\pi_0 \fC_{k} = \pi_0 \cent{G}{\Ld}_{k}$ and both $\pi_0\fC$ and $\pi_0\cent{G}{\Ld}$ are finite \'{e}tale (by Corollary~\ref{cor:double-centralizer}), it follows that $\pi_0 \fC = \pi_0 \cent{G}{\Ld}$ and hence $\fC = \cent{G}{\Ld}$ by the five lemma.
\end{proof}

\begin{prop} \label{prop:smoothness of delta etc}
Let $(\fC,\Dt)$ be a decomposition type adapted to $\Lbar$. 
Suppose that $p$ is \emph{pretty good} for $G$ and $Z(\sD(\Dt^{0}))_k$ is smooth. 
Then $Z(\Dt)$ is smooth over $\cO$ and the natural map $Z(\Dt)^0 \to \Dt^{\ab,0}$ is a \emph{smooth} isogeny of $\cO$-tori. 
\end{prop}

\begin{proof}
By hypothesis and Proposition \ref{prop:adapted complete dvr} there is a group $\Ld \subset G(\cO)$ such that $\cent{G}{\Ld}=\fC$. 
To show the smoothness of $Z(\Dt)$, first note that $Z(\Dt)= Z(\fC)=Z(\cent{G}{\Ld})$ by Lemma \ref{lem:symmetry}. By Proposition~\ref{prop:center}, the smoothness of $Z(\Dt)$ would then follow from the smoothness of $Z(\cent{G}{\Ld}^0)$.  As $Z(\Dt)$ is of multiplicative type and hence $\cO$-flat, it suffices to check smoothness on fibers.  But since $p$ is pretty good for $G$, Lemma \ref{lem:cglambda pretty good} and Remark \ref{rmk:pretty good} imply that $Z(\cent{G}{\Ld}^0)$ is smooth. Thus $Z(\Dt)$ is smooth.  By Proposition \ref{prop:centers and abelianization}(\ref{prop:centers and abelianization3}) it follows that the natural map $Z(\Dt)^0 \to \Dt^{\ab,0}$ is an isogeny of $\cO$-tori.  But $Z(\sD(\Dt^{0}))_k$ is smooth by assumption, so the isogeny is smooth by \textit{loc. cit.} 
\end{proof}

\begin{prop} \label{prop:existence of decomp type}
Let $\Lbar$ be as in Definition \ref{def:adapted}.
Suppose that $p$ is pretty good for $G$. Then there exists a decomposition type $(\fC,\Delta)$ over $\cO$ adapted to $\Lbar$. 
\end{prop}

\begin{proof}
Fix a lift $\Ld \subset G(\cO)$ of $\Lbar$ such that the natural map $\Ld \to \Lbar$ is an isomorphism. 
Let $\fC=\cent{G}{\Ld}$ and let $\Dt=\cent{G}{\cent{G}{\Ld}}$. By Corollary \ref{cor:double-centralizer}, $(\fC,\Delta)$ is a decomposition type over $\cO$ adapted to $\Lbar$. 
\end{proof}

\section{Clifford theory} \label{sec:clifford}

This section proves a result about lifting and extending representations that is the analog of Part 2 of the argument for $\GL_n$ sketched in Section~\ref{ss:weakly-reductive}.

\subsection{Local Galois Groups} \label{ss:galois}

Fix a local field $F$ with residue characteristic $\ell \neq p$.
Let the residue field of $F$ have size $q$, a power of $\ell$.

\begin{lemma} \label{lem:tamestructure}
The maximal tame extension of $F$ has inertia group isomorphic to $\prod_{\ell' \neq \ell} \ZZ_{\ell}$.
The Galois group of the maximal tame extension is a semi-direct product of the inertia group with $\Zhat$, with $\Zhat$ acting via the cyclotomic character.  In particular, conjugation by the Frobenius is multiplication by $q$ on $\prod_{\ell' \neq \ell} \ZZ_{\ell}$.
\end{lemma}

\begin{proof}
This is standard.  
\end{proof}

Let $\Gamma_F$ be the absolute Galois group of $F$ and let $I_F$ be the inertia group.  Using Lemma~\ref{lem:tamestructure}, we may fix a surjection $I_F \to \ZZ_p$ with kernel $\Lambda_F$.

\begin{lem} \label{lem:tq}
The group $\Lambda_F$ is normal in $\Gamma_F$.  The quotient $\Gamma_F / \Lambda_F$ is isomorphic to $T_q \colonequals \Zhat \ltimes \Z_p$, where conjugating by $1 \in \Zhat$ is multiplication by $q$ on $\Zp$.
\end{lem}

\begin{proof}
This is also well-known.
\end{proof}

\begin{lemma} \label{lem:semidirect prod}
The exact sequence
\[
1 \to \Ld_F \to \Gamma_F \to \Gamma_F / \Lambda_F \to 1
\]
is topologically split, so $\Gamma_F$ is a semi-direct product.
\end{lemma}

\begin{proof}
This is basically \cite[2.4.10]{cht08}. 
\end{proof}

For the rest of this paper, fix a preimage $\sigma$ of a topological generator of $\Z_p$ under the surjection $I_F \to \Z_p$ and a Frobenius $\phi$ satisfying  
\[ \phi \sigma \phi^{-1}=\sigma^q. \]

\subsection{\texorpdfstring{$\nu$}{v}-Tame Extensions} \label{ss:tameextension}

Let $\cO$ be the ring of integers in a $p$-adic field with residue field $k$ and $\Dt$ be as in Definition \ref{def:decomp type}. Recall Proposition \ref{prop:centers and abelianization}. We write $Z$ for $Z(\Dt)$ and write $S$ for $\Dt^{\ab}$, which are both multiplicative groups. The natural morphism $\nu: \Dt \to S$ restricts to an isogeny $Z \to S$ (also denoted by $\nu$). 
Let $R$ be a complete local noetherian $\cO$-algebra with residue field $k$.

\begin{defn} \label{def:nutame}
We say a homomorphism $\rho: I_F \to \Dt(R)$ is \emph{$\nu$-tame} if for any $\sigma \in I_F$, $(\nu \circ \rho)(\sigma) \in S(R)$ is of finite prime-to-$p$ order. 
\end{defn}

We will prove the following:

\begin{prop} \label{prop:extension}
Given $\Delta$ as above. Assume that $\nu \colon Z \to S$ is smooth.   
Given a continuous homomorphism $\tau: \Ld_F \to \Dt(R)$ such that: 
\begin{enumerate}
    \item \label{prop:extensioni} $\Ker\tau=\Ker\taubar$ (in particular, the image of $\tau$ is finite);
    \item \label{prop:extensionii} $\taubar^{\sigma} \cong \taubar$ for any $\sigma \in I_F$ (i.e. the representations are conjugate by an element in $\Dt(k)$); 
    \item \label{prop:extensioniii} $\cent{\Dt(R)}{\tau(\Ld_F)}=Z(R)$;
\end{enumerate}
then $\tau$ admits a unique continuous, $\nu$-tame extension to $I_F$.
\end{prop}

We first establish some useful lemmas.  Recall that for a group scheme $H$ over $\cO$, $\widehat{H}(R)$ is the kernel of the reduction map $H(R) \to H(k)$.

\begin{lem}\label{isomorphism}
For any $\sigma \in I_F$, $\tau^{\sigma} \cong \tau$ (i.e. they are conjugate by an element in $\Dt(R)$).
\end{lem}
\begin{proof}
Note that $\mathrm{H}^1(\Ld_F, \ad \taubar)=0$ since $\Ld_F$ has pro-order prime to $p$ and $\ad \taubar$ has order a power of $p$.  Therefore if $\tau': \Ld_F \to \Dt(R)$ is another continuous lift of $\taubar$ then $\tau'$ is $\widehat{\Dt}(R)$-conjugate to $\tau$.  (See \cite[\S3]{tilouine96} for deformation theory in this level of generality.)
By assumption for $\sigma \in I_F$ we have  $\taubar^{\sigma} \cong \taubar$, i.e. $\taubar^{\sigma} = \bar g \taubar \bar g^{-1}$ for some $\bar g \in \Dt(k)$.   As $\Dt$ is $\cO$-smooth, there exists a lift $g \in \Dt(R)$ of $\bar g$. Now both $\tau^{\sigma}$ and $g \tau g^{-1}$ reduce to $\taubar^{\sigma}$ and hence they are $\widehat \Dt(R)$-conjugate.  Thus $\tau^\sigma \cong \tau$.
\end{proof}

\begin{lem} \label{lem:identityhat}
If $H$ is a smooth group scheme over $\cO$, 
then $\widehat{H}(R)=\widehat{H^0}(R)$ for any $\cO$-algebra $R$.
\end{lem}

\begin{proof}
If an $R$-point reduces to the identity, it lies in the identity component of $H$. 
\end{proof}

\begin{lem}\label{lem:decomposition of S}
The group $S(R)$ is the product of $\widehat{S}(R)$ and a finite group whose order is prime to $p$.  
\end{lem}

\begin{proof}
Recall that $S$ is smooth multiplicative by Proposition \ref{prop:centers and abelianization}. In particular, $S(k)$ has order prime to $p$.
Consider the exact sequence of groups
\[ 1 \to \widehat S(R) \to S(R) \to S(k) \to 1. \] 
Lemma \ref{lem:identityhat} implies that $\widehat S(R)=\widehat{S^0}(R)$, the latter is pro-$p$. 
Then since $S$ is commutative, $S(R)$ is the product of the pro-$p$ group $\widehat{S^0}(R)$ and the Teichmuller lifts of elements in $S(k)$.
\end{proof}

\begin{lem}\label{lem:isogeny of tori}
The map $\nu \colon \widehat{Z}(R) \to \widehat{S}(R)$ is an isomorphism.
\end{lem}

\begin{proof}
By the completeness of $R$, it suffices to prove the following: suppose $A$ is a local Artinian $\cO$-algebra and $I \subset A$ is a nilpotent ideal, let $z \in Z(A/I)$ and $s \in S(A/I)$ such that $\nu(z)=s$, then for any $\tilde s \in S(A)$ mapping to $s \in S(A/I)$, there exists a unique element $\tilde z \in Z(A)$ mapping to $z \in Z(A/I)$ such that $\nu(\tilde z)=\tilde s$. By assumption $\nu \colon Z \to S$ is a smooth isogeny of multiplicative groups, in particular it is \'etale, so the above follows immediately from the infinitesimal criterion for \'etale morphisms. 
\end{proof}

\begin{proof}[Proof of Proposition \ref{prop:extension}]
A continuous extension of $\tau$ to $I_F$ is determined by its value on $\sigma$, a chosen topological generator of the $\Zp$-part of the tame inertia. By Lemma \ref{isomorphism}, there is an $A \in \Dt(R)$ such that for $g \in \Ld_F$ 
\[ \tau^\sigma(g) = \tau(\sigma g\sigma^{-1})=A\tau(g)A^{-1}. \]
As $\taubar$ is continuous it factors through a finite quotient, so there is a power $p^b$ such that $\taubar^{\sigma^{p^b}}=\taubar$ as $\sigma$ has pro-$p$ order.

Since $\Ker\tau=\ker\taubar$, we see that for all $g \in \Lambda_F$
\[
A^{p^b}\tau(g) A^{-p^b}=\tau^{\sigma^{p^b}}(g)=\tau(g).
\] 
It follows that $A^{p^b} \in \cent{\Dt(R)}{\tau(\Ld_F)}$, which equals $Z(R)$ by assumption. 
As $k$ is perfect and $Z$ is multiplicative,
the $p$-th power map is an automorphism of $Z(k)$.  
We can then modify $A$ by a lift of an appropriate element of $Z(k)$ 
so that $A^{p^b}$ reduces to the identity.  
By Lemma \ref{lem:decomposition of S} and (the surjectivity part of) Lemma \ref{lem:isogeny of tori} we can further multiply $A$ by an element in $\widehat{Z}(R)$ so that $\nu(A)$ has finite prime-to-$p$ order.
We can now obtain a continuous, $\nu$-tame extension $\tau : I_F \to \Delta(R)$ by sending $\sigma$ to $A$.

It remains to show this extension is unique.  Suppose another (continuous and $\nu$-tame) extension sends $\sigma$ to $B \in \Dt(R)$.  Note that $BA^{-1}$ commutes with $\tau(\Lambda_F)$ so $z\colonequals BA^{-1} \in Z(R)$. By continuity, there is a power $p^b$ such that $A^{p^b}$ and $B^{p^b}$ reduce to the identity
in $\Delta(k)$,
and hence $z^{p^b}$ does as well.  Hence $z$ reduces to the identity as the $p$-th power map is an automorphism of $Z(k)$. Since both extensions are $\nu$-tame, we see that $\nu(z)$ has finite order relatively prime to $p$. Since $\nu(z)$ also reduces to the identity, we see that $\nu(z)=1$. By Lemma \ref{lem:isogeny of tori}, $\nu: \widehat{Z}(R) \to \widehat{S}(R)$ is injective, we conclude that $z=1$ and hence the extensions are the same. 
\end{proof}

\begin{remark}
We call this step in the argument ``Clifford theory'' as the analogous step for $\GL_n$ \cite[Lemma 2.4.11]{cht08} makes use of ideas from Clifford theory (see \cite[\S11]{CR81}).
\end{remark}

\section{Lifts and Minimally Ramified Deformations} \label{sec:liftingsmr}

As before, let $F$ be a local field of residue characteristic $\ell \neq p$ and $G$ be a weakly reductive group scheme over a $p$-adic ring of integers $\cO$ with residue field $k$.  

\subsection{Lifting Residual Representations} \label{ss:lifting}

Recall the terminology for decomposition types from Definitions~\ref{def:decomp type} and \ref{def:adapted}.  We will prove the following theorem over the course of this subsection.

\begin{theorem} \label{thm:liftingfixed}
Let $\rhobar : \Gamma_F \to G(k)$ be a continuous homomorphism.  Suppose there exists a decomposition type $(\fC,\Delta)$ over $\cO$ adapted to $\rhobar(\Lambda_F)$.  Suppose that $p$ is good for the decomposition type $(\fC,\Delta)$.
Then there exists a continuous homomorphism $\rho : \Gamma_F \to G(\cO)$ that lifts $\rhobar$ such that $\cent{G}{\rho(I_F)}$ is $\cO$-smooth. %
\end{theorem}

\begin{remark} \label{remark:extraconditions}
The proof will also show that $\rho(\Lambda_F) \subset \Delta(\cO)$ and give some control over the ``inertial type'' of the lift. We remark that the $\cO$-smoothness of $\cent{G}{\rho(I_F)}$ is crucial for identifying a formally smooth component of the deformation ring, see Theorem \ref{thm:liftingcondition}. In the course of the proof, we will construct ``many'' lifts parameterized by an element $z \in \fC^0(\cO)$. This is not strictly necessary for the proof, and in fact, one can take $z$ to be the identity element throughout this subsection. However, we use this $z$ to build flexibility into our argument and allow inertial types that are not minimally ramified, which correspond to other components of the deformation ring. 
\end{remark}

Set $ \overline{\tau} \colonequals \rhobar |_{\Lambda_F}$ and let $\overline{\Ld} \colonequals  \taubar(\Ld_F)$. 
As the decomposition type is adapted to $\rhobar(\Lambda_F)$, by Proposition~\ref{prop:adapted complete dvr}
    there exists $\Lambda \subset \Delta(\cO)$ lifting $\Lbar$ 
    such that $C_G(\Lambda)= \fC$ and $\Delta = C_G(\fC)$.  As the pro-order of $\Lambda_F$ is prime to $p$ there is also
a lift $\tau : \Lambda_F \to \Delta(\cO)$ of $\rhobar|_{\Lambda_F}$ with $\ker(\tau) = \ker(\overline{\rho}|_{\Lambda_F})$. Thus $\Lambda = \tau(\Lambda_F)$.

\begin{lemma} \label{saturation}
The order of $\pi_0\cent{G}{\Lambda}$ is prime to $p$. 
\end{lemma}
\begin{proof}
This follows from Corollary~\ref{cor:double-centralizer}. 
\end{proof}

\begin{lemma} \label{lem:normalizers} 
We have that $\nlz{G}{\Lambda} \subset \nlz{G}{\Dt}$.  
\end{lemma}

\begin{proof}
Any point of $\nlz{G}{{\Lambda}}$ normalizes ${\Ld}$ and hence normalizes $\cent{G}{\Ld}$ and therefore normalizes $\cent{G}{\cent{G}{\Ld}}=\Dt$.
\end{proof}

As in Section~\ref{ss:galois}, fix a preimage $\sigma$ of a topological generator of $\Z_p$ under the surjection $I_F \to \Z_p$ and a Frobenius $\phi$ satisfying  
\[ \phi \sigma \phi^{-1}=\sigma^q \]
where $q$ is the size of the residue field of $F$ (a power of $\ell$).  
As $\sigma$ and $\phi$ normalize $\Ld_F$, $\rhobar(\sigma)$ and $\rhobar(\phi)$ belong to $\nlz{G(k)}{\overline{\Lambda}}$.
Then Lemma~\ref{lem:normalizers} shows that
$\rhobar(\sigma)$ and $\rhobar(\phi)$ are contained in $\cN(k) \colonequals \nlz{G(k)}{\Dt(k)}$.
As $I_F/\Ld_F$ is pro-p and $p$ does not divide the index of 
$\Delta(k) \fC(k)$ in $\cN(k)$ as $p$ is good for the decomposition type, it follows that $\rhobar(\sigma) \in \Dt(k) \cdot \fC(k)$.  

Recall the morphism $\nu : \Delta \to \Dt^{\ab}$ and the notion of $\nu$-tameness from Definition~\ref{def:nutame}.

\begin{proposition} \label{prop:extension R=O}
There is a continuous, $\nu$-tame extension of $\tau$ to a homomorphism
\[ \tau: I_F \to \Dt(\cO). \]
\end{proposition}

We use $\tau$ and $\taubar$ to denote the extensions as well as the functions originally defined on $\Lambda_F$. 

\begin{proof}
We will apply 
Proposition \ref{prop:extension} with $R=\cO$. 
First, since $p$ is good for the decomposition type, Proposition \ref{prop:smoothness of delta etc} implies that $\nu: Z(\Dt)^0 \to \Dt^{\ab,0}$ is smooth. 
It remains to check the three hypotheses.  The first is built into the construction of $\tau$. Since $\rhobar(\sigma) \in \Dt(k) \cdot \fC(k)$ and $\Dt$ and $\fC$ commute, it follows that $\taubar^{\sigma}=\rhobar(\sigma) \taubar \, \rhobar(\sigma)^{-1}$ and $\taubar$ are isomorphic as $\Dt(k)$-valued representations, whence the second condition. Lemma~\ref{lem:adapted} gives the third condition.
\end{proof}

\begin{corollary} \label{cor:conjugation}
We have
$\taubar^{\phi}=\rhobar(\phi) \cdot \taubar \cdot \rhobar(\phi)^{-1}$ on $I_F$.
\end{corollary}

\begin{proof}
Since $\rhobar(\phi) \in \nlz{G(k)}{\Dt(k)}$,
this is a consequence of the uniqueness part of Proposition~\ref{prop:extension}.
\end{proof}

We next analyze $\rhobar$ as a combination of $\taubar$ (valued in $\Delta$) and a representation valued in $\fC^0$.

\begin{proposition}\label{prop:omega} 
Continuing the standing assumptions:
\begin{enumerate}
    \item There is a continuous homomorphism 
$\omegabar: I_F/\Ld_F \to \fC^0(k)$
such that 
\[ \rhobar|_{I_F}=\taubar \cdot \omegabar = \omegabar \cdot \taubar. \]
\item  The element $\omegabar(\sigma)$ is unipotent, and there exists a pure fiberwise unipotent $u \in \fC^0(\cO)$ lifting $\omegabar(\sigma)$.  

\item  For any element $z \in \fC^0(\cO)$ which reduces to the identity in $\fC^0(k)$ and commutes with $u$, define a homomorphism $\omega_z : I_F / \Lambda_F \to \fC^0(\cO)$ sending $\sigma$ to $z u$.
 The function $ \tau \cdot \omega_z : I_F \to \Delta(\cO) \fC^0(\cO) \subset G(\cO)$ is a continuous homomorphism lifting $\overline{\rho}|_{I_F}$.
\end{enumerate}
\end{proposition}

(At a first reading, it is fine to take $z$ to be the identity.)

\begin{proof}
We define $\omegabar \colonequals \taubar^{-1} \cdot \rhobar|_{I_F}$. For any $g \in I_F$ and for any $h \in \Ld_F$, 
\[\omegabar(g)\rhobar(h)\omegabar(g)^{-1}=\taubar(g^{-1})\rhobar(ghg^{-1})\taubar(g)=\taubar(g^{-1})\taubar(ghg^{-1})\taubar(g)=\taubar(h)=\rhobar(h),\]
where the second equality holds because $ghg^{-1} \in \Ld_F$ and $\rhobar|_{\Ld_F}=\taubar|_{\Ld_F}$. Thus $\omegabar$ is valued in $\cent{G_{k}}{\Lbar}(k)$. 
Since $I_F/\Ld_F$ is pro-$p$ and $\pi_0\cent{G_{k}}{\Lbar}$ has order prime to $p$ by Lemma \ref{saturation},
$\omegabar$ is valued in $\cent{G}{\Lbar}^0(k)=\fC^0(k)$, with the equality following from Definition \ref{def:adapted}.

Moreover, for any $g,h \in I_F$ we see that
\[
\omegabar(gh)=\taubar(gh)^{-1}\rhobar(gh)=\taubar(h)^{-1}\taubar(g)^{-1}\rhobar(g)\rhobar(h)=\taubar(g)^{-1}\rhobar(g)\taubar(h)^{-1}\rhobar(h)=\omegabar(g)\omegabar(h),\] where the third equality uses that $\taubar(g)^{-1}\rhobar(g)=\omegabar(g) \in \cent{G_{k}}{\Lbar}(k)=\fC(k)$ commutes with $\taubar(h)^{-1} \in \Dt(k)$ (Definition \ref{def:decomp type}).
So $\omegabar$ is a group homomorphism. It is continuous because both $\rhobar|_{I_F}$ and $\taubar$ are continuous.

Second, we show that $\omegabar(\sigma)$ is unipotent.
This element decomposes as a commuting product of semisimple and unipotent elements of $\fC^0(k)$. The
order of a semisimple element in $\fC^0(k)$ is prime to $p$,  while by continuity there is an $r \geq 0$ such that $\sigma^{p^r} \in \ker \omegabar$. Thus $\omegabar(\sigma)$ is unipotent. By Lemma~\ref{lemma:purelift} there exists a pure unipotent in $\fC^0(\cO)$ lifting $\rhobar(\sigma)$.

For the last statement, note that $\omega_z$ is continuous as the image of a topological generator of $I_F /\Lambda_F \simeq \Z_p$ reduces to a unipotent element.  The function $\tau \cdot \omega_z$ is a homomorphism 
 as $\Dt$ and $\fC^0$ commute. 
\end{proof}

\begin{defn} \label{def:theta}
For $z$ as above, define $\theta_z : I_F \to \Delta(\cO) \fC^0(\cO)$ to be the product $\tau \cdot \omega_z$.
\end{defn}

 Note that $\bar\theta^{\phi}_z= \rhobar^{\phi} |_{I_F} =  \rhobar(\phi) \cdot \bar\theta_z \cdot \rhobar(\phi)^{-1}$.  Corollary~\ref{cor:conjugation} implies that $\omegabar^{\phi}=\rhobar(\phi) \cdot \omegabar \cdot \rhobar(\phi)^{-1}$, hence using the structure of $\Gamma_F/\Lambda_F$ from Lemma~\ref{lem:tq} that
\begin{equation} \label{eq:conjugationmodp}
     \overline{\omega}(\sigma)^q = \rhobar(\phi) \overline{\omega}(\sigma) \rhobar(\phi)^{-1}.
 \end{equation}

\begin{lemma} \label{lemma:nlifting}
Continuing the standing assumptions, there exists $n \in G(\cO)$ lifting $\rhobar(\phi)$ such that 
\begin{enumerate}
    \item $\tau^\phi = n \tau n^{-1}$ on $I_F$ and
    \item  $u^q = n u n^{-1}$.
\end{enumerate}
\end{lemma}

\begin{proof}
There exists an element $n \in G(\cO)$ lifting $\rhobar(\phi)$ as $G$ is $\cO$-smooth. Since the mod-$p$ reductions of $\tau^{\phi}$ and $n \tau n^{-1}$ agree on $\Ld_F$ and $\Ld_F$ is prime to $p$, we may and do multiply $n$ by an element in $\widehat G(\cO)$ so that $\tau^{\phi}=n \tau n^{-1}$ on $\Ld_F$. In particular, using Lemma~\ref{lem:normalizers} we see that 
$n \in N_{G(\cO)}(\Ld) \subset N_{G(\cO)}(\Dt(\cO))$.

Since $\tau$ is $\nu$-tame, so is $\tau^{\phi}$. On the other hand, $\Ad n: \Dt_{\cO} \to \Dt_{\cO}$ induces an automorphism $A: \Dt^{\ab}_{\cO} \to \Dt^{\ab}_{\cO}$ with $\nu \circ \Ad n = A \circ \nu$ by natural properties of the abelianization map $\nu: \Dt \to \Dt^{\ab}$. So for any $\sigma \in I_F$, $\nu(n \cdot \tau(\sigma) \cdot n^{-1})=A(\nu(\tau(\sigma)))$ has finite prime to $p$ order. Therefore $n \tau n^{-1}$ is $\nu$-tame, and the uniqueness part of Proposition~\ref{prop:extension} then implies that $\tau^{\phi}=n \tau n^{-1}$ on $I_F$.

Note that $u$ is pure by construction, and Lemma~\ref{lem:conjugacypower} implies that $u^q$ is pure. Furthermore, $n u n^{-1} \in \fC^0(\cO)$ is also pure.  
Now \eqref{eq:conjugationmodp} shows that $u^q$ and $n u n^{-1}$ agree in the special fiber.
Recall that $p$ is pretty good for $G$, so Lemma \ref{lem:cglambda pretty good} implies that $p$ is pretty good for $\fC=\cent{G}{\Ld}$.
Hence we can apply \cite[Theorem 5.11]{cotner} to conclude that $u^q$ and $n u n^{-1}$ are $\fC^0(\cO)$-conjugate, i.e. $c n u n^{-1} c^{-1} = u^q $ for some $c \in \fC^0(\cO)$.  

We claim that we may choose $c$ so that it reduces to the identity.  We know that the reduction of $c$ centralizes $\overline{u}^q = \rhobar(\phi) \overline{u} \rhobar(\phi)^{-1}$.  As the centralizer $C_{\fC}(u^q)$ is smooth by Theorem \ref{theorem:unipotentcentralizer} (recall that $\fC$ is weakly reductive and $u^q$ is pure fiberwise unipotent), we may modify $c$ so that it reduces to the identity.  
Making this choice, it is clear that $cn$ reduces to $\rhobar(\phi)$ and that $u^q = (cn) u (cn)^{-1}$.  As $\tau^{\phi}=n \tau n^{-1}$ takes values in $\Dt$ and the groups $\Dt$ and $\fC$ commute, we have $\tau^{\phi}=cn \tau n^{-1}c^{-1}$ on $I_F$ as desired.
\end{proof}

We now have the necessary ingredients to complete the proof.

\begin{proof}[Proof of Theorem~\ref{thm:liftingfixed}]
Given $\rhobar : \Gamma_F \to G(k)$, 
construct $\theta_z : I_F \to G(\cO)$ lifting $\rhobar|_{I_F}$ as in Proposition~\ref{prop:omega}, depending on the choice of pure fiberwise unipotent $u \in \fC^0(\cO)$ lifting $\omegabar(\sigma)$ and choice of $z$.  
Then Lemma~\ref{lemma:nlifting} gives an element $n \in G(\cO)$ lifting $\rhobar(\phi)$ such that $\tau^{\phi} = n \tau n^{-1}$ on $I_F$ and such that $u^q = n u n^{-1}$.  

We now fix a choice of $z \in \fC^0(\cO)$ so that:
\begin{itemize}
    \item  $z$ reduces to the identity in $\fC^0(k)$;

    \item  $z$ commutes with $u$;
    
    \item  $n z n^{-1} = z^q$.
\end{itemize}
(For example, we may take $z=1$.)
We define the lift $\rho : \Gamma_F \to G(\cO)$ to be $\theta_z$ on $I_F$ and by setting $\rho(\phi) = n$.  Note that
\[
\omega_z^\phi(\sigma) = \omega_z(\sigma^q) = (z u)^ q = z^q u^q .
\]
Then $\rho$ is a homomorphism as
\[
\rho(\phi) \theta_z(\sigma) \rho(\phi)^{-1} = n \tau(\sigma) n^{-1} n z u n^{-1} = \tau^\phi(\sigma) \omega_z^\phi(\sigma) = \theta_z^{\phi}(\sigma)
\]
where the penultimate step uses that $n u n^{-1} = u^q$ and that $n z n^{-1} = z^q$.   
It is continuous as $\theta_z$ is continuous and $\overline{\rho}(\phi) = \bar{n}$ has finite order. 

Finally, Lemma \ref{lem:centralizer of theta} implies that if moreover $z \in Z(\fC)(\cO)$ then the centralizer of the inertia is $\cO$-smooth. 
\end{proof}

\begin{example} \label{ex:inertialtype}
Taking $z=1$ is simplest, and gives a minimally ramified lift.  Other choices of $z$ give lifts with different inertial types, which are of interest but are not the focus of the present work.  We will just give an example.

 After fixing the choice of $\tau$, the restriction of $\rho$ to $I_F$ depends on $u$ and $z$.  Let $V = \spn_{\cO}\{e_1,e_2,e_3,e_4\}$ and $\fC$ be the Levi preserving the grading $V= \spn \{e_1,e_2\} \oplus \spn \{e_3,e_4\}$; observe $V \simeq \GL_2 \times \GL_2 \subset \GL_4$.  Suppose $u$ is the unipotent such that
\[
u(e_1) = e_1, \quad  u(e_2) = e_1  +e_2, \quad u(e_3) = e_3, \quad u(e_4) = e_3 + e_4.
\]
One possible case is that $\rhobar(\phi)$ and hence $n$ swap $\spn \{e_1,e_2\}$ and $\spn \{e_3,e_4\}$.  In that case, for any $p^a$-th root of unity $\zeta \in \cO$ where $p^a | (q-1)$, we can take the scalar matrix $z = \zeta I_4$.
It is clear that $\zeta$ reduces to the identity in $k$, that $n$ and $u$ commute with $z$, and that $z^q = z$.  In this case, the centralizer of $\rho(I_F)$ equals $C_{\fC}(u)$ and is $\cO$-smooth.
\end{example}

\begin{lem} \label{lem:centralizer of theta}
Under the standing assumptions, if $z \in Z(\fC)(\cO)$ and satisfies the conditions on $z$ in the proof of Theorem~\ref{thm:liftingfixed} then
$\cent{G}{\theta_z(I_F)}=\cent{\fC}{\omega_z(\sigma)}$ is $\cO$-smooth.
\end{lem}

\begin{proof}
Note that $\theta_z(I_F)$ is generated by $\theta_z(\Ld_F)=\tau(\Ld_F)$ and $\theta_z(\sigma)=\tau(\sigma)\omega_z(\sigma)$. 
For any $\cO$-algebra $R$, we have
$$\cent{G(R)}{\theta_z(I_F)}=\cent{G(R)}{\tau(\Ld_F)} \cap \cent{G(R)}{\theta_z(\sigma)}=\fC(R) \cap \cent{G(R)}{\theta_z(\sigma)}.$$
Now an element of $\fC(R)$ commutes with $\theta_z(\sigma) = \tau(\sigma) \omega_z(\sigma)$ if and only if it commutes with $\omega_z(\sigma)$ since $\tau(\sigma) \in \Delta(\cO)$ and $\Delta$ commutes with $\fC$, and hence
$\cent{G(R)}{\theta_z(I_F)}=\cent{\fC(R)}{\omega_z(\sigma)}$. When $z$ is central in $\fC$, we see that $\cent{\fC(R)}{\omega_z(\sigma)} = \cent{\fC(R)}{u}$.  Thus the smoothness follows from Theorem \ref{theorem:unipotentcentralizer} (recall that $\fC$ is weakly reductive and $u$ is pure fiberwise unipotent). 
\end{proof}

\begin{remark}
If $z$ is not central in $\fC$, the centralizer need not be smooth.  In particular, the dimension of the centralizer of the semisimple part of $\omega_z(\sigma) = z u$ is not locally constant as would be needed to apply \cite[Theorem 1.1]{cotner} (since $z$ reduces to the identity). See \cite[Remark 5.5]{cotner} for some examples of the failure of flatness for centralizers of pure sections.
\end{remark}

\subsection{The Minimally Ramified Deformation Condition} \label{ss:mrdefcondition}

We now fix a residual representation $\rhobar : \Gamma_F \to G(k)$.  
For a continuous homomorphism $\theta \colon I_F \to G(\cO)$ lifting $\rhobar|_{I_F}$, let $\theta_R$ denote the composition $I_F \to G(\cO) \to G(R)$ for an $\cO$-algebra $R$. Recall that $\CLN_{\cO}$ is the category of coefficient $\cO$-algebras (complete local Noetherian rings with residue
field $k$, with morphisms local homomorphisms inducing the identity map on $k$ and with the
structure morphism a map of coefficient rings). 
For the rest of this subsection, we suppose there exists a decomposition type $(\fC, \Dt)$ over $\cO$ adapted to $\rhobar(\Ld_F)$ and that $p$ is good for $(\fC, \Dt)$. By the paragraph below Remark \ref{remark:extraconditions}, there is a lift $\tau \colon \Ld_F \to \Dt(\cO)$ of $\rhobar|_{\Ld_F}$ such that $\cent{G}{\tau(\Ld_F)}=\fC$. Moreover, $\tau$ extends to a homomorphism $\tau \colon I_F \to \Dt(\cO)$ by Proposition \ref{prop:extension R=O}. We fix this homomorphism throughout this subsection.

\begin{lemma} \label{lem:omega over rings}
Let $\theta \colon I_F \to G(\cO)$ be a continuous homomorphism lifting $\rhobar|_{I_F}$. Then there exists $g \in \widehat G(\cO)$ and $\omega \colon I_F/\Ld_F \to \fC(\cO)$ such that 
$\theta^g = \tau \cdot \omega$. 
Moreover, $\omega$ is unique up to $\fC(\cO)$-conjugacy.
\end{lemma}

\begin{proof}
Since both $\theta|_{\Ld_F}$ and $\tau|_{\Ld_F}$ lift $\rhobar|_{\Ld_F}$, they are $\widehat G(\cO)$-conjugate. So after conjugating, we may assume that $\theta|_{\Ld_F}=\tau|_{\Ld_F}$. An argument similar to the first part of the proof of Proposition \ref{prop:omega} then gives the first statement. For the second part, suppose $\theta^g=\tau \cdot \omega$ and $\theta^{g'}=\tau \cdot \omega'$ for $\omega, \omega' \colon I_F/\Ld_F \to \fC(\cO)$. 
Restricting to $\Ld_F$, we see that $g^{-1}g'$ centralizes $\tau(\Ld_F)$, and hence $c \colonequals g^{-1}g' \in \fC(\cO)$. So $\theta^{g'}=\theta^{gc}=\tau^c \cdot \omega^c=\tau \cdot \omega^c$, and hence $\omega'=\omega^c$. 
\end{proof}

\begin{cor} \label{cor:unique theta}
There is a unique $\widehat G(\cO)$-conjugacy class of lifts $\theta \colon I_F \to G(\cO)$ of $\rhobar|_{I_F}$ such that the associated homomorphism $\omega$ satisfies that $\omega(\sigma)$ is a pure unipotent element in $\fC^0(\cO)$. Moreover, $\theta$ extends to a continuous homomorphism $\gma{F} \to G(\cO)$.
\end{cor}

\begin{proof}
The existence follows from Section \ref{ss:lifting} (taking $z=1$), which also shows that $\theta$ extends to $\gma{F}$.
For uniqueness, by Lemma \ref{lem:omega over rings} it suffices to show that if $u$ and $u'$ are pure unipotents in $\fC^0(\cO)$ lifting $\omegabar(\sigma) \in \fC^0(k)$ then they are $\fC^0(\cO)$-conjugate. 
Recall that $p$ is pretty good for $G$, so Lemma \ref{lem:cglambda pretty good} implies that $p$ is pretty good for $\fC=\cent{G}{\tau(\Ld_F)}$.
Since $u$ and $u'$ agree on the special fiber, \cite[Theorem 5.11]{cotner} implies that $u$ and $u'$ are $\fC^0(\cO)$-conjugate. 
\end{proof}

\begin{defn} \label{def:minimally ramified}
Let $\theta \colon I_F \to G(\cO)$ be a continuous homomorphism lifting $\rhobar|_{I_F}$.
Let $\mathrm{Lift}_{\rhobar}^{\theta} \colon \widehat{\cC}_{\cO} \to \textrm{Sets}$ be the functor whose set of $R$-points $\mathrm{Lift}_{\rhobar}^{\theta}(R)$ is given by lifts $\rho: \gma{F} \to G(R)$ of $\rhobar$ such that there exists $g \in \widehat{G}(R)$ such that $\rho^g|_{I_F}=\theta_R$.

Let $\mathrm{Lift}_{\rhobar}^{\mr}$ be this functor when $\theta$ is a homomorphism as in Corollary \ref{cor:unique theta}. We call this the \emph{minimally ramified lifting condition} for $\rhobar$. 
\end{defn}

We will use the notions of lifting conditions and deformation conditions for Galois representations which are reviewed in \cite[Definition 2.3]{Booher19}, in particular the notion of liftability and its connection with formal smoothness of the representing object.

\begin{thm} \label{thm:liftingcondition}
Let $\theta \colon I_F \to G(\cO)$ be a continuous homomorphism lifting $\rhobar|_{I_F}$.
Suppose that $\cent{G}{\theta(I_F)}$ is $\cO$-smooth and $\theta$ extends to a continuous homomorphism $\rho_0 \colon \Gamma_F \to G(\cO)$.  Then $\mathrm{Lift}_{\rhobar}^{\theta}$ is a well-defined lifting condition. Moreover, it is liftable and the tangent space to the corresponding deformation functor $\mathrm{Def}_{\rhobar}^{\theta}$ has dimension $\dim_{k} \mathrm{H}^0(\gma{F}, \rhobar(\fg))$. 
\end{thm}

\begin{proof}
The functor $\mathrm{Lift}_{\rhobar}^{\theta}$ is obviously closed under strict equivalence.  The key to checking the second condition of \cite[Definition 2.3]{Booher19} is that $\cent{G}{\theta(I_F)}$ is $\cO$-smooth.  In particular, suppose we have a Cartesian diagram in $\CLN_{\cO}$
 \[
  \xymatrix{
  R_1 \times_{R_0} R_2 \ar[r]\ar[d] & R_2\ar[d]\\
  R_1 \ar[r] & R_0
  }
 \]
and $\rho_1 \in \mathrm{Lift}_{\rhobar}^{\theta}(R_1) $ and $\rho_2 \in \mathrm{Lift}_{\rhobar}^{\theta}(R_2)$ with common image in $\mathrm{Lift}_{\rhobar}^{\theta}(R_0)$.  There exists $g_1 \in \widehat{G}(R_1)$ and $g_2 \in \widehat{G}(R_2)$ such that $g_1 \rho_1|_{I_F} g_1^{-1} = \theta_{R_1}$ and $g_2 \rho_2 |_{I_F} g_2^{-1} = \theta_{R_2}$.  Looking at the images in $R_0$, we conclude that $g_1 g_2^{-1} \in C_{G(R_0)}(\theta(I_F))$. Since $\cent{G}{\theta(I_F)}$ is $\cO$-smooth, there exists a lift $x \in C_{G(R_1)}(\theta(I_F))$.  Then the element $(x g_1, g_2) \in G(R_1 \times_{R_0} R_2)$ conjugates $(\rho_1,\rho_2)$ to $(\theta_{R_1},\theta_{R_2})$ so $(\rho_1,\rho_2) \in \mathrm{Lift}_{\rhobar}^{\theta}(R_1 \times_{R_0} R_2)$ as desired.

To check smoothness, let $R \to S$ be a small morphism in $\CLN_{\cO}$, and let $\rho_S \in \mathrm{Lift}_{\rhobar}^{\theta}(S)$.  We need to show that there exists $\rho_R \in \mathrm{Lift}_{\rhobar}^{\theta}(R)$ mapping to $\rho_S$ under the morphism $R \to S$. Since $G$ is $\cO$-smooth, we may assume that $\rho_S|_{I_F}=\theta_S$ by $G(S)$-conjugation. 
Since $\theta^{\phi}= \rho_0(\phi) \theta \rho_0(\phi)^{-1}$ and 
$\theta^{\phi}_S = \rho_S(\phi) \theta_S \rho_S(\phi)^{-1}$, $\rho_0(\phi)^{-1} \rho_S(\phi)$ centralizes $\theta_S(I_F)$.  Using the $\cO$-smoothness of $\cent{G}{\theta(I_F)}$, we obtain an element $c \in C_{G(R)}( \theta(I_F))$ lifting $\rho_0(\phi)^{-1} \rho_S(\phi)$.  Then $\rho_0(\phi) c$ lifts $\rho_S(\phi)$ and on $I_F$
\[
(\rho_0(\phi) c) \theta_R (\rho_0(\phi) c)^{-1} = \rho_0(\phi) (\rho_0|_{I_F}) \rho_0(\phi)^{-1} = \theta_R^{\phi}.
\]
Define $\rho_R: \gma{F} \to G(R)$ by $\rho_R|_{I_F}=\theta_R$ and $\rho_R(\phi)=\rho_0(\phi) c$, and observe that $\rho_R \in \mathrm{Lift}_{\rhobar}^{\theta}(R)$ and that it maps to $\rho_S$. 

Notice that the tangent space to $\mathrm{Def}_{\rhobar}^{\theta}$ equals 
$\Ker (\mathrm{H}^1(\gma{F}, \rhobar(\fg)) \to \mathrm{H}^1(I_F, \rhobar(\fg))$. The last claim then follows from inflation-restriction and the fact that $\dim \mathrm{H}^1(\gma{F}/I_F, \rhobar(\fg)^{I_F})=\dim \mathrm{H}^0(\gma{F}, \rhobar(\fg))$ \cite[Lemma 1]{washington95}.
\end{proof}

\begin{corollary}  \label{cor:liftingcondition}
Let $\rhobar : \Gamma_F \to G(k)$ be a continuous homomorphism.  Suppose there exists a decomposition type $(\fC,\Delta)$ over $\cO$ adapted to $\rhobar(\Lambda_F)$, and that $p$ is good for $(\fC,\Delta)$.
Then $\mathrm{Lift}_{\rhobar}^{\mr}$ is a formally smooth lifting condition. Moreover, the corresponding deformation condition
$\mathrm{Def}_{\rhobar}^{\mr}$ has dimension $\dim_{k} \mathrm{H}^0(\gma{F}, \rhobar(\fg))$.  Equivalently, $\Spf R^{\mr,\square}_{\rhobar}$ is $\cO$-formally smooth of relative dimension $\dim_k G_k$.
\end{corollary}

\begin{proof}
Construct a lift $\rho$ as in Section \ref{ss:lifting} with $z=1$, and observe that $\cent{G}{\theta(I_F)}$ is $\cO$-smooth by Lemma~\ref{lem:centralizer of theta}.  
Applying Theorem~\ref{thm:liftingcondition} with $\theta=\rho|_{I_F}$, we obtain that $\mathrm{Lift}_{\rhobar}^{\theta}$ is formally smooth. Note that $\theta$ satisfies Corollary \ref{cor:unique theta} by construction (since $z=1$ in Proposition \ref{prop:omega}), and so $\mathrm{Lift}_{\rhobar}^{\theta}=\mathrm{Lift}_{\rhobar}^{\mr}$. 
Finally, note that adding a framing increases the relative dimension by $\dim_k G_k$ minus the dimension of the automorphisms compatible with the $\Gamma_F$-action, $\dim_k H^0(\Gamma_F,\rhobar(\fg))$.
\end{proof}

\begin{remark}
Let $\nu_R : G \to G^{\mathrm{ab}}=G/\sD(G)$ be the quotient map (see Proposition \ref{prop:abelianization}).
We could also formulate a variant of the deformation condition where the morphism $\nu_R \circ \rho_R : \Gamma_F \to G(R) \to G^{\mathrm{ab}}(R)$ is a fixed lift of $\nu_k \circ \rhobar$.
This generalizes adding the requirement that the determinant of the lift be fixed when $G = \GL_n$.
\end{remark}

\begin{remark}
The deformation conditions $\mathrm{Def}_{\rhobar}^{\theta}$ where $\theta$ has different inertial types should be of interest when investigating $\ell \neq p$ versions of the Breuil-M\'{e}zard conjecture. Lifts where $C_G(\theta(I_F))$ is $\cO$-smooth give formally smooth components of the lifting ring $\Spf R^{\square}_{\rhobar}$.
For example, we may take $\theta=\theta_z$ for $z \neq 1$ as in Lemma \ref{lem:centralizer of theta}. 
\end{remark}

We can now prove our main theorem from the introduction. Recall that $G$ is a weakly reductive group defined over $\cO$.

\begin{theorem} \label{thm:main theorem}
Let $\rhobar : \Gamma_F \to G(k)$ be a continuous homomorphism. Suppose that $p$ is large enough relative to the root datum of $G$ (the bound can be made effective). 
Then $\mathrm{Lift}_{\rhobar}^{\mr}$ is a formally smooth lifting condition. Moreover, the corresponding deformation condition
$\mathrm{Def}_{\rhobar}^{\mr}$ has dimension $\dim_{k} \mathrm{H}^0(\gma{F}, \rhobar(\fg))$.  Equivalently, $\Spf R^{\mr,\square}_{\rhobar}$ is $\cO$-formally smooth of relative dimension $\dim_k G_k$.
\end{theorem}

\begin{proof}
By Proposition \ref{prop:existence of decomp type}, if $p$ is pretty good for $G$, then there exists a decomposition type $(\fC,\Delta)$ over $\cO$ adapted to $\Lbar \colonequals \rhobar(\Lambda_F)$. We need to ensure that $p$ is good for $(\fC,\Delta)$ in the sense of Definition \ref{def:goodprime}.
Condition (1) is trivially satisfied.  If $p>\rank \sD(G^0) +1$, condition (2) follows as well. In fact, for any connected semisimple group $H$ and any prime $p$ dividing the order of $Z(H)$, $p \leq \rank H +1$. To check this,   
we may assume that $H$ is simply-connected, and hence it is a direct product of simple, simply-connected groups, in which case the claim follows from a simple case-checking.  
It remains to consider Definition \ref{def:goodprime}, (3). By Proposition \ref{prop:p not divide weyl}, this will hold if $p \nmid |W_G|$ and $p>|\pi_0 \fC_k|=|\pi_0 \cent{G_k}{\Lbar}|$. Theorem \ref{thm:bound nonconnected} gives a bound for the last quantity which depends only on $G$.
The theorem now follows from Corollary \ref{cor:liftingcondition}. 
\end{proof}

\begin{remark} \label{rmk:bound for G}
By the proof of the above theorem, the condition that $p$ is large enough relative to the root datum of $G$ is equivalent to the following: 
\begin{itemize}
    \item $p$ is pretty good for $G$ in the sense of Definition \ref{def:prettygoodprimes}.
    \item $p>\rank \sD(G^0) +1$.
    \item $p$ does not divide the order of $W_G$.
    \item $p$ is larger than the constant in Theorem \ref{thm:bound nonconnected}. 
\end{itemize}

The last condition is explicit but not always pleasant.  To illustrate, notice that the constant in Theorem \ref{thm:bound nonconnected} must be at least the number of components of the centralizer of any particular prime-to-$p$ solvable subgroup of $G_k$.  

 When $G = \GL_n$, Schur's lemma shows that centralizer of a completely reducible subgroup is a product of general linear groups.  Hence the component group is always trivial.

 For the exceptional group of type $G_2$, the lower bound for $p$ is 72, see Example \ref{example:g2 bound}.

 When $G = \PGL_n$, there is an example of a subgroup whose centralizer is finite of order $n^2$ \cite[Examples(3)]{liebeck}.  Similarly, when $G = \Sp_{2n}$ there is an example of a subgroup whose centralizer is finite of order $2^n$ \cite[Examples(2)]{liebeck}.

In particular, for $G = \Sp_{2n}$ the last example shows that we must require $p>2^n$.  This is far from optimal: the techniques from \cite{Booher19} work when $p>2n$.
\end{remark}

\begin{remark}
Note that the last bulleted point in Remark \ref{rmk:bound for G} is needed to ensure Definition \ref{def:goodprime}, (3) for the decomposition type $(\fC, \Dt)$ adapted to the residual representation $\rhobar$. This in turn implies that $\rhobar(I_F) \subset \fC(k) \Dt(k)$, which is the starting point of the argument in Section \ref{ss:lifting}. Currently we do not know how to construct lifts of $\rhobar$ without this condition. 
\end{remark}

\begin{remark}\label{remark:dhkm-interpretation}
    A variant of the arguments in this section can be used to show the following mild strengthening of Theorem~\ref{thm:main theorem}. Maintain all of the assumptions on $F$, $\cO$, $G$, and so on. Let $W$ be the Weil group of $F$, and let $W^0$ be the ``discretized version" from \cite{DHKM}, obtained by choosing a topological generator of tame intertia. Let $\sH_W$ be the moduli space $\underline{\Hom}_{\cO\textrm{-}\rm{gp}}(W^0, G)$, which is representable by \cite{DHKM}. Moreover, \cite[Theorem 1.5]{DHKM} shows that the $p$-adic formal completion of $\sH_W$ represents the functor on $p$-adically complete (and separated) $\cO$-algebras given by
    \[
    R \mapsto \Hom_{\mathrm{cts}}(W, G(R)).
    \]
    There is a similar functor $\sH_{\Lambda}$, obtained from $\sH_W$ by replacing $W$ by $\Lambda_F$. Attached to $\sH_{\Lambda}$ are the universal centralizer $\fC$ and the universal double centralizer $\Delta$.

    For any continuous $f: W \to G(R)$ as above, the restriction $f|_{I_F}$ factors uniquely as $f|_{I_F} = \tau \cdot \omega$, where $\tau: I_F \to \Delta(R)$ is the unique tame extension of $f|_{\Lambda_F}$ (which exists by a variant of Proposition~\ref{prop:extension}) and $\omega: I_F/\Lambda_F \to \fC(R)$ is a homomorphism. Let $\sH_W^{\rm{pure}}$ be the subfunctor of $\sH_W$ with $\sH_W^{\rm{pure}}(R)$ consisting of those $f$ as above such that $\omega(\sigma)$ is pure fiberwise unipotent, where $\sigma$ is a generator of the pro-$p$ part of $I_F$. (If $R$ is not reduced, ``pure fiberwise unipotent" should be interpreted to mean that $f(\sigma)$ is fppf-locally on $R$ conjugate to a pure fiberwise unipotent section of $G(\cO')$, for some finite extension $\cO'$ of $\cO$.)
    
    The arguments of this section can be used to show that $\sH_W^{\rm{pure}}$ is formally smooth and the inclusion map $\sH_W^{\rm{pure}} \to \sH_W$ gives a stratification of $\sH_W$ into smooth pieces, each of which is open in an irreducible component of $\sH_W$ (for dimension reasons coming from Theorem~\ref{thm:main theorem}). Because of this openness statement, it follows that every point $f$ in the special fiber of $\sH_W$ lies in an irreducible component $X$ such that $f$ is a smooth point of $X_{\rm{red}}$.
\end{remark}

\appendix

\section{Some group theory}
The following Theorem is due to Liebeck \cite{liebeck}, which originated in an email correspondence between the third author (S.T.) and Martin Liebeck. S.T. would like to thank Martin Liebeck for his interest in our question and for answering it, which allows us to obtain an effective lower bound for $p$ in Theorem \ref{thm:intromain}.

\begin{thm} \label{thm:liebeck}
Let $G$ be a connected semisimple group over an algebraically closed field $k$ and let $H$ be a $G$-irreducible subgroup. Then there is a constant $c \leq 197$ such that $\# \cent{G}{H} \leq c^{\rank G}\# Z(G)$.
\end{thm}

By \cite{liebeck}, the constant $c$ is at most 16 if all the simple factors of $G$ are classical, and \cite[Lemma 2.5]{liebeck} gives precise bounds for $c$ when $G$ is of exceptional type. 

\begin{cor} \label{cor:liebeck}
Let $G$ be a connected reductive group over an algebraically closed field $k$ and let $H$ be a $G$-irreducible subgroup. Then 
$\# \pi_0 \cent{G}{H} \leq c^{\rank G^{\ad}} \# \pi_0 Z(G)$.
\end{cor}

\begin{proof}
Consider the exact sequence
\[ 1 \to Z(G) \to \cent{G}{H} \to \cent{G^{\ad}}{\overline{H}} \] 
where $\overline{H} \subset G^{\ad}$ is the image of $H$ under the natural map $G \to G^{\ad}$. 
Note that $\overline{H}$ is $G^{\ad}$-irreducible, so Theorem \ref{thm:liebeck} implies that 
$\# \cent{G^{\ad}}{\overline{H}} \leq c^{\rank G^{\ad}}$. On the other hand, since $H$ is $G$-irreducible, $\cent{G}{H}^0 = Z(G)^0$, so the above exact sequence implies that $\pi_0 \cent{G}{H}/\pi_0 Z(G)$ injects into $\cent{G^{\ad}}{\overline{H}}$. Therefore, 
$\# \pi_0 \cent{G}{H} \leq \# \pi_0 Z(G) c^{\rank G^{\ad}}$.
\end{proof}

\begin{lem}\label{lem:injection into weyl groups}
Let $H$ be a closed, completely reducible subgroup of $G$. Choose a maximal torus $S$ of $K \colonequals C_G(H)^0$ and let $L=C_G(S)$. 
Denote by $W(G,L)$ (resp. $W(K,S)$) the quotient $N_G(L)/L$ (resp. $N_{K}(S)/S$); the latter can be naturally identified as a subgroup of the former.
Then there is a canonical injection
$$
N_G(H)/N_L(H)\cdot K \hookrightarrow N_{W(G,L)}(W(K,S))/W(K,S).
$$
\end{lem}
\begin{proof}
Observe that $Z_L^0=S$, which implies $N_G(L)=N_G(S)$. So $N_G(L)/L=\nlz{G}{S}/\cent{G}{S}$ contains $\nlz{K}{S}/\cent{K}{S}=N_{K}(S)/S$ as a subgroup. 
Any $n \in N_G(H)$ normalizes $K=C_G(H)^0$, so $S^n$ is another maximal tori of $K$. By the conjugacy of maximal tori, there is an element $c \in K$ such that $(S^n)^c=S$, i.e. $nc \in N_G(S)=N_G(L)$. For any element $w \in N_G(L)$, denote by $\bar w$ its image in $W(G,L)$. Define a map 
\[
\varphi: N_G(H) \to N_{W(G,L)}(W(K,S))/W(K,S)
\]
by $\varphi(n)=\overline{nc}W(K,S)$. 
First note that $n \mapsto \overline{nc} W(K,S)$ is a well-defined map from $N_{G}(H)$ to the quotient set $W(G,L)/W(K,S)$, and that $nc$ normalizes $\nlz{K}{S}$, which implies $\overline{nc}$ normalizes $W(K,S)$. It follows that $\varphi$ defined above is a well-defined map.
We check that $\varphi$ is a group homomorphism. Let $n_1, n_2 \in N_{G}(H)$ with $c_1, c_2 \in K$ such that $n_1c_1, n_2c_2$ both normalize $S$. So $\varphi(n_i)=\overline{n_ic_i}W(K,S)$ for $i=1,2$. 
On the other hand, the product $(n_1c_1)(n_2c_2)$ normalizes $S$ and we have $(n_1c_1)(n_2c_2)=(n_1n_2)(n_2^{-1}c_1n_2)c_2$ with $(n_2^{-1}c_1n_2)c_2 \in K$, so $\varphi(n_1n_2)=\overline{(n_1n_2)(n_2^{-1}c_1n_2)c_2} W(K,S)=\overline{(n_1c_1)(n_2c_2)}W(K,S)=\overline{n_1c_1}W(K,S)\overline{n_2c_2}W(K,S)=\varphi(n_1)\varphi(n_2)$. 
We now compute $\Ker \varphi$. It consists of $n \in \nlz{G}{H}$ for which 
$\overline{nc} \in W(K,S)=\nlz{K}{S}/S=N_{K}(S)/C_{K}(S)=N_{K}(S)/N_{K}(S) \cap C_{G}(S)=N_{K}(S) \cdot L/L$ for some $c \in K$. So $nc \in N_{K}(S) \cdot L$, which implies $n \in K \cdot L$. As $n$ normalizes $H$ and $K$ commutes with $H$, it follows that $n \in K \cdot \nlz{L}{H}$. On the other hand, if $n=b \cdot c \in \nlz{L}{H} \cdot K$, then $\varphi(n)=\overline{nc^{-1}}W(K,S)=\overline{b} W(K,S)=W(K,S)$, where the first equality holds since $S^{b}=S$ ($b \in L=\cent{G}{S}$), and the last equality holds since $b \in \nlz{L}{H} \subset L$ and hence $\overline b=\overline 1 \in W(G,L)$. Thus, $\Ker \varphi = K \cdot \nlz{L}{H} = \nlz{L}{H} \cdot K$ (the last equality holds since $\nlz{L}{H}$ normalizes $K=\cent{G}{H}^0$). 
\end{proof}

\begin{thm} \label{thm:bound cpnt groups}
Let $G$ be a connected reductive group over an algebraically closed field $k$ and let $H$ be a $G$-completely reducible subgroup. 
Then the size of $\pi_0 \cent{G}{H}$ is bounded by 
\[
c_G \colonequals \# W_G \cdot \sup_{L \subset G} \left\{ c^{\rank L^{\ad}} \cdot \# \pi_0 Z(L) \right\}
\]
where $L$ runs over the finitely many conjugacy classes of Levi subgroups of $G$.
\end{thm}

\begin{proof}
Let $L$ be as in Lemma \ref{lem:injection into weyl groups}. Note that $H$ is an irreducible subgroup of $L$. 
There is a natural surjection from $\pi_0(C_G(H))$ to $C_G(H)/C_L(H)\cdot C_G(H)^0$, the latter injects into the group on the left side of the inclusion in Lemma \ref{lem:injection into weyl groups}, and the kernel is a quotient of $\pi_0 C_L(H)$. The theorem now follows from Corollary \ref{cor:liebeck} and Lemma \ref{lem:injection into weyl groups}. 
\end{proof}

\begin{example} \label{example:g2 bound}
The constant $c_G$ is a very crude bound, which can be made much smaller for specific $G$ by going through the proof of Theorem \ref{thm:bound cpnt groups}. For example, for $G$ the exceptional group of type $G_2$, if $H$ is not $G$-ir, then the minimal Levi containing $H$ (unique up to conjugacy) is either a maximal torus or isomorphic to $\GL_2$, and $\pi_0 \cent{G}{H}$ injects into the group on the left side of the inclusion in Lemma \ref{lem:injection into weyl groups}, so $\#\pi_0 \cent{G}{H} \leq W_G=12$. Now if $H$ is $G$-ir, then $\pi_0 \cent{G}{H} \leq 8.5^2=72.25$ by Theorem \ref{thm:liebeck} and \cite[Lemma 2.5]{liebeck}. Thus, the bound $c_G$ can be improved to $72$ in this case. %
\end{example}

For the remaining of this section, we will generalize Theorem \ref{thm:bound cpnt groups} to non-connected reductive groups in the case that $H$ is solvable. In what follows, we make no attempt in optimizing the bounds for the component groups. 

\begin{lemma} \label{lem:cpnt group isogeny}
Let $f \colon G' \to G$ be a central isogeny of connected reductive groups over an algebraically closed field $k$, and let $\lambda'$ and $\lambda$ be compatible automorphisms of $G'$ and $G$, respectively. Then 
$\# \pi_0 \cent{G}{\lambda} \leq \# \pi_0 \cent{G'}{\lambda'} \cdot \# \ker f$.
\end{lemma}

\begin{proof}
Note that there is an exact sequence 
\[
1 \to \cent{G'}{\lambda'} \to f^{-1}(\cent{G}{\lambda}) \to \ker f
\]
where the rightmost map is $x \mapsto \lambda'(x) x^{-1}$. The component group of $f^{-1}(\cent{G}{\lambda})$ is thus an extension of a subgroup of $(\ker f)_{\mathrm{red}}$ by a quotient of the component group of $\cent{G'}{\lambda'}$. 

Moreover, there is an obvious exact sequence
\[
1 \to \ker f \to f^{-1}(\cent{G}{\lambda}) \to \cent{G}{\lambda} \to 1
\]
Thus the component group of $\cent{G}{\lambda}$ is a quotient of the component group of $f^{-1}(\cent{G}{\lambda})$. The lemma follows.  
\end{proof}

\begin{lemma} \label{lem:autom of tori}
Let $T$ be a split torus and let $\lambda$ be an automorphism of $T$ of order $n$. Then $\# \pi_0 \cent{T}{\lambda}$ divides $n^{\dim T}$.
\end{lemma}

\begin{proof}
This follows immediately from \cite[Lemma 1.2(1)]{dm18}.
\end{proof}

\begin{lemma} \label{lem:lambda bound}
Let $G$ be a connected reductive group over an algebraically closed field $k$. Let $\lambda$ be a semisimple automorphism of $G$ of order $n$. Then $\# \pi_0 \cent{G}{\lambda} \leq 4^{\rank \sD(G)} \cdot n^{\dim Z(G)}$. 
\end{lemma}

\begin{proof}
By a theorem of Steinberg, if $G$ is semisimple, then $\pi_0 \cent{G}{\lambda}$ can be identified with a subgroup of $\pi_1(G)$. Moreover, $\# \pi_1(G) \leq 2^{\rank G}$ (realized by $G=(\PGL_2)^{\rank G}$). 
Now we suppose that $G = T \times H$ for a torus $T$ and a semisimple group $H$. Then $\lambda$ induces acts on $T$ and $H$. By the above, $\# \pi_0 \cent{H}{\lambda} \leq 2^{\rank H}=2^{\rank \sD(G)}$. By Lemma \ref{lem:autom of tori}, $\# \pi_0 \cent{T}{\lambda} \leq n^{\dim T} = n^{\dim Z(G)}$. So $\# \pi_0 \cent{G}{\lambda} \leq 2^{\rank \sD(G)} \cdot n^{\dim Z(G)}$.

For the general case, note the canonical isogeny $Z(G)^0 \times \sD(G) \to G$, whose kernel is contained in $Z(\sD(G))$. By Lemma \ref{lem:cpnt group isogeny} and the above, it follows that 
\[ \# \pi_0 \cent{G}{\lambda} \leq 2^{\rank \sD(G)} \cdot n^{\dim Z(G)} \cdot \#Z(\sD(G)). \]
Finally, note that if $H$ is semisimple, then $\#Z(H) \leq \#Z(H^{\mathrm{sc}})=\#\pi_1 H^{\mathrm{ad}} \leq 2^{\rank H}$. So $\#Z(\sD(G)) \leq 2^{\rank \sD(G)}$. The lemma follows. 
\end{proof}

\begin{lemma} \label{lem:F bound}
Let $H$ be a (possibly nonconnected) reductive group over an algebraically closed field $k$. Let $F$ be a solvable finite group with prime-to-$p$ order acting on $H$. Then the size of $\pi_0 \cent{H}{F}$ is bounded by a constant depending only on $\rank H^0$, $\#\pi_0 H$ and $\# F$. 
\end{lemma}

\begin{proof}
This will follow from the preceding lemma and induction. 
First note the exact sequence
\[
1 \to \cent{H^0}{F} \to \cent{H}{F} \to \pi_0(H)
\]
and that $\cent{H^0}{F}^0=\cent{H}{F}^0$. It follows that 
$\# \pi_0 \cent{H}{F} \leq \# \pi_0 H \cdot \# \pi_0 \cent{H^0}{F}$.

Since $F$ is solvable, there is a composition series
\[
F=F_0 \supset F_1 \supset \cdots \supset F_n=\{1\}
\]
such that $F_{i+1}$ is normal in $F_i$ and $F_i/F_{i+1}$ is cyclic. 
Since $F_{i+1}$ is normal in $F_i$, $F_i$ acts on $\cent{H}{F_{i+1}}$, and in fact
$\cent{H}{F_i}=\cent{\cent{H}{F_{i+1}}}{F_i}=\cent{\cent{H}{F_{i+1}}}{F_i/F_{i+1}}$. 
So 
\[
\# \pi_0 \cent{H}{F_i} \leq \# \pi_0 \cent{H}{F_{i+1}} \cdot \# \pi_0 \cent{\cent{H}{F_{i+1}}^0}{F_i/F_{i+1}}
\]
and hence
\[
\# \pi_0 \cent{H}{F} \leq \prod_{0 \leq i < n} \# \pi_0 \cent{\cent{H}{F_{i+1}}^0}{F_i/F_{i+1}} \cdot \# \pi_0 H.
\]
We now conclude by applying Lemma \ref{lem:lambda bound} to the group $\cent{\cent{H}{F_{i+1}}^0}{F_i/F_{i+1}}$.
\end{proof}

\begin{thm} \label{thm:bound nonconnected}
Let $G$ be a (possibly nonconnected) reductive group over an algebraically closed field $k$. Let $\Ld \subset G(k)$ be a solvable finite group with prime-to-$p$ order. Then the size of $\pi_0 \cent{G}{\Ld}$ is bounded by a constant depending only on $c_{G^0}$, $\rank G^0$ and $\# \pi_0 G$. In the special case when $G$ is connected, $\# \pi_0 \cent{G}{\Ld} \leq c_G$ by Theorem \ref{thm:bound cpnt groups}.
\end{thm}

\begin{proof}
By the first part of the proof of Lemma \ref{lem:F bound}, we have
$\# \pi_0 \cent{G}{\Ld} \leq \# \pi_0 \cent{G^0}{\Ld} \cdot  \# \pi_0 G$. 
Let $\Ld'$ be $\Ld \cap G^0(k)$, so that $\Ld/\Ld' \into G/G^0$.
Note that $\cent{G^0}{\Ld}=\cent{\cent{G^0}{\Ld'}}{\Ld/\Ld'}$. By Theorem \ref{thm:bound cpnt groups} (recall $\Ld$ has prime-to-$p$ order by assumption), the size of $\pi_0 \cent{G^0}{\Ld'}$ is at most $c_{G^0}$. Applying Lemma \ref{lem:F bound} with $H=\cent{G^0}{\Ld'}$ and $F=\Ld/\Ld'$ yields a bound for $\# \pi_0 \cent{G^0}{\Ld}$ in terms of 
$c_{G^0}$, $\rank G^0$ and $\# \pi_0 G$.
\end{proof}

\section{A curious application}\label{section:appendix-b}

This section is not relevant to the main aims of this paper, but it follows from the methods developed here, so we will give a (terse) proof.  

\begin{theorem}
Let $G$ be a connected reductive group over a field $k$, and let $\Lambda \subset G(k)$ be a finite subgroup of order $n$, prime to $\chara k$. The only prime numbers dividing the order of $\pi_0 C_G(\Lambda)$ also divide $n$.
\end{theorem}

\begin{proof}
We may extend $k$ to assume that $G$ is split, and if $k$ is of positive characteristic we may then lift $\Lambda$ to characteristic $0$. Using Corollary~\ref{cor:double-centralizer}, it therefore suffices to assume that $\chara k = 0$. By spreading out and specializing, we may and do assume that $k$ is a number field. Let $p$ be a prime number not dividing $n$, and let $v$ be a place of $k$ dividing $p$. The argument of \cite[Lemma A.8]{Griess-Ryba} shows that after passing to a finite extension of $k$ and conjugating, we may assume $\Lambda \subset \sG(\sO_v)$, where $\sG$ is the split model of $G$ over $\bZ$ and $\sO_v$ is the ring of integers of the completion of $k$ at $v$. But now Corollary~\ref{corollary:centralizer-of-weakly-reductive} shows that $C_{\sG}(\Lambda)$ is weakly reductive, so in particular $p$ does not divide the order of $\pi_0 C_G(\Lambda)$, as desired.
 \end{proof}

 \section{A Proof of Corollary~\ref{corollary:lifting}} \label{appendix-corollary}

\begin{proof}
It is enough to check the third bulleted point in \cite[Theorem A]{fkp21}. By Theorem \ref{thm:intromain}, for $v \neq p$, $\rhobar|_{\gma{\QQ_v}}$ has a $p$-adic lift. On the other hand, by \cite[Theorem C]{lin_g2}, $\rhobar_p \colonequals \rhobar|_{\gma{\QQ_p}}$ has a crystalline lift $\rho_p \colon \gma{\QQ_p} \to G(\Zpbar)$ for $p>3$. We claim that the Hodge--Tate cocharacter of the lift can be chosen to be regular. If $\rhobar_p$ is irreducible, this follows from \cite[Theorem 2]{lin2020}. Otherwise, $\rhobar_p$ factors through a maximal parabolic subgroup $P$ of $G$ with Levi factor $L \cong \GL_2$, and there is a corresponding representation $\bar r_p \colon \gma{\QQ_p} \to L(\Fpbar)$. By \cite[\S 7.2.1-7.2.2]{lin_g2}, $r_p$ has a crystalline lift $r^o$ with regular Hodge--Tate cocharacter such that $\phi^{\Lie}(r^o)$ has Hodge--Tate weights slightly less than $\underline{0}$ (in the terminology of loc. cit.). Now the second half of \cite[Theorem C]{lin_g2} implies that $\rho_p$ can be chosen such that it factors through $P$ and its associated $L$-valued representation lies on the same irreducible component of the spectrum of the crystalline lifting ring that $r^o$ does; in particular, its Hodge--Tate cocharacter is the same as that of $r^o$. Thus \cite[Theorem A]{fkp21} gives the desired lift of $\rhobar$.
\end{proof}

\begin{remark}
The lower bound for $p$ in Corollary~\ref{corollary:lifting} has to do with the global lifting theorem \cite[Theorem A]{fkp21}, the local lifting theorem in the $\ell=p$ case \cite[Theorem C]{lin_g2}, and Theorem \ref{thm:intromain}. The bound for \cite[Theorem A]{fkp21} can be made explicit, see \cite[Remark 6.17]{fkp21}. The bound for \cite[Theorem C]{lin_g2} is 3, and the bound for Theorem \ref{thm:intromain} in the $G_2$ case is 72 (Remark \ref{rmk:bound for G}).  
\end{remark}

\bibliographystyle{halpha-abbrv}
\bibliography{isotypic}

\newcommand{\etalchar}[1]{$^{#1}$}
\begin{thebibliography}{DHKM24}
\expandafter\ifx\csname url\endcsname\relax
  \def\url#1{\texttt{#1}}\fi
\expandafter\ifx\csname doi\endcsname\relax
  \def\doi#1{\burlalt{doi:#1}{http://dx.doi.org/#1}}\fi
\expandafter\ifx\csname urlprefix\endcsname\relax\def\urlprefix{URL }\fi
\expandafter\ifx\csname href\endcsname\relax
  \def\href#1#2{#2}\fi
\expandafter\ifx\csname burlalt\endcsname\relax
  \def\burlalt#1#2{\href{#2}{#1}}\fi

\bibitem[AHR23]{Alper-Hall-Rydh}
J.~Alper, J.~Hall, and D.~Rydh.
\newblock The \'etale local structure of algebraic stacks, 2023,
  \burlalt{1912.06162}{http://arxiv.org/abs/1912.06162}.
\newblock \urlprefix\url{https://arxiv.org/abs/1912.06162}.

\bibitem[AJ84]{Andersen-Jantzen}
H.~H. Andersen and J.~C. Jantzen.
\newblock Cohomology of induced representations for algebraic groups.
\newblock {\em Math. Ann.}, 269(4):487--525, 1984.
\newblock \doi{10.1007/BF01450762}.

\bibitem[Alp14]{Alper-adequate}
J.~Alper.
\newblock Adequate moduli spaces and geometrically reductive group schemes.
\newblock {\em Algebr. Geom.}, 1(4):489--531, 2014.
\newblock \doi{10.14231/AG-2014-022}.

\bibitem[AOV08]{AOV}
D.~Abramovich, M.~Olsson, and A.~Vistoli.
\newblock Tame stacks in positive characteristic.
\newblock {\em Ann. Inst. Fourier (Grenoble)}, 58(4):1057--1091, 2008.
\newblock \urlprefix\url{http://aif.cedram.org/item?id=AIF_2008__58_4_1057_0}.

\bibitem[Art74]{Artin-stacks}
M.~Artin.
\newblock Versal deformations and algebraic stacks.
\newblock {\em Invent. Math.}, 27:165--189, 1974.
\newblock \doi{10.1007/BF01390174}.

\bibitem[BFH{\etalchar{+}}22]{boeckle-cyclic}
G.~B{\"o}ckle, T.~Feng, M.~Harris, C.~Khare, and J.~A. Thorne.
\newblock Cyclic base change of cuspidal automorphic representations over
  function fields.
\newblock {\em arXiv preprint arXiv:2205.04499}, 2022.

\bibitem[BG19]{bg19}
R.~Bellovin and T.~Gee.
\newblock {$G$}-valued local deformation rings and global lifts.
\newblock {\em Algebra Number Theory}, 13(2):333--378, 2019.
\newblock \doi{10.2140/ant.2019.13.333}.

\bibitem[BHKT19]{BHKT}
G.~B\"{o}ckle, M.~Harris, C.~Khare, and J.~A. Thorne.
\newblock {$\hat G$}-local systems on smooth projective curves are potentially
  automorphic.
\newblock {\em Acta Math.}, 223(1):1--111, 2019.
\newblock \doi{10.4310/ACTA.2019.v223.n1.a1}.

\bibitem[BMR05]{BMR05}
M.~Bate, B.~Martin, and G.~R\"{o}hrle.
\newblock A geometric approach to complete reducibility.
\newblock {\em Invent. Math.}, 161(1):177--218, 2005.
\newblock \doi{10.1007/s00222-004-0425-9}.

\bibitem[Boo19a]{Booher19}
J.~Booher.
\newblock Minimally ramified deformations when {$\ell \neq p$}.
\newblock {\em Compos. Math.}, 155(1):1--37, 2019.
\newblock \doi{10.1112/S0010437X18007546}.

\bibitem[Boo19b]{booher19jnt}
J.~Booher.
\newblock Producing geometric deformations of orthogonal and symplectic
  {G}alois representations.
\newblock {\em J. Number Theory}, 195:115--158, 2019.
\newblock \doi{10.1016/j.jnt.2018.05.022}.

\bibitem[BP19]{bp19}
J.~Booher and S.~Patrikis.
\newblock {$G$}-valued {G}alois deformation rings when {$\ell\neq p$}.
\newblock {\em Math. Res. Lett.}, 26(4):973--990, 2019.
\newblock \doi{10.4310/MRL.2019.v26.n4.a2}.

\bibitem[Bri21]{Brion}
M.~Brion.
\newblock Homomorphisms of algebraic groups: representability and rigidity,
  2021, \burlalt{math/2101.12460}{http://arxiv.org/abs/math/2101.12460}.

\bibitem[Car85]{carter85}
R.~W. Carter.
\newblock {\em Finite groups of {L}ie type}.
\newblock Pure and Applied Mathematics (New York). John Wiley \& Sons, Inc.,
  New York, 1985.
\newblock Conjugacy classes and complex characters, A Wiley-Interscience
  Publication.

\bibitem[CHT08]{cht08}
L.~Clozel, M.~Harris, and R.~Taylor.
\newblock Automorphy for some {$l$}-adic lifts of automorphic mod {$l$}
  {G}alois representations.
\newblock {\em Publ. Math. Inst. Hautes \'{E}tudes Sci.}, (108):1--181, 2008.
\newblock \doi{10.1007/s10240-008-0016-1}.
\newblock With Appendix A, summarizing unpublished work of Russ Mann, and
  Appendix B by Marie-France Vign\'{e}ras.

\bibitem[Con14]{conrad}
B.~Conrad.
\newblock Reductive group schemes.
\newblock In {\em Autour des sch\'{e}mas en groupes. {V}ol. {I}}, volume 42/43
  of {\em Panor. Synth\`eses}, pages 93--444. Soc. Math. France, Paris, 2014.

\bibitem[Cot22a]{cotner}
S.~Cotner.
\newblock Centralizers of sections of a reductive group scheme, 2022.
\newblock \doi{10.48550/ARXIV.2203.15133}.

\bibitem[Cot22b]{Integral-Springer}
S.~Cotner.
\newblock Springer isomorphisms over a general base scheme, 2022.

\bibitem[CR81]{CR81}
C.~W. Curtis and I.~Reiner.
\newblock {\em Methods of representation theory. {V}ol. {I}}.
\newblock Pure and Applied Mathematics. John Wiley \& Sons, Inc., New York,
  1981.
\newblock With applications to finite groups and orders.

\bibitem[Dem15]{Demarche}
C.~Demarche.
\newblock Cohomologie de {H}ochschild non {A}b\'{e}lienne et extensions de
  {F}aisceaux en groupes.
\newblock In {\em Autours des sch\'{e}mas en groupes. {V}ol. {II}}, volume~46
  of {\em Panor. Synth\`eses}, pages 255--292. Soc. Math. France, Paris, 2015.

\bibitem[DG70]{DG}
M.~Demazure and P.~Gabriel.
\newblock {\em Groupes alg\'{e}briques. {T}ome {I}: {G}\'{e}om\'{e}trie
  alg\'{e}brique, g\'{e}n\'{e}ralit\'{e}s, groupes commutatifs}.
\newblock Masson \& Cie, \'{E}diteurs, Paris; North-Holland Publishing Co.,
  Amsterdam, 1970.
\newblock Avec un appendice {{\i}t Corps de classes local} par Michiel
  Hazewinkel.

\bibitem[DHKM24]{DHKM}
J.-F. Dat, D.~Helm, R.~Kurinczuk, and G.~Moss.
\newblock Moduli of langlands parameters, 2024,
  \burlalt{2009.06708}{http://arxiv.org/abs/2009.06708}.
\newblock \urlprefix\url{https://arxiv.org/abs/2009.06708}.

\bibitem[DM18]{dm18}
F.~Digne and J.~Michel.
\newblock Quasi-semisimple elements.
\newblock {\em Proc. Lond. Math. Soc. (3)}, 116(5):1301--1328, 2018.
\newblock \doi{10.1112/plms.12121}.

\bibitem[EGA]{EGA}
J.~Dieudonn{\'e} and A.~Grothendieck.
\newblock \'{E}l\'ements de g\'eom\'etrie alg\'ebrique.
\newblock {\em Inst. Hautes \'Etudes Sci. Publ. Math.}, 4, 8, 11, 17, 20, 24,
  28, 32, 1961--1967.

\bibitem[FKP21]{fkp21}
N.~Fakhruddin, C.~Khare, and S.~Patrikis.
\newblock Relative deformation theory, relative {S}elmer groups, and lifting
  irreducible {G}alois representations.
\newblock {\em Duke Math. J.}, 170(16):3505--3599, 2021.
\newblock \doi{10.1215/00127094-2021-0003}.

\bibitem[FKP22]{fkp22}
N.~Fakhruddin, C.~Khare, and S.~Patrikis.
\newblock Lifting and automorphy of reducible mod {$p$} {G}alois
  representations over global fields.
\newblock {\em Invent. Math.}, 228(1):415--492, 2022.
\newblock \doi{10.1007/s00222-021-01085-7}.

\bibitem[FS24]{Fargues-Scholze}
L.~Fargues and P.~Scholze.
\newblock Geometrization of the local {}anglands correspondence, 2024,
  \burlalt{2102.13459}{http://arxiv.org/abs/2102.13459}.
\newblock \urlprefix\url{https://arxiv.org/abs/2102.13459}.

\bibitem[GR98]{Griess-Ryba}
R.~L. Griess, Jr. and A.~J.~E. Ryba.
\newblock Embeddings of {${\mathrm PGL}_2(31)$} and {${\mathrm SL}_2(32)$} in
  {$E_8(\bold C)$}.
\newblock {\em Duke Math. J.}, 94(1):181--211, 1998.
\newblock \doi{10.1215/S0012-7094-98-09409-1}.
\newblock With appendices by Michael Larsen and J.-P. Serre.

\bibitem[Har18]{Hardesty}
W.~Hardesty.
\newblock On the centralizer of a balanced nilpotent section, 2018,
  \burlalt{1810.06157}{http://arxiv.org/abs/1810.06157}.

\bibitem[Her13]{herpel}
S.~Herpel.
\newblock On the smoothness of centralizers in reductive groups.
\newblock {\em Trans. Amer. Math. Soc.}, 365(7):3753--3774, 2013.
\newblock \doi{10.1090/S0002-9947-2012-05745-X}.

\bibitem[hm]{McNinch-Mathoverflow}
G.~M. (https://mathoverflow.net/users/4653/george mcninch).
\newblock Is the normalizer of a reductive subgroup reductive?
\newblock MathOverflow,
  \burlalt{https://mathoverflow.net/q/114419}{http://arxiv.org/abs/https://mathoverflow.net/q/114419}.
\newblock \urlprefix\url{https://mathoverflow.net/q/114419}.
\newblock URL:https://mathoverflow.net/q/114419 (version: 2012-11-25).

\bibitem[HR08]{hr08}
S.~Hamblen and R.~Ramakrishna.
\newblock Deformations of certain reducible {G}alois representations. {II}.
\newblock {\em Amer. J. Math.}, 130(4):913--944, 2008.
\newblock \doi{10.1353/ajm.0.0008}.

\bibitem[Jan97]{Jantzen-semisimple}
J.~C. Jantzen.
\newblock Low-dimensional representations of reductive groups are semisimple.
\newblock In {\em Algebraic groups and {L}ie groups}, volume~9 of {\em Austral.
  Math. Soc. Lect. Ser.}, pages 255--266. Cambridge Univ. Press, Cambridge,
  1997.

\bibitem[Jan03]{Jantzen}
J.~C. Jantzen.
\newblock {\em Representations of algebraic groups}, volume 107 of {\em
  Mathematical Surveys and Monographs}.
\newblock American Mathematical Society, Providence, RI, second edition, 2003.

\bibitem[Jan04]{jantzen04}
J.~C. Jantzen.
\newblock Nilpotent orbits in representation theory.
\newblock In {\em Lie theory}, volume 228 of {\em Progr. Math.}, pages 1--211.
  Birkh\"{a}user Boston, Boston, MA, 2004.

\bibitem[Knu71]{Knutson-algebraic-spaces}
D.~Knutson.
\newblock {\em Algebraic spaces}.
\newblock Lecture Notes in Mathematics, Vol. 203. Springer-Verlag, Berlin-New
  York, 1971.

\bibitem[Lie23]{liebeck}
M.~W. Liebeck.
\newblock A bound for the orders of centralizers of irreducible subgroups of
  algebraic groups.
\newblock {\em J. Group Theory}, 26(4):795--801, 2023.
\newblock \doi{10.1515/jgth-2022-0111}.

\bibitem[Lin20a]{lin2020}
Z.~Lin.
\newblock Crystalline lifts and a variant of the steinberg-winter theorem.
\newblock {\em arXiv preprint arXiv:2011.08766}, 2020.

\bibitem[Lin20b]{lin_g2}
Z.~Lin.
\newblock Lyndon-demushkin method and crystalline lifts of galois
  representations valued in the exceptional group $ g\_2 $ and classical
  groups.
\newblock {\em arXiv preprint arXiv:2011.08773}, 2020.

\bibitem[Mar03]{martin03}
B.~M.~S. Martin.
\newblock Reductive subgroups of reductive groups in nonzero characteristic.
\newblock {\em J. Algebra}, 262(2):265--286, 2003.
\newblock \doi{10.1016/S0021-8693(03)00189-3}.

\bibitem[Mat89]{Matsumura}
H.~Matsumura.
\newblock {\em Commutative ring theory}, volume~8 of {\em Cambridge Studies in
  Advanced Mathematics}.
\newblock Cambridge University Press, Cambridge, second edition, 1989.
\newblock Translated from the Japanese by M. Reid.

\bibitem[Maz89]{mazur87}
B.~Mazur.
\newblock Deforming {G}alois representations.
\newblock In {\em Galois groups over {${\bf Q}$} ({B}erkeley, {CA}, 1987)},
  volume~16 of {\em Math. Sci. Res. Inst. Publ.}, pages 385--437. Springer, New
  York, 1989.
\newblock \doi{10.1007/978-1-4613-9649-9\_7}.

\bibitem[Maz97]{mazur95}
B.~Mazur.
\newblock An introduction to the deformation theory of {G}alois
  representations.
\newblock In {\em Modular forms and {F}ermat's last theorem ({B}oston, {MA},
  1995)}, pages 243--311. Springer, New York, 1997.

\bibitem[McN98]{McNinch-semisimple}
G.~J. McNinch.
\newblock Dimensional criteria for semisimplicity of representations.
\newblock {\em Proc. London Math. Soc. (3)}, 76(1):95--149, 1998.
\newblock \doi{10.1112/S0024611598000045}.

\bibitem[McN08]{mcninch08}
G.~J. McNinch.
\newblock The centralizer of a nilpotent section.
\newblock {\em Nagoya Math. J.}, 190:129--181, 2008.
\newblock \doi{10.1017/S0027763000009594}.

\bibitem[Pat16]{patrikis16}
S.~Patrikis.
\newblock Deformations of {G}alois representations and exceptional monodromy.
\newblock {\em Invent. Math.}, 205(2):269--336, 2016.
\newblock \doi{10.1007/s00222-015-0635-3}.

\bibitem[Pre03]{premet03}
A.~Premet.
\newblock Nilpotent orbits in good characteristic and the {K}empf-{R}ousseau
  theory.
\newblock volume 260, pages 338--366. 2003.
\newblock \doi{10.1016/S0021-8693(02)00662-2}.
\newblock Special issue celebrating the 80th birthday of Robert Steinberg.

\bibitem[PY02]{Prasad-Yu}
G.~Prasad and J.-K. Yu.
\newblock On finite group actions on reductive groups and buildings.
\newblock {\em Invent. Math.}, 147(3):545--560, 2002.
\newblock \doi{10.1007/s002220100182}.

\bibitem[Ram02]{ram02}
R.~Ramakrishna.
\newblock Deforming {G}alois representations and the conjectures of {S}erre and
  {F}ontaine-{M}azur.
\newblock {\em Ann. of Math. (2)}, 156(1):115--154, 2002.
\newblock \doi{10.2307/3597186}.

\bibitem[Ray70]{Faisceaux-amples}
M.~Raynaud.
\newblock {\em Faisceaux amples sur les sch\'{e}mas en groupes et les espaces
  homog\`enes}.
\newblock Lecture Notes in Mathematics, Vol. 119. Springer-Verlag, Berlin-New
  York, 1970.

\bibitem[Rom22]{Romagny}
M.~Romagny.
\newblock Algebraicity and smoothness of fixed point stacks, 2022,
  \burlalt{2205.11114}{http://arxiv.org/abs/2205.11114}.
\newblock \urlprefix\url{https://arxiv.org/abs/2205.11114}.

\bibitem[Ser03]{Serre03}
J.~P. Serre.
\newblock 1998 {M}oursund lectures at the {U}niversity of {O}regon, 2003,
  \burlalt{math/0305257}{http://arxiv.org/abs/math/0305257}.
\newblock \urlprefix\url{https://arxiv.org/abs/math/0305257}.

\bibitem[Ser05]{Serre-cr}
J.-P. Serre.
\newblock Compl\`ete r\'{e}ductibilit\'{e}.
\newblock Number 299, pages Exp. No. 932, viii, 195--217. 2005.
\newblock S\'{e}minaire Bourbaki. Vol. 2003/2004.

\bibitem[SGA3]{sga3}
P.~Gille and P.~Polo, editors.
\newblock {\em Sch\'{e}mas en groupes ({SGA} 3). {T}ome {III}. {S}tructure des
  sch\'{e}mas en groupes r\'{e}ductifs}, volume~8 of {\em Documents
  Math\'{e}matiques (Paris) [Mathematical Documents (Paris)]}.
\newblock Soci\'{e}t\'{e} Math\'{e}matique de France, Paris, 2011.
\newblock S\'{e}minaire de G\'{e}om\'{e}trie Alg\'{e}brique du Bois Marie
  1962--64. [Algebraic Geometry Seminar of Bois Marie 1962--64], A seminar
  directed by M. Demazure and A. Grothendieck with the collaboration of M.
  Artin, J.-E. Bertin, P. Gabriel, M. Raynaud and J-P. Serre, Revised and
  annotated edition of the 1970 French original.

\bibitem[Sho18]{shotton18}
J.~Shotton.
\newblock The {B}reuil-{M}\'{e}zard conjecture when {$l\neq p$}.
\newblock {\em Duke Math. J.}, 167(4):603--678, 2018.
\newblock \doi{10.1215/00127094-2017-0039}.

\bibitem[Sho22]{shotton20}
J.~Shotton.
\newblock Generic local deformation rings when {$l\ne p$}.
\newblock {\em Compos. Math.}, 158(4):721--749, 2022.
\newblock \doi{10.1112/s0010437x22007461}.

\bibitem[Sho23]{shotton23}
J.~Shotton.
\newblock Irreducible components of the moduli space of langlands parameters,
  2023, \burlalt{2302.10125}{http://arxiv.org/abs/2302.10125}.

\bibitem[Spr66]{Springer}
T.~A. Springer.
\newblock Some arithmetical results on semi-simple {L}ie algebras.
\newblock {\em Inst. Hautes \'{E}tudes Sci. Publ. Math.}, (30):115--141, 1966.

\bibitem[SS70]{Springer-Steinberg}
T.~A. Springer and R.~Steinberg.
\newblock Conjugacy classes.
\newblock In {\em Seminar on {A}lgebraic {G}roups and {R}elated {F}inite
  {G}roups ({T}he {I}nstitute for {A}dvanced {S}tudy, {P}rinceton, {N}.{J}.,
  1968/69)}, Lecture Notes in Mathematics, Vol. 131, pages 167--266. Springer,
  Berlin, 1970.

\bibitem[{Sta}21]{stacks-project}
T.~{Stacks project authors}.
\newblock The stacks project.
\newblock \url{https://stacks.math.columbia.edu}, 2021.

\bibitem[SX10]{Scott-Xi}
L.~Scott and N.~Xi.
\newblock Some non-trivial {K}azhdan-{L}usztig coefficients of an affine {W}eyl
  group of type {$A_n$}.
\newblock {\em Sci. China Math.}, 53(8):1919--1930, 2010.
\newblock \doi{10.1007/s11425-010-4041-4}.

\bibitem[Tan19]{tang19}
S.~Tang.
\newblock Algebraic monodromy groups of {$l$}-adic representations of {${\rm
  Gal}(\overline{\Bbb Q}/\Bbb Q)$}.
\newblock {\em Algebra Number Theory}, 13(6):1353--1394, 2019.
\newblock \doi{10.2140/ant.2019.13.1353}.

\bibitem[Til96]{tilouine96}
J.~Tilouine.
\newblock {\em Deformations of {G}alois representations and {H}ecke algebras}.
\newblock Published for The Mehta Research Institute of Mathematics and
  Mathematical Physics, Allahabad; by Narosa Publishing House, New Delhi, 1996.

\bibitem[TT90]{Thomason-Trobaugh}
R.~W. Thomason and T.~Trobaugh.
\newblock Higher algebraic {$K$}-theory of schemes and of derived categories.
\newblock In {\em The {G}rothendieck {F}estschrift, {V}ol. {III}}, volume~88 of
  {\em Progr. Math.}, pages 247--435. Birkh\"{a}user Boston, Boston, MA, 1990.
\newblock \doi{10.1007/978-0-8176-4576-2\_10}.

\bibitem[Was97]{washington95}
L.~C. Washington.
\newblock Galois cohomology.
\newblock In {\em Modular forms and {F}ermat's last theorem ({B}oston, {MA},
  1995)}, pages 101--120. Springer, New York, 1997.

\bibitem[Wel71]{Extensions}
C.~Wells.
\newblock Automorphisms of group extensions.
\newblock {\em Trans. Amer. Math. Soc.}, 155:189--194, 1971.
\newblock \doi{10.2307/1995472}.

\bibitem[Zhu21]{Zhu}
X.~Zhu.
\newblock Coherent sheaves on the stack of langlands parameters.
\newblock {\em arXiv}, 2021.
\newblock \doi{10.48550/ARXIV.2008.02998}.

\end{thebibliography}

\end{document}